\documentclass[11pt]{article}
\usepackage{lmodern}
\usepackage{mainsty}
\usepackage[utf8]{inputenc}
\usepackage{authblk}
\usepackage{upref}

\title{A new measure of dependence: Integrated $R^2$}
\author{Mona Azadkia, Pouya Roudaki\footnote{Department of Statistics, London School of Economics \& Political Science
}}

\date{\today}

\begin{document}


\maketitle
\begin{abstract}
    We introduce a novel measure of dependence that captures the extent to which a random variable $Y$ is determined by a random vector $\bbx$. The measure equals zero precisely when $Y$ and $\bbx$ are independent, and it attains one exactly when $Y$ is almost surely a measurable function of $\bbx$. We further extend this framework to define a measure of conditional dependence between $Y$ and $\bbx$ given $\bbz$. We propose a simple and interpretable estimator with computational complexity comparable to classical correlation coefficients, including those of Pearson, Spearman, and Chatterjee. Leveraging this dependence measure, we develop a tuning-free, model-agnostic variable selection procedure and establish its consistency under appropriate sparsity conditions. Extensive experiments on synthetic and real datasets highlight the strong empirical performance of our methodology and demonstrate substantial gains over existing approaches.
\end{abstract}


\section{Introduction}

Measuring the degree of dependence between two random variables is a longstanding problem in statistics, with numerous methods proposed over the years; for recent surveys, see~\cite{josse2016measuring, chatterjee2024survey}. Among the most widely used classical measures of statistical association are Pearson’s correlation coefficient, Spearman’s $\rho$, and Kendall’s $\tau$. These coefficients are highly effective for identifying monotonic relationships, and their asymptotic behaviour is well-established. However, a major limitation is that they perform poorly in detecting non-monotonic associations, even when there is no noise in the data.

To address this deficiency, there have been many proposals, such as the maximal correlation coefficient~\cite{breiman1985estimating, gebelein1941statistische, hirschfeld1935connection, renyi1959measures}, various methods based on joint cumulative distribution functions, and
ranks \cite{bergsma2014consistent, blum1961distribution, csorgo1985testing, deb2019multivariate, drton2018high, gamboa2018sensitivity, han2017distribution, hoeffding1948non, nandy2016large, puri1971nonparametric, romano1988bootstrap, rosenblatt1975quadratic, wang2017generalized, weihs2016efficient, weihs2018symmetric, yanagimoto1970measures, zhou2025association}, kernel-based methods~\cite{gretton2005measuring, gretton2008kernel, pfister2018kernel, sen2014testing, zhang2018large}  information theoretic coefficients~\cite{kraskov2004estimating, linfoot1957informational, reshef2011detecting}, coefficients based on copulas \cite{dette2013copula, lopez2013randomized, schweizer1981nonparametric, sklar1959fonctions, zhang2019bet, griessenberger2022multivariate}, and coefficients based on pairwise
distances~\cite{friedman1983graph, heller2013consistent, lyons2013distance, szekely2009brownian, szekely2007measuring, pan2020ball}. 

Some of these coefficients are widely used in practice; however, they suffer from two common limitations. First, most are primarily designed to test for independence rather than to quantify the strength of the dependence between variables. Second, many of these coefficients lack simple asymptotic distributions under the null hypothesis of independence, which hampers the efficient computation of p-values, since they rely on permutation-based tests.

Recently, Chatterjee introduced a new coefficient of correlation~\cite{chatterjee2021new} that is as simple to compute as classical coefficients, yet it serves as a consistent estimator of a dependence measure $\xi(X, Y)$ that equals 0 if and only if the variables are independent, and 1 if and only if one is a measurable function of the other. Moreover, like classical coefficients, it enjoys a simple asymptotic theory under the null hypothesis of independence. The limiting value $\xi(X, Y)$ was previously introduced in~\cite{dette2013copula} as the limit of a copula-based estimator in the case where $X$ and $Y$ are continuous. 

The simplicity, efficiency, and interpretability of Chatterjee's correlation have sparked significant interest, leading to a growing body of research on the behaviour of the coefficient and its extensions to more complex settings~\cite{azadkia2021simple, cao2020correlations, shi2022power, gamboa2022global, deb2020measuringassociationtopologicalspaces, huang2022kernel, auddy2024exact, lin2023boosting, lin2024failure, fuchs2024quantifying, griessenberger2022multivariate, zhang2023asymptotic, bickel2022measures, han2024azadkia, ansari2022simple, zhang2025relationships, shi2024azadkia, strothmann2024rearranged, dette2025simple, bucher2024lack, kroll2024asymptotic, tran2024rank,yang2025coverage, huang2025multivariate}.

\subsection{Key Contributions}
Building on this line of work, the first contribution of this paper is a new coefficient of dependence with the following properties 
\begin{enumerate}
    \item it has a simple expression,
    \item it is fully non-parametric,
    \item it requires no tuning parameters,
    \item it does not rely on estimating densities or characteristic functions,
    \item it can be computed from data in $O(n \log n)$ time, where $n$ denotes the sample size,
    \item asymptotically, it converges to a limit in $[0,1]$, where the limit equals $0$ if and only if the random variable $Y$ and random vector $\bbx$ are independent, and equals $1$ if and only if $Y$ is almost surely a measurable function of $\bbx$,
    \item the limiting quantity admits a natural interpretation as a generalisation of the familiar partial $R^2$ statistic for quantifying the dependence of $Y$ on $\bbx$,
    \item moreover, it extends to a coefficient of conditional dependence of $Y$ on $\bbx$ given $\bbz$, with the corresponding limit lying in $[0,1]$, equalling $0$ if and only if $Y$ is conditionally independent of $\bbx$ given $\bbz$, and equalling $1$ if and only if $Y$ is almost surely a measurable function of $\bbx$ given $\bbz$, and
    \item all of the above hold without any structural assumptions on the joint distribution of the random variables.
\end{enumerate}
The second contribution of this paper is a variable selection algorithm that demonstrates the substantial performance gains of our proposed dependence measure over~\cite{chatterjee2021new, azadkia2021simple}. While our approach is motivated by the FOCI framework introduced in~\cite{azadkia2021simple}, it significantly outperforms FOCI in both detection power and selection accuracy. Our algorithm preserves the desirable properties of being model-free, tuning-free, and provably consistent under sparsity assumptions, while delivering markedly improved empirical performance.

Finally, we highlight that this newly introduced coefficient of dependence can be interpreted as a novel \emph{discrepancy measure} on the space of permutations.

The paper is organised as follows. Section~\ref{sec:def} introduces our new measure of dependence, compares it with the Dette--Siburg--Stoimenov~\cite{dette2013copula} coefficient, interprets it as a generalisation of the classical $R^2$ measure, and extends it to a measure of conditional dependence. Section~\ref{sec:estimator} presents our general estimator, describes a simplified one-dimensional version, and establishes its rate of convergence. Section~\ref{sec:var} develops variable selection via the FORD procedure and examines the performance of the resulting algorithm. Section~\ref{sec:permutations} introduces a permutation metric derived from the dependence measure. Section~\ref{sec:example} reports simulation results and empirical illustrations. Finally, Section~\ref{sec:proof} contains the proofs of the main theoretical results.

\section{A New Measure of Dependence}\label{sec:def}

Let $Y$ be a random variable and $\bbx = (X_1, \ldots, X_p)$ a random vector defined on the same probability space. For clarity, when $p = 1$, we denote the vector $\bbx$ simply by $X$. Let $\mu$ be the probability law of $Y$. Let $S\subseteq\rr$ be the support of $\mu$. If $S$ attains a maximum $s_{\max}$ let $\tilde{S} = S\setminus \{s_{\max}\}$ otherwise let $\tilde{S} = S$. We define a probability measure $\tilde{\mu}$ on $S$ where for any measurable set $A\subseteq S$, $\tilde{\mu}(A) = \mu(A\cap\tilde{S})/\mu(\tilde{S})$. We propose the following quantity as a measure of dependence of $Y$ on $\bbx$:
\begin{eqnarray}\label{eq:def_nu}
    \nu(Y, \bbx) &:=& \int\frac{\var(\ee[\bone\{Y >  t\}\mid \bbx])}{\var(\bone\{Y > t\})}d\tilde{\mu}(t),
\end{eqnarray}
where $\bone\{Y > t\}$ is the indicator of the event $\{Y > t\}$. We note that a symmetrized form of $\nu$ was previously mentioned in \cite{kong2019composite} for the special case of one--dimensional $Y$ and $X$ (see equation~2.6 in \cite{kong2019composite}). However, that work did not provide theoretical development of the measure nor an accompanying estimation methodology. Our contribution is hence to formalize this measure, establish its properties, and develop estimators that enable its application in practice.

Observe that $\nu$ is a deterministic quantity determined entirely by the joint distribution of $(Y, \bbx)$. Because taking conditional expectations cannot increase variance, we have
\begin{eqnarray*}
    \var(\ee[\bone\{Y > t\}\mid\bbx]) \leq \var(\bone\{Y > t\}),
\end{eqnarray*}
which guarantees that $\nu\in[0, 1]$. If $Y$ is almost surely a measurable function of $\bbx$, then for almost every $t$ we have $\var(\ee[\bone\{Y > t\}\mid\bbx]) = \var(\bone\{Y > t\})$ and thus $\nu = 1$. On the other hand, if $Y$ is independent of $\bbx$, then for almost every $t$ we have $\var(\ee[\bone\{Y > t\}\mid\bbx]) = 0$ and thus $\nu = 0$. We will show that the converses of these statements also hold. The following theorem summarizes the key properties of $\nu$.

\begin{thm}\label{thm:NuProperties}
    For random variables $Y$ and $\bbx$ such that $Y$ is not almost surely a constant, $\nu(Y, \bbx)$ belongs to the interval $[0, 1]$, it is $0$ if and only if $Y$ and $\bbx$ are independent, and it is $1$ if and only if there exists a measurable function $f:\rr^p\rightarrow\rr$ such that  $Y = f(\bbx)$ almost surely.
\end{thm}

\begin{remark}\label{re:mutilde}
To explain the need for replacing $\mu$ with $\tilde{\mu}$, first note that no modification is required when $\mu$ is absolutely continuous; in that case $\tilde{\mu} = \mu$. The adjustment becomes necessary only when the support of $S$ has a maximum point $s_{\max}$ at which $\mu$ places positive mass. At such a point, the indicator $\bone\{Y > s_{\max}\}$ is identically zero, implying $\var(\bone\{Y > s_{\max}\}) = 0$. Since this indicator is a deterministic constant, it can be viewed either as independent of $\bbx$ or as trivially measurable with respect to $\bbx$. 

To ensure that $\nu$ reflects meaningful notions of dependence, it is therefore necessary to remove this degenerate threshold from consideration and focus on the portion of the support where variability---and hence dependence---is well defined. Because $\bone\{Y > s_{\max}\}$ exhibits no variation, it carries no information regarding the relationship between $Y$ and $\bbx$, and its influence should be excluded via the modified measure $\tilde{\mu}$.
\end{remark}

\subsection{Comparison to Dette-Siburg-Stoimenov}
\label{subsec:comparison_DSS}

For random variables $X$ and $Y$ with continuous marginal distributions, an early work in measuring dependence is~\cite{dette2013copula}, which defined the {\it Dette-Siburg-Stoimenov coefficient}, the association measure
\begin{align}\label{eq:hd1}
    \xi(X,Y) = 6\int_{[0,1]^2}\big ( \partial_1C(u,v)\big)^2 du dv - 2.
\end{align}
Here $C$ denotes the copula of the vector $(X,Y)$ and $\partial_1C$ its partial derivative with respect to the first coordinate. Later, in \cite{chatterjee2021new} the following measure was considered
\begin{align}\label{eq:hd2}
    \frac{\int\var(\ee[\bone(Y \geq t)\mid X]) d\mu(t)}{ \int\var(\bone(Y \geq t)) d\mu(t)},
\end{align}
with a corresponding estimator, meanwhile known as {\it Chatterjee's rank correlation}. It turns out that for continuous distributions the two measures \eqref{eq:hd1} and \eqref{eq:hd2} actually coincide. Later in \cite{azadkia2021simple} this measure was extended for multidimensional $\bbx$ as
\begin{eqnarray}\label{eq:T}
    T(Y, \bbx) = \frac{\int\var(\ee[\bone\{Y\geq t\}\mid \bbx])d\mu(t)}{\int\var(\bone\{Y\geq t\})d\mu(t)}.
\end{eqnarray}

To understand similarity and difference of $\nu$ and $T$, we consider the case where $Y$ and $\bbx$ have continuous density with no point mass. In this case we can write 
\begin{align*}
    \nu(Y, \bbx) = \int\frac{\var(\ee[\bone\{Y >  t\}\mid \bbx])}{\var(\bone\{Y > t\})}d\mu(t),\qquad T(Y, \bbx) = \int\frac{\var(\ee[\bone\{Y > t\}\mid \bbx])}{\int\var(\bone\{Y > t\})d\mu(t)} d\mu(t).
\end{align*}
We argue that $\nu$ is the more ``natural" dependence measure compared with $T$. Both quantities assess the strength of dependence of $Y$ on $\bbx$ by averaging the variability of the indicators $\bone\{Y > t\}$ conditional on $\bbx$ across all threshold values $t$. However, the two measures differ fundamentally in how this variability is normalized. The measure $\nu$ employs a local normalization: for each $t$, the quantity 
\begin{align*}
    \var(\ee[\bone\{Y > t\}\mid \bbx])
\end{align*}
is compared to the corresponding marginal variability $\var(\bone\{Y > t\})$, respecting the natural inequality
\begin{align*}
    \var(\ee[\bone\{Y > t\}\mid \bbx]) \leq \var(\bone\{Y > t\}).
\end{align*}
In contrast, $T$ uses a global normalization, dividing every term by the constant
\begin{align*}
    \int\var(\ee[\bone\{Y > t\}\mid \bbx])d\mu(t),
\end{align*}
regardless of the threshold under consideration.

This distinction has important consequences: under $T$, thresholds $t$ at which $\var(\bone\{Y > t\})$ is small receive the same normalization weight as thresholds where this variance is large. As a result, values of $t$ for which the indicator $\bone\{Y > t\}$ is nearly deterministic, e.g. tail values $t$, contribute little to the dependence measure, even if their conditional variability $\var(\ee[\bone\{Y > t\}\mid\bbx])$ indicates strong dependence on $\bbx$. 

By normalizing each term relative to its own marginal variability, $\nu$ appropriately highlights dependence even when $\var(\bone\{Y > t\})$ is close to zero. This local adaptivity makes $\nu$ conceptually more coherent and statistically more informative as a measure of dependence.

In Section~\ref{sec:example}, using both simulated and real data, we demonstrate that this distinction yields a substantial improvement in the performance of $\nu$ relative to $T$, particularly in the context of variable selection.

\subsection{Explained Variation Interpretation}\label{subsec:Rsquared}
The measure $\nu(Y, \bbx)$ admits a natural interpretation in terms of explained variation. Consider first the special case in which $Y$ is binary, taking values in $\{0,1\}$, so that $Y = \bone\{Y > 0\}$. By the law of total variance,
\[
    \nu(Y, \bbx)
    = \frac{\var(\ee[Y \mid \bbx])}{\var(Y)}
    = 1 - \frac{\ee[\var(Y \mid \bbx)]}{\var(Y)}.
\]
Thus, $\nu(Y, \bbx)$ coincides with the classical coefficient of determination $R^{2}_{Y,\bbx}$, representing the proportion of variance in $Y$ explained by $\bbx$.

For a general real-valued variable $Y$, define for each $t \in \rr$ the binary variable $Y_t := \bone\{Y > t\}$. Then
\[
    \nu(Y, \bbx)
    = \int \left( 1 - \frac{\ee[\var(Y_t \mid \bbx)]}{\var(Y_t)} \right) d\tilde{\mu}(t)
    = \int R^{2}_{Y_t,\bbx}\, d\tilde{\mu}(t).
\]
Hence, $\nu(Y, \bbx)$ can be viewed as an average, over all thresholds $t$, of the explained-variance coefficients $R^{2}_{Y_t,\bbx}$ with respect to the measure $\tilde{\mu}$. Since $Y$ can be represented as an integral (or linear combination) of the indicators $\{Y_t\}_{t \in \rr}$, the quantity $\nu(Y, \bbx)$ serves as a measure of the overall proportion of variation in $Y$ that is explainable by $\bbx$.
\subsection{Conditional Dependence}\label{sec:conditional}

From definition of $\nu(Y, \bbx)$ in~\eqref{eq:def_nu}, we extend $\nu$ to a measure that quantifies the conditional dependence of $Y$ on $\bbx$ given $\bbz$. Given $(Y, \bbx, \bbz)$, we define 
\begin{align}\label{eq:def_cond_nu}
    \nu(Y, \bbx\mid\bbz) := \frac{\nu(Y, (\bbx, \bbz)) - \nu(Y, \bbz)}{1 - \nu(Y, \bbz)}.
\end{align}
The following theorem establishes that $\nu(Y, \bbx\mid\bbz)$ is well-defined and satisfies the desired properties.
\begin{thm}\label{thm:conditional}
    Suppose that $Y$ is not almost surely equal to a measurable function of $\bbz$. Then $\nu(Y, \bbx\mid\bbz)$ is well-defined and belongs to $[0, 1]$. Moreover, $\nu=0$ if and only if $Y$ and $\bbx$ are conditionally independent given $\bbz$, and $\nu=1$ if and only if $Y$ is almost surely equal to a measurable function of $\bbx$ given $\bbz$. 
\end{thm}
We have shown in Section \ref{sec:example} how the estimated conditional dependence captures complex relationships.
\section{Estimator}\label{sec:estimator}
Having defined $\nu$, we now address the question of whether it can be efficiently estimated from data. We introduce the estimator $\nu_n(Y, \bbx)$ for $\nu(Y, \bbx)$ and study its statistical properties. Suppose we observe $(Y_1, \bbx_1), \ldots, (Y_n, \bbx_n)$, where $n \ge 3$ and the pairs are i.i.d.\ copies of $(Y, \bbx)$. For each $i$, let $R_i$ denote the rank of $Y_i$, defined by
\[
R_i = \sum_{j=1}^n \bone\{Y_j \le Y_i\}.
\]
For any distinct indices $i,j \in \{1,\ldots,n\}$, let $N^{-j}(i)$ be the index of the nearest neighbour of $\bbx_i$ (under the Euclidean metric on $\mathbb{R}^p$) among the points $\{\bbx_k : k \neq i, j\}$, with ties broken uniformly at random. Define
\[
    \mathcal{R}_i^j := \big[\min\{R_i, R_{N^{-j}(i)}\},\, \max\{R_i, R_{N^{-j}(i)}\}\big].
\]
Let
\[
    n_{\max} := \big|\{i : Y_i = \max_{j \in [n]} Y_j\}\big|
    \quad\text{and}\quad
    c_{\min} :=
    \begin{cases}
        1, & \text{if } \big|\{i : Y_i = \min_{j \in [n]} Y_j\}\big| = 1, \\
        0, & \text{otherwise}.
    \end{cases}
\]
Set $n_0 = n_{\max} + c_{\min}$.  
In the absence of ties among the $Y_j$'s, we have $n_{\max} = c_{\min} = 1$.  
When $n_0 < n$, we define the estimator $\nu_n$ by
\begin{equation}\label{eq:nun2}
    \nu_n(Y, \bbx)
    := 1
    - \frac{1}{2}\left(\frac{n - 1}{n - n_0}\right)
      \sum_{\substack{j :\, R_j \notin \{1, n\}}}
      \sum_{i \neq j}
      \frac{\bone\{R_j \in \mathcal{R}_i^j\}}{(R_j - 1)(n - R_j)}.
\end{equation}
If $n = n_0$, the data provide no information about variability in $Y$, and in this case we cannot construct an estimator for $\nu$.

The following theorem establishes that $\nu_n$ is a consistent estimator of $\nu$.

\begin{thm}\label{thm:consistency}
    If $Y$ is not almost surely constant, then $\nu_n$ converges almost surely to $\nu$ as $n \to \infty$.
\end{thm}

We leverage $\nu_n$ in~\eqref{eq:nun2} to estimate the conditional quantity $\nu(Y, \bbx \mid \bbz)$ through a simple plug-in approach. Given a sample $(Y_1, \bbx_1, \bbz_1), \ldots, (Y_n, \bbx_n, \bbz_n)$, we estimate $\nu(Y, \bbx \mid \bbz)$ by
\begin{align*}
    \nu_n(Y, \bbx\mid\bbz) = \frac{\nu_n(Y, (\bbx, \bbz)) - \nu_n(Y, \bbz)}{1 - \nu_n(Y, \bbz)}.
\end{align*}

\begin{cor}\label{cor:conditional}
Suppose that $Y$ is not almost surely equal to a measurable function of $\bbz$, then as $n \rightarrow \infty, \nu_n(Y, \bbx\mid\bbz) \rightarrow \nu(Y, \bbx\mid\bbz)$ almost surely.
\end{cor}

\begin{remark}
\noindent (1) When $p$ is fixed, the statistic $\nu_n$ can be computed in $O(n \log n)$ time. Nearest neighbours may be identified in $O(n \log n)$ time \cite{friedman1977algorithm}, and the quantities $\bone\{Y_j \in \mathcal{R}_i^j\}$ together with the ranks $R_j$ can likewise be computed in $O(n \log n)$ time \cite{knuth1997art}. At first glance, \eqref{eq:simpleEst} appears to require a double loop over all $j$ and all intervals $\mathcal{R}_i^j$, suggesting a computational cost of order $O(n^2)$. However, the essential task reduces to counting how many integer intervals contain a given integer, which can be carried out in $O(n)$ time using a \emph{difference array method}.

\noindent (2) No assumptions are required on the joint distribution of $(Y, \bbx)$ beyond the non-degeneracy condition that $Y$ is not almost surely constant. This condition is essential: if $Y$ were almost surely constant, it would simultaneously be independent of $\bbx$ and a measurable function of $\bbx$, making it impossible for any dependence measure between $Y$ and $\bbx$ to be meaningfully defined.

\noindent (3) Although Theorem~\ref{thm:NuProperties} ensures that $\nu$ lies in the interval $[0,1]$ and Theorem~\ref{thm:consistency} establishes the almost sure convergence of $\nu_n$ to $\nu$, the finite-sample values of $\nu_n$ need not themselves be constrained to the interval $[0,1]$.

\noindent (4) The coefficient $\nu_n(Y, \bbx)$ is invariant under strictly increasing transformations of $Y$, as its construction depends solely on the ranks of the $Y_i$.

\noindent (5) We have developed an \texttt{R} package, \texttt{FORD}~\cite{azadkiafordpackage}, available on CRAN,\footnote{\url{https://cran.r-project.org/web/packages/FORD/index.html}} which provides functions for computing $\nu_n$ and for implementing the FORD variable selection procedure described in Section~\ref{sec:var}.
 
\noindent (6) Besides variable selection, another natural area of applications of our coefficient is graphical models; similar ideas as in \cite{azadkia2021fast, chatterjee2022estimating} are being investigated.

\noindent (7) If the $\bbx_i$'s contain ties, then $\nu_n(Y, \bbx)$ becomes a randomized estimate of $\nu(Y, \bbx)$ due to the randomness introduced by tie-breaking. While this effect diminishes as $n$ grows large, a more robust estimate can be obtained by averaging $\nu_n$ over all possible tie-breaking configurations.

\noindent (8) Note that $\nu_n$ is based on nearest neighbour graphs and, as a result, generally lacks scale invariance; that is, changes in the scale of certain covariates can significantly alter the graph structure. To address this issue, a rank-based variant, similar to that proposed in~\cite{tran2024rank}, can be considered.

\noindent (9) Note that $\nu(Y, X)$ is not symmetric in $Y$ and $X$. This asymmetry is intentional, as our objective is often to assess whether $Y$ depends on $X$, rather than merely whether one variable is a function of the other. If a symmetric measure of dependence is desired, it can be obtained by taking $\max\{\nu(Y, X),\, \nu(X, Y)\}$.
\end{remark}

\subsection{A simpler estimator for one-dimensional case}
\label{subsubsec:simpler_estim}

Consider the case where $p = 1$, so that $\bbx$ is a univariate random variable. To emphasize this, we write $\bbx$ as $X$ throughout this section. Following \cite{chatterjee2021new}, we introduce a related estimator which takes advantage of the canonical ordering on $\rr$. Let $(Y_1, X_1), \ldots, (Y_n, X_n)$ be i.i.d.\ samples from the distribution of $(Y, X)$, with $n \ge 2$. Rearrange the data as $(Y_{(1)}, X_{(1)}), \ldots, (Y_{(n)}, X_{(n)})$ such that
\[
    X_{(1)} \le \cdots \le X_{(n)}.
\]
If the $X_i$'s are distinct, this ordering is unique; if ties occur, we select an ordering uniformly at random among all permutations that preserve monotonicity. For each $i$, let $r_i = R_{(i)}$ denote the rank of $Y_{(i)}$. Define the interval
\[
    \mathcal{K}_i := \big[\min\{r_i, r_{i+1}\},\, \max\{r_i, r_{i+1}\}\big].
\]
Define
\begin{equation}\label{eq:simpleEst}
    \nu_{n}^{\text{1-dim}}(Y, X)
    := 1 - \frac{1}{2}\left(\frac{n - 1}{n - n_0}\right)
    \sum_{\substack{j \\ r_j \notin \{1, n\}}}
    \sum_{\substack{i \neq j,\, j-1,\, n}}
    \frac{\bone\{r_j \in \mathcal{K}_i\}}{(r_j - 1)(n - r_j)}.
\end{equation}

The following theorem establishes that $\nu_n$ is a consistent estimator of $\nu_{n}^{\text{1-dim}}$. 
\begin{thm}\label{thm:simpleConsistency}
    Let $X$ and $Y$ be random variables. If $Y$ is not almost surely constant, then $\nu_{n}^{\text{1-dim}}$ converges almost surely to $\nu$ as $n \to \infty$.
\end{thm}

Having established consistency of $\nu_n$ and $\nu_{n}^{\text{1-dim}}$, we next examine their behaviour under the null hypothesis of independence. The following propositions derive the expectation and asymptotic variance of these estimators under independence.

\begin{prop}\label{prop:momentsNun}
    Suppose that $\bbx$ and $Y$ are independent and both have continuous distributions, then
    \begin{align*}
        \ee[\nu_n(Y, \bbx)] = \frac{-1}{n-2}, \qquad \var(\nu_n(Y, \bbx)) = O(\frac{1}{n}).
    \end{align*}
\end{prop}

\begin{prop}\label{thm:nullSimpleAsymp} Suppose that $X$ and $Y$ are independent and both have continuous distributions, then
\begin{eqnarray*}
\ee[\nu_{n}^{\text{1-dim}}(Y, X)] = 2/n, \qquad \lim_{n\rightarrow\infty}n\var(\nu_{n}^{\text{1-dim}}(Y, X)) = \pi^2/3 - 3.
\end{eqnarray*}
\end{prop}

We conjecture that, under independence, both $\sqrt{n}\,\nu_n$ and $\sqrt{n}\,\nu_{n}^{\text{1-dim}}$ satisfy a central limit theorem. At present, however, we do not know how to establish these results. A key requirement is the variance scaling $\var(\nu_n(Y,\mathbf{X}))=\Theta(1/n)$, which is supported by simulations, although deriving it analytically appears to be cumbersome.

Proving a CLT in this setting is technically challenging, even for $\nu_{n}^{\text{1-dim}}$, which in principle should be more tractable because the problem reduces to statistics on random permutations (see Section~\ref{sec:permutations}). Existing techniques for permutation statistics, such as those in ~\cite{hoeffding1951combinatorial, chatterjee2017central}, or for stabilizing functionals~\cite{penrose2001central}, do not appear applicable, since $\nu_n$ and $\nu_{n}^{\text{1-dim}}$ depend not only on the relative ordering of ranks but also on their positional values, which means the effect of replacing a sample point by an independent copy does not remain local. We have also explored a martingale CLT approach, but the required second-moment calculations do not seem to yield tractable expressions.

Although we are unable to prove a central limit theorem for $\nu_n$ and $\nu_{n}^{\text{1-dim}}$ under independence, we can nevertheless describe their finite-sample behaviour through a non-asymptotic concentration bound. Remarkably, this bound holds for arbitrary $(Y, \bbx)$, not only under independence. The corresponding result is stated below.
\begin{thm}\label{thm:concentration}
    There are constants $C_1$ and $C_2$ such that 
    \begin{eqnarray*}
        \pp(\abs{\nu_n - \ee[\nu_n]}\geq \varepsilon) \leq C_1e^{-C_2n\varepsilon^2/\log^2 n}, \qquad\pp(\abs{\nu_{n}^{\text{1-dim}} - \ee[\nu_{n}^{\text{1-dim}}]}\geq \varepsilon) \leq C_1e^{-C_2n\varepsilon^2/\log^2 n}.
    \end{eqnarray*}
\end{thm}

\subsection{Rate of Convergence}\label{subsec:rate}
To obtain a convergence rate for $\nu_n$ to $\nu$, we must impose certain assumptions on the distribution of $(Y, \bbx)$. Without such assumptions, the convergence can, in principle, be arbitrarily slow. The primary challenge lies in controlling the sensitivity of the conditional distribution of $Y$ given $\bbx$ with respect to variations in $\bbx$, which is addressed by the first assumption below. The second assumption is introduced for technical convenience.

    \begin{itemize}[leftmargin=4em]
        \item[(A1)] There are nonnegative real numbers $\beta$ and $C$ such that for any $t \in \rr$, $\bx, \bx^{\prime} \in \rr^p$,
        \begin{align*}
            \lefteqn{\abs{\pp\left(Y \leq t \mid \bbx=\bx\right)-\pp\left(Y \leq t \mid \bbx=\bx^{\prime}\right)}\leq }\\
            && \qquad\qquad C\left(1+\|\bx\|^\beta+\|\bx^{\prime}\|^\beta\right)\|\bx-\bx^{\prime}\|\min\{F(t), 1 - F(t)\}.
        \end{align*}

        \item[(A2)] There exists a constant $K > 0$ such that $\pp(\|\bbx\| \leq K) = 1$; that is, $\bbx$ has bounded support.
    \end{itemize}
Assumption (A1) implies that the conditional distribution function
\[
t \mapsto \pp\bigl(Y \le t \,\big\vert\, \bbx = \bx\bigr)
\]
is locally Lipschitz in $\bx$, with a Lipschitz constant that may grow at most polynomially in $\|\bx\|$ and $\|\bx'\|$. Because the bound in (A1) is multiplied by $\min\{F(t), 1 - F(t)\}$, the Lipschitz requirement becomes stricter for tail values of $Y$.

Under Assumptions (A1) and (A2), the following theorem shows that $\nu_n$ converges to $\nu$ at the rate $n^{-1/(p\vee 2)}$, up to a logarithmic factor.
\begin{thm}\label{thm:rate}
Suppose that $p \geq 1$, and assumptions \textnormal{(A1)} and \textnormal{(A2)} holds for some $C$, $\beta$, and $K$. Then as $n\rightarrow\infty$
    \begin{eqnarray*}
        \nu_n - \nu = O_\pp\left(\frac{(\log n)^{1 + \bone\{p = 1\}}}{n^{1/(p\vee 2)}}\right).
    \end{eqnarray*}
\end{thm}
Assumption~\textnormal{(A2)} simply requires that $\bbx$ have bounded support. In contrast, it may be less transparent when Assumption~\textnormal{(A1)} holds. Consider, for example, a generating model of the form
\[
    Y = m(\bbx) + s(\bbx)\varepsilon,
\]
where $m(\cdot)$ is a Lipschitz function, $s(\bx) \ge c > 0$ for all $\bx$ and some constant $c$, and $\varepsilon$ is independent of $\bbx$ with density $f_\varepsilon$ and distribution function $F_\varepsilon$. Suppose that
\begin{align}\label{eq:epscond}
    \frac{f_\varepsilon(t)}{\min\{F_\varepsilon(t),\, 1 - F_\varepsilon(t)\}}
\end{align}
is bounded on $\rr$. For such distributions, Assumption~\textnormal{(A1)} is satisfied. This setting includes many commonly used models, such as linear regression, additive models, and heteroskedastic regression.

Note that the ratio in~\eqref{eq:epscond} is bounded if and only if both $f_\varepsilon(t)/(1 - F_\varepsilon(t))$ and $f_\varepsilon(t)/F_\varepsilon(t)$ are bounded. These quantities are known respectively as the \emph{hazard ratio} and the \emph{reverse hazard ratio}. For example, both ratios are bounded for the Laplace, chi-squared, and Student's $t$ distributions (with degrees of freedom greater than one).

More broadly, the next result demonstrates that condition \textnormal{(A1)} holds for many densities with suitable regularity and decay.

\begin{prop}\label{prop:reg}
Assume $Y$ has a strictly positive, continuously differentiable density $f$, and for each $\bx$, the conditional density $f_{Y \mid \bbx = \bx}$ exists and is continuously differentiable in $\bx$. Moreover, there exist $\beta \geq 0$ and $K_1<\infty$ such that
\begin{align}\label{eq:cond2}
\left\|\nabla_\bx f_{Y \mid \bbx = \bx}(t)\right\| \leq K_1\left(1+\|\bx\|^\beta\right) f(t) \quad \text { for all } x \in \rr^p, t \in \rr .
\end{align}
Then there exists a constant $C<\infty$ (depending only on $K_0, K_1$) and the same $\beta$ such that for all $t \in \rr$ and all $\bx, \bx^{\prime} \in \rr^p$ that satisfies (A1).
\end{prop}

\section{Variable Selection: Feature Ordering by Dependence}\label{sec:var}
Many commonly used variable selection methods in the statistics literature are based on linear or additive models. This includes several classical approaches~\cite{breiman1995better, chen1994basis, efron2004least, friedman1991multivariate, george1993variable, hastie2009elements, miller2002subset, tibshirani1996regression} as well as modern
ones~\cite{candes2007dantzig, fan2001variable, ravikumar2009sparse, yuan2006model, zou2006adaptive, zou2005regularization}, which are both powerful and widely adopted in practice. However, these methods can struggle when interaction effects or nonlinear relationships are present. 

Such problems can sometimes be overcome by model-free methods ~\cite{amit1997shape, battiti1994using, breiman1996bagging, breiman2001random, breiman2017classification, candes2018panning, freund1996experiments, hastie2009elements, ho1998random, vergara2014review}. These, too, are powerful and widely used techniques, and they perform better than model-based methods if interactions are present. On the flip side, their theoretical foundations are usually weaker than those of model-based methods. 

A related yet distinct direction focuses on handling ultra-high-dimensional settings through screening procedures, most notably the Sure Independence Screening (SIS) framework and its variants~\cite{Fan2008SureIndependenceScreening, Li2012DistanceCorrelation, Balasubramanian2013HSICSIS, Zhou2018BKR, Xu2020PCor, Pan2019SureIndependenceScreening}. These methods evaluate each covariate’s marginal relationship with the response and use this information to preselect a manageable subset of variables before applying more sophisticated model-based or model-free selection techniques. Importantly, the objective of SIS is not to isolate a minimal sufficient set of predictors, but rather to retain—with high probability—the truly relevant variables within a reduced but still relatively large pool. Although SIS procedures offer strong scalability and well-established screening guarantees in ultra-high-dimensional regimes, their reliance on marginal associations can limit their effectiveness when relevant and irrelevant variables are correlated or when the underlying signal arises predominantly from joint rather than marginal effects.

In this section, we propose a new variable selection algorithm for multivariate regression using a forward stepwise algorithm based on $\nu$. Our algorithm in nature follows precisely the idea of FOCI~\cite{azadkia2021simple} for multivariate regression. We call our method \textit{Feature Ordering by Integrated $R^2$ Dependence} (FORD).

The method is as follows. Let $Y$ be the response variable and let $\bbx = (X_1, \ldots, X_p)$ be the set of predictors. The data consists of $n$ i.i.d. copies of $(Y, \bbx)$. First, choose $j_1$ to be the index $j$ that maximizes $\nu_n(Y, X_j)$. If $\nu_n(Y, X_{j_1}) \leq 0$, declare $\hat{V}$ to be empty set and terminate the process. Otherwise, having obtained $j_1, \ldots, j_k$, we select $j_{k+1}$ as the index $j \notin \{j_1, \ldots, j_k\}$ that maximizes $\nu_n\!\left(Y, (X_{j_1}, \ldots, X_{j_k}, X_j)\right)$ (equivalently, the index that maximizes $\nu_n\!\left(Y, X_j \mid X_{j_1}, \ldots, X_{j_k}\right)$).
Continue like this until arriving at the first $k$ such that 
\begin{align}\label{eq:stoppingRule}
    \nu_n(Y, (X_{j_1}, \ldots, X_{j_k}, X_{j_{k+1}})) \leq \nu_n(Y, (X_{j_1}, \ldots, X_{j_k})),
\end{align}
which is equivalent to $\nu_n(Y, X_{j_{k+1}}\mid (X_{j_1}, \ldots, X_{j_k})) \leq 0$, and then declare the chosen subset to be $\hat{V} := \left\{j_1, \ldots, j_k\right\}$. If there is no such $k$, define $\hat{V}$ as the whole set of variables.

Note that the algorithm closely follows the setup of FOCI~\cite{azadkia2021simple}  by replacing $T_n$ in FOCI by $\nu_n$. Several extensions of FOCI have since been proposed. For example, KFOCI~\cite{huang2022kernel} incorporates kernel-based methods to estimate conditional dependence, ~\cite{pavasovic2025differentiable} introduce a parametric, differentiable approximation of the same conditional dependence measure, which is used to evaluate feature importance in neural networks. Some other model-agnostic variable important scores are~\cite{zhang2020floodgate, wang2023total, huang2025factor}.

\begin{remark}\label{remark:stoppingRule}
    The stopping criterion in~\eqref{eq:stoppingRule} may at first seem counterintuitive. In principle, for any random variable $Y$ and random vectors $\bbx$ and $\mathbf{Z}$, we have
    \begin{equation}\label{eq:monotone}
        \nu(Y, \bbx) \leq \nu\big(Y, (\bbx, \mathbf{Z})\big).
    \end{equation}
    One might therefore anticipate an inequality in the opposite direction. The key point, however, is that~\eqref{eq:stoppingRule} is expressed in terms of the sample-based estimator $\nu_n$, not the population-level quantity $\nu$. Sampling variability and estimation error can cause $\nu_n$ to deviate from its population analogue, and the monotonicity property need not hold for the estimator. Moreover, when $Y$ is conditionally independent of $\bbz$ given $\bbx$, we have $\nu(Y, \bbx\mid\bbz) = 0$ which is equivalent to $\nu(Y, \bbx) = \nu(Y, (\bbx, \bbz))$, in which case adding $\bbz$ should not increase the measure. 
    Criterion~\eqref{eq:stoppingRule} is designed to detect precisely this situation by halting when the inclusion of additional variables fails to yield an increase in $\nu_n$, indicating that the currently selected variables already capture the relevant dependence structure.
\end{remark}

\subsection{Efficacy of FORD}
\label{subsec:efficacy}
Let $(Y, \bbx)$ be as defined in the previous section. For any subset of indices $V \subseteq \{1, \ldots, p\}$, define $\bbx_V := (X_j)_{j \in V}$ and let $V^c := \{1, \ldots, p\} \setminus V$. A subset $V$ is said to be \emph{sufficient}~\cite{vergara2014review} if $Y$ and $\bbx_{V^c}$ are conditionally independent given $\bbx_V$. This definition allows for the possibility that $V$ is the empty set, in which case it simply implies that $Y$ and $\bbx$ are independent.

We will prove later that $\nu(Y, \bbx_{V^{\prime}}) \geq \nu(Y, \bbx_V)$ whenever $V^{\prime} \supseteq V$, with equality if and only $Y$ and $\bbx_{V^{\prime} \backslash V}$ are conditionally independent given $\bbx_V$. Thus if $V^{\prime} \supseteq V$, the difference $\nu\left(Y, \bbx_{V^{\prime}}\right) - \nu(Y, \bbx_V)$ is a measure of how much extra predictive power is added by appending $\bbx_{V^{\prime} \backslash V}$ to the set of predictors $\bbx_V$.

Let $\delta$ be the largest constant such that for every insufficient subset $V \subseteq \{1, \ldots, p\}$, there exists some index $j \notin V$ satisfying
\begin{align}\label{eq:delta}
    \nu(Y, \bbx_{V \cup \{j\}}) \geq \nu(Y, \bbx_V) + \delta.
\end{align}
In other words, if $V$ is insufficient, then appending at least one variable $\bbx_j$ with $j \notin V$ to the set $\bbx_V$ increases the dependence with $Y$ by at least $\delta$. The main result of this section, stated below, shows that if $\delta$ is bounded away from zero, then under certain regularity conditions on the distribution of $(Y, \bbx)$, the subset selected by FORD is sufficient with high probability.

It is worth noting that the assumption that $\delta$ is not too small implicitly encodes a sparsity condition: by definition, $\delta$ guarantees the existence of a sufficient subset of size at most $1/\delta$.

To demonstrate the efficacy of our method, we need the following two technical assumptions on the joint distribution of $(Y, \bbx)$. They are generalisations of the assumptions (A1) and (A2) from Subsection~\ref{subsec:rate}.
\begin{itemize}[leftmargin=4em]
\item[(A1$^\prime$)] There are nonnegative real numbers $\beta$ and $C$ such that for any set $V \subseteq\{1, \ldots, p\}$ of size $\leq 1 / \delta+2$, any $\bx, \bx^{\prime} \in \rr^V$ and any $t \in \rr$,
\begin{align*}
    \lefteqn{\abs{\pp\left(Y \leq t \mid \bbx_V=\bx\right)-\pp\left(Y \leq t \mid \bbx_V=\bx^{\prime}\right)}\leq }\\
    &\qquad\qquad\qquad C\left(1+\|\bx\|^\beta+\|\bx^{\prime}\|^\beta\right)\|\bx-\bx^{\prime}\|\min\{F(t), 1 - F(t)\}.
\end{align*}

\item[(A2$^\prime$)] There exists a constant $K > 0$ such that for any subset $V \subseteq \{1, \ldots, p\}$ with cardinality at most $1/\delta + 2$, we have $\pp\left(\|\bbx_V\| \le K\right) = 1$; that is, $\bbx_V$ has bounded support.

\end{itemize}
\begin{thm}\label{thm:FOCI}
    Suppose that $\delta > 0$, and that the assumptions \textnormal{(A1$^\prime$)} and \textnormal{(A2$^\prime$)} hold. Let $\hat{V}$ be the subset selected by FORD with a sample of size $n$. There are positive real numbers $L_1, L_2$ and $L_3$ depending only on $C, \beta, K$, and $\delta$ such that $\pp(\hat{V}$ is sufficient $) \geq 1-L_1 p^{L_2} e^{-L_3 n}$.
\end{thm}
Theorem~\ref{thm:FOCI} demonstrates that FORD, like FOCI~\cite{azadkia2021simple}, differs from many traditional variable selection methods in that it is not only model-free but also incorporates a principled stopping rule and provides a theoretical guarantee that the selected subset is sufficient with high probability. A closely related approach in the literature is the mutual information-based method proposed by~\cite{battiti1994using}; however, in contrast to FOCI and FORD, it does not include a well-defined stopping criterion.

To clarify the role of quantity $\delta$ defined in~\eqref{eq:delta}, let us consider the classic example of linear regression with normally distributed predictor variables. Suppose that $\bbx$ is a normal random vector with zero mean and some arbitrary covariance matrix, and that
\[
    Y=\beta  \bbx + \varepsilon,
\]
where $\beta \in \rr^p$ is a vector of coefficients and $\varepsilon \sim N\left(0, \sigma^2\right)$ is independent of $\bbx$, with nonzero $\sigma$. Then $Y$ is also a normal random variable with mean zero. Let $\tau^2 := \var(Y)$. Let $\delta$ be the quantity defined in \eqref{eq:delta}, for this $Y$ and $\bbx$.

For any nonempty $S \subset\{1, \ldots, p\}$ and any $j \in\{1, \ldots, p\} \backslash S$, let $\rho(S, j)$ be the partial $R^2$ of $Y$ and $X_j$ given $\bbx_S$. Let $\rho(\emptyset, j)$ be the usual $R^2$, i.e. squared correlation between $Y$ and $X_j$.

Note that using the normal structure, $S$ is a sufficient set of predictors, if and only if $\rho(S, j) = 0$ for any $j \notin S$. So if $S$ is insufficient, then there is at least one $j \notin S$ such that $\rho(S, j)>0$.

Let $\delta^{\prime}$ be the largest number such that for any insufficient set $S$, there is some $j \notin S$ such that $\rho(S, j) \geq \delta^{\prime}$. The following result shows that $\delta^{\prime}$ is comparable to $\delta$, up to constant multiples depending only on $\sigma$ and $\tau$.

\begin{thm}\label{thm:deltaAndPrime}
    Let all the notations be as above. There exist positive constants $c$ and $C$, depending only on $\tau$ and $\sigma$ such that 
    \begin{align*}
         c\delta^\prime \leq \delta \leq C\delta^\prime.
    \end{align*}
\end{thm}
In particular, in the Gaussian setting, the quantity $\delta$ is equivalent (up to constants depending only on $\tau$ and $\sigma$) to the analogous measure $\delta^\prime$ obtained by replacing our dependence metric with the usual partial $R^2$.

\section{A Metric on Permutations}
\label{sec:permutations}

Consider the setting where both $X$ and $Y$ are one-dimensional random variables. In this case, any measure of dependence between $X$ and $Y$ may be understood as inducing a metric on the space of permutations of the sample indices. This viewpoint is natural, as dependence measures typically quantify the extent to which the joint ordering of $(X,Y)$ departs from the ordering expected under independence, and such departures can be encoded as distances between permutations. Motivated by this perspective, we show that $\nu_{n}^{\text{1-dim}}$ corresponds to a permutation-based discrepancy measure, distinct from and complementary to classical permutation metrics.

Without loss of generality, assume $\{X_i\} = \{Y_i\} = [n]$. Let $\pi$ and $\sigma$ be the permutations of $[n]$ such that
\[
    X_{\pi(1)} < \cdots < X_{\pi(n)}
    \qquad\text{and}\qquad
    Y_{\sigma(1)} < \cdots < Y_{\sigma(n)}.
\]
Let $I$ denote the identity permutation. In this representation,
\[
    r_i = \mathrm{rank}(Y_{\pi(i)}) = \sigma^{-1}\pi(i),
\]
and hence
\[
    \nu_{n}^{\text{1-dim}}(Y, X)
    = 1 - \left(\frac{n - 1}{n - 2}\right) d_\nu(\sigma, \pi),
\]
where
\begin{equation}\label{eq:dnu}
    d_\nu(\sigma, \pi)
    := \frac{1}{2} \sum_{\ell = 2}^{n-1} \sum_{i = 1}^{n-1}
    \frac{\bone\{\ell \text{ lies between }
        \sigma^{-1}\pi(i) \text{ and }
        \sigma^{-1}\pi(i+1)\}}
    {(\ell - 1)(n - \ell)}.
\end{equation}

The function $d_\nu$ satisfies the following properties:
\begin{enumerate}[label=\arabic*.]
    \item \emph{Left-invariance}: $d_\nu(\sigma, \pi) = d_\nu(\tau\sigma, \tau\pi)$ for any permutation $\tau$;
    \item $d_\nu(\sigma, \pi) = 0$ if and only if $\sigma = \pi$;
    \item In general, $d_\nu(\sigma, \pi)$ is not necessarily equal $d_\nu(\pi, \sigma)$, though a symmetric version may be obtained by
    \[
        d_\nu^{\mathrm{sym}}(\sigma, \pi)
        := \frac{1}{2}\big(d_\nu(\sigma, \pi) + d_\nu(\pi, \sigma)\big).
    \]
\end{enumerate}
Thus $d_\nu(\sigma, \pi)$ may be viewed as a valid discrepancy measure between permutations.

Numerous metrics have been proposed in the literature to quantify distances between permutations, including:
\begin{itemize}
    \item Spearman's footrule: $d_s(\sigma, \pi) = \sum_{i=1}^n |\sigma(i) - \pi(i)|$;
    \item Spearman’s rho: $d_\rho^2(\sigma, \pi) = \sum_{i=1}^n (\sigma(i) - \pi(i))^2$;
    \item Kendall’s tau: $d_\tau(\sigma, \pi) =$ the minimum number of adjacent transpositions that transform $\pi$ into $\sigma$;
    \item Cayley distance: $d_C(\sigma, \pi) =$ the minimum number of transpositions needed to transform $\pi$ into $\sigma$;
    \item Hamming distance: $d_H(\sigma, \pi) = |\{i : \sigma(i) \neq \pi(i)\}|$;
    \item Ulam distance: $d_U(\sigma, \pi) = n - \text{length of the longest increasing subsequence}$.
\end{itemize}

The discrepancy $d_\nu(\sigma, \pi)$ is most closely related to Spearman's footrule and to the oscillation measure $\mathrm{Osc}(\sigma^{-1}\pi)$ \cite{levcopoulos1989heapsort}, in that it quantifies how much $\sigma^{-1}\pi$ oscillates as $i$ moves from $i$ to $i+1$. Unlike these classical metrics, however, $d_\nu$ incorporates position-dependent weights: the contributing oscillations are scaled by $1/[(\ell - 1)(n - \ell)]$, assigning greater emphasis to oscillations occurring near the extremal ranks. This weighting structure, together with left-invariance, distinguishes $d_\nu$ from metrics such as Spearman's footrule, which are right-invariant, and highlights the distinct way in which $d_\nu$ assesses discrepancies between permutations.

Based on these differences, $d_\nu$ may be particularly useful in rank estimation settings where positional discrepancies carry unequal importance. In such contexts—such as search result evaluation or recommendation systems, where inaccuracies near the top or bottom of the ranking are substantially more consequential—$d_\nu$ offers a more sensitive means of quantifying deviations than classical metrics such as Spearman’s footrule or Kendall’s tau. A systematic investigation of these applications is left for future work.

\section{Examples}\label{sec:example}
This section presents applications of our methods to simulated and real datasets. In all cases, covariates were standardised before analysis.

\subsection{Simulation Examples}
\label{subsec:simulation_example}
\begin{ex}\label{exGen}(general behaviour)
Figure~\ref{fig:exGen} illustrates the general performance of $\nu_n$ as a measure of association. The figure consists of three rows, each beginning with a scatterplot in which $Y$ is a noiseless function of $X$, where $X$ is drawn from the uniform distribution on $[-1,1]$. Moving to the right within each row, increasing levels of noise are added to $Y$. The sample size is fixed at $n = 100$ across all cases, demonstrating that $\nu_n$ performs well even with relatively small samples.
\begin{figure}[H]
    \centering
    \begin{subfigure}[b]{0.3\textwidth}
        \includegraphics[width=\textwidth]{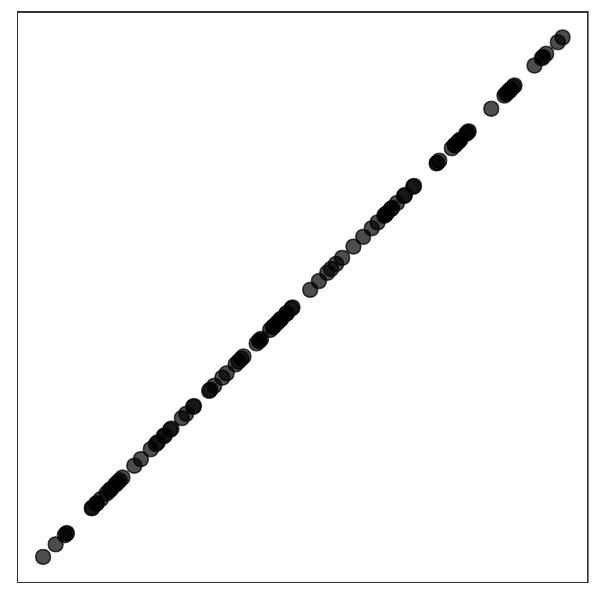}
        \caption*{\scalebox{0.95}{\tiny$\nu_n = 0.985$}}
    \end{subfigure}
    \begin{subfigure}[b]{0.3\textwidth}
        \includegraphics[width=\textwidth]{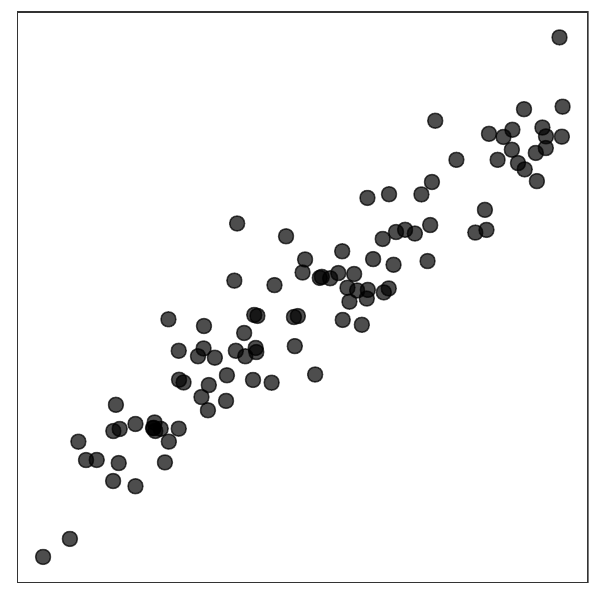}
        \caption*{\scalebox{0.95}{\tiny$\nu_n = 0.772$}}
    \end{subfigure}
    \begin{subfigure}[b]{0.3\textwidth}
        \includegraphics[width=\textwidth]{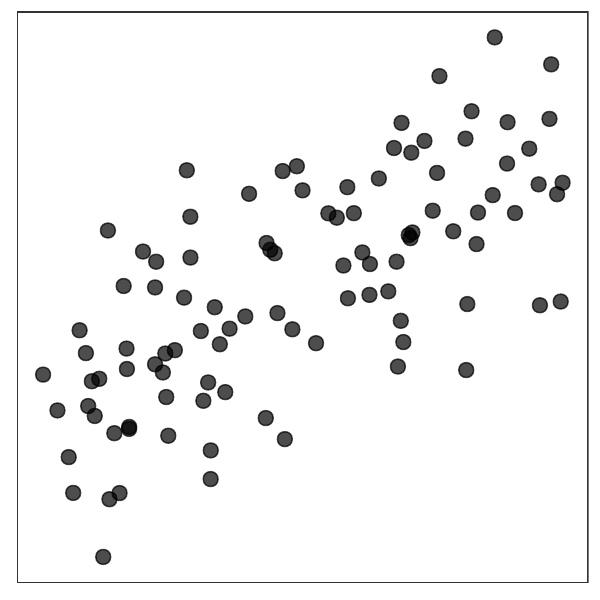}
        \caption*{\scalebox{0.95}{\tiny$\nu_n = 0.212$}}
    \end{subfigure}

    \begin{subfigure}[b]{0.3\textwidth}
        \includegraphics[width=\textwidth]{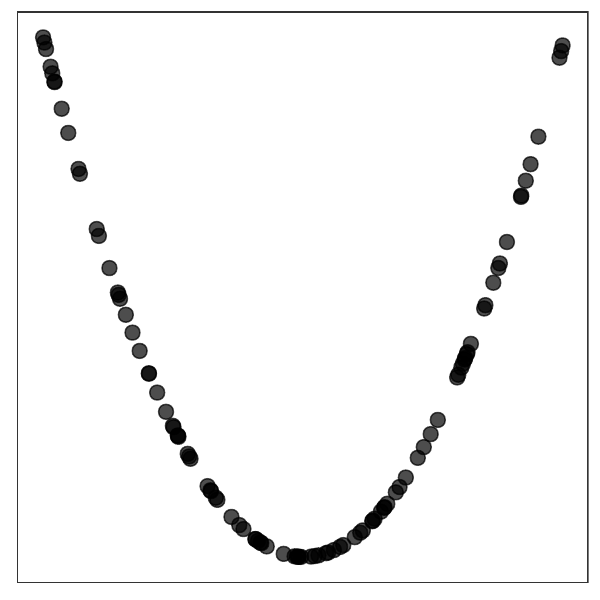}
        \caption*{\scalebox{0.95}{\tiny$\nu_n = 0.973$}}
    \end{subfigure}
    \begin{subfigure}[b]{0.3\textwidth}
        \includegraphics[width=\textwidth]{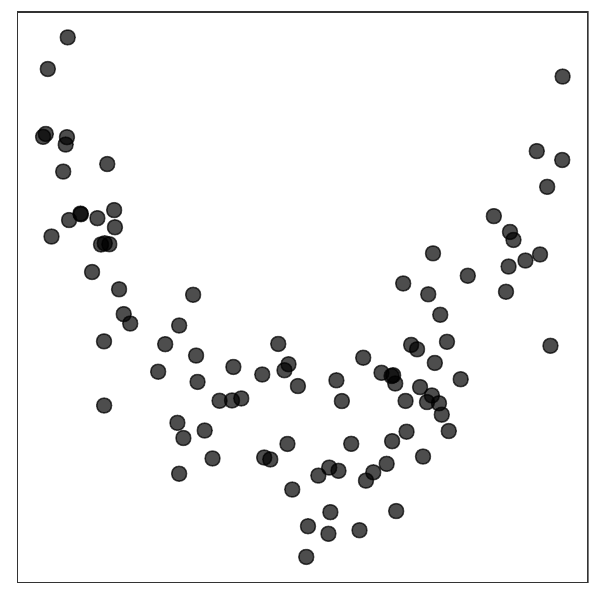}
        \caption*{\scalebox{0.95}{\tiny$\nu_n = 0.703$}}
    \end{subfigure}
    \begin{subfigure}[b]{0.3\textwidth}
        \includegraphics[width=\textwidth]{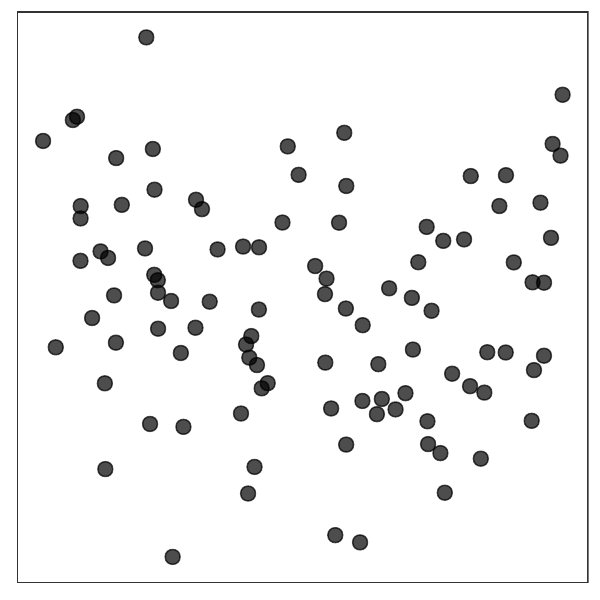}
        \caption*{\scalebox{0.95}{\tiny$\nu_n = 0.284$}}
    \end{subfigure}

    \begin{subfigure}[b]{0.3\textwidth}
        \includegraphics[width=\textwidth]{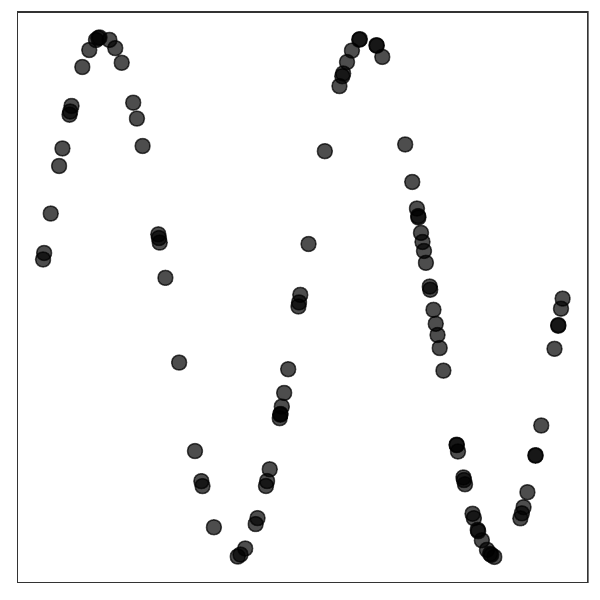}
        \caption*{\scalebox{0.95}{\tiny$\nu_n = 0.901$}}
    \end{subfigure}
    \begin{subfigure}[b]{0.3\textwidth}
        \includegraphics[width=\textwidth]{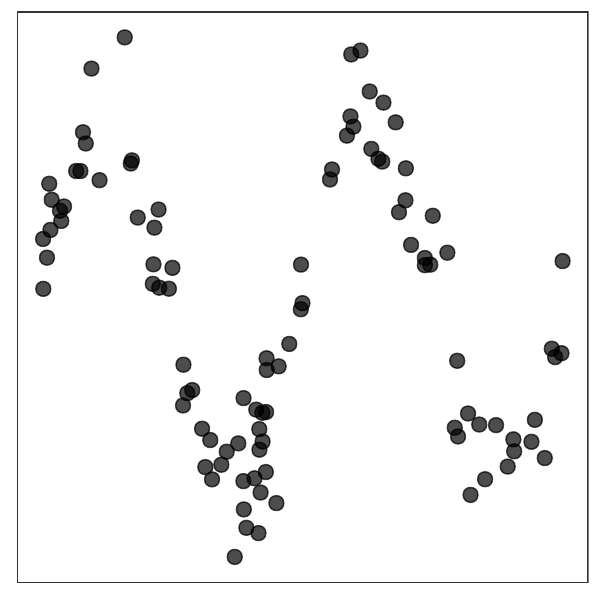}
        \caption*{\scalebox{0.95}{\tiny$\nu_n = 0.667$}}
    \end{subfigure}
    \begin{subfigure}[b]{0.3\textwidth}
        \includegraphics[width=\textwidth]{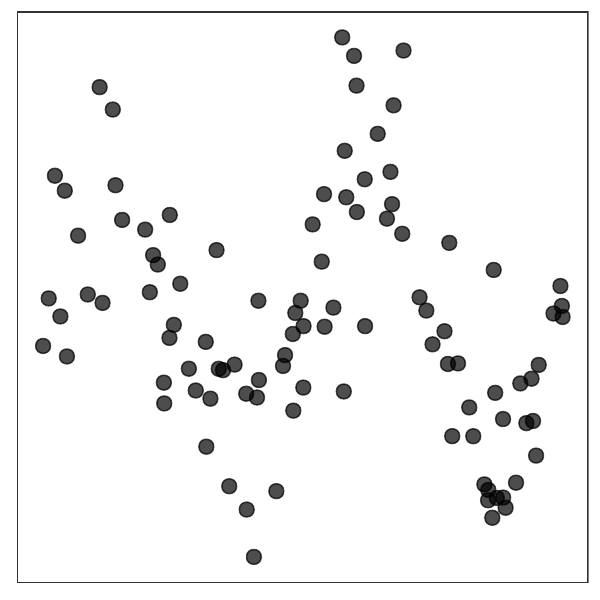}
        \caption*{\scalebox{0.95}{\tiny$\nu_n = 0.301$}}
    \end{subfigure}
    \caption{Values of $\nu_n(Y, X)$ for various kinds of scatterplots with $n = 100$. Noise increases from left to right.}
    \label{fig:exGen}
\end{figure}
In each row, we observe that $\nu_n$ is close to 1 in the leftmost plot and gradually decreases as more noise is introduced. In each column, we observe that the values of $\nu_n$ are comparable, meaning that $\nu_n$ satisfies the notion of \textit{equitability} defined in \cite{reshef2011detecting}: ``to assign similar scores to equally noisy relationships of different types''.
\end{ex}

\begin{ex}\label{exAsymptotic}(asymptotic behaviour) We numerically study the distribution of $\nu_{n}^{\text{1-dim}}(Y, X)$ under the assumption that $Y$ and $X$ are independent. In particular, we take the $\{X_i\}$ and $\{Y_i\}$ to be independent and identically distributed $\mathrm{Uniform}[0,1]$ random variables and focus first on the case $n = 20$. Using $10{,}000$ Monte Carlo replications, we obtain the empirical distribution of $\nu_{n}^{\text{1-dim}}(Y, X)$; the resulting histogram is presented in Figure~\ref{fig:hist_unif_ind20}. Even at this relatively small sample size, the normal approximation provides a reasonable fit. For comparison, Figure~\ref{fig:hist_unif_ind1000} displays the corresponding histogram for $n = 1000$, where the alignment with the normal distribution becomes even more pronounced.
\begin{figure}[ht]
    \centering
    \begin{subfigure}[b]{0.4\textwidth}
        \includegraphics[width=\textwidth]{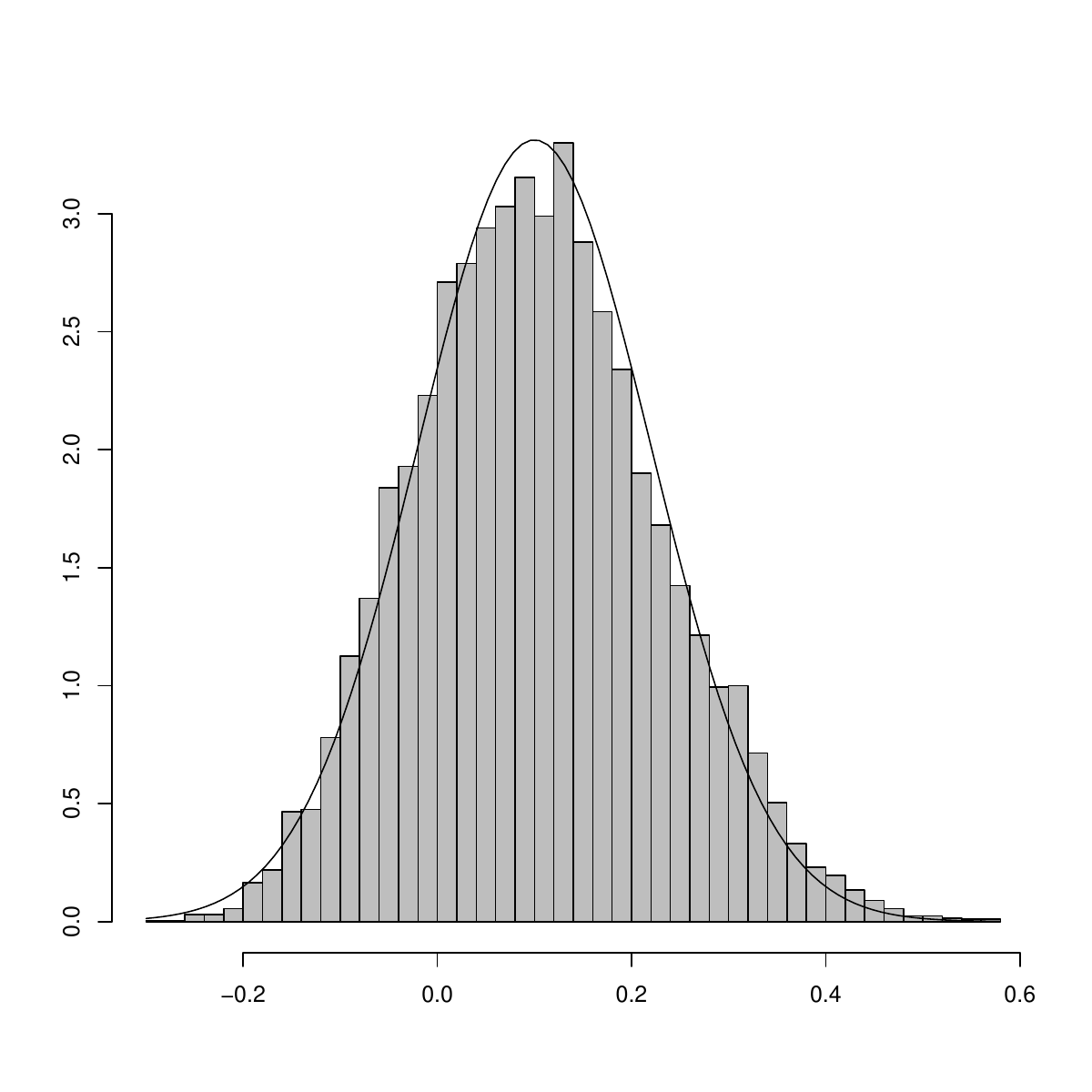}
        \caption{$n = 20$}
        \label{fig:hist_unif_ind20}
    \end{subfigure}
    ~ 
    \begin{subfigure}[b]{0.4\textwidth}
        \includegraphics[width=\textwidth]{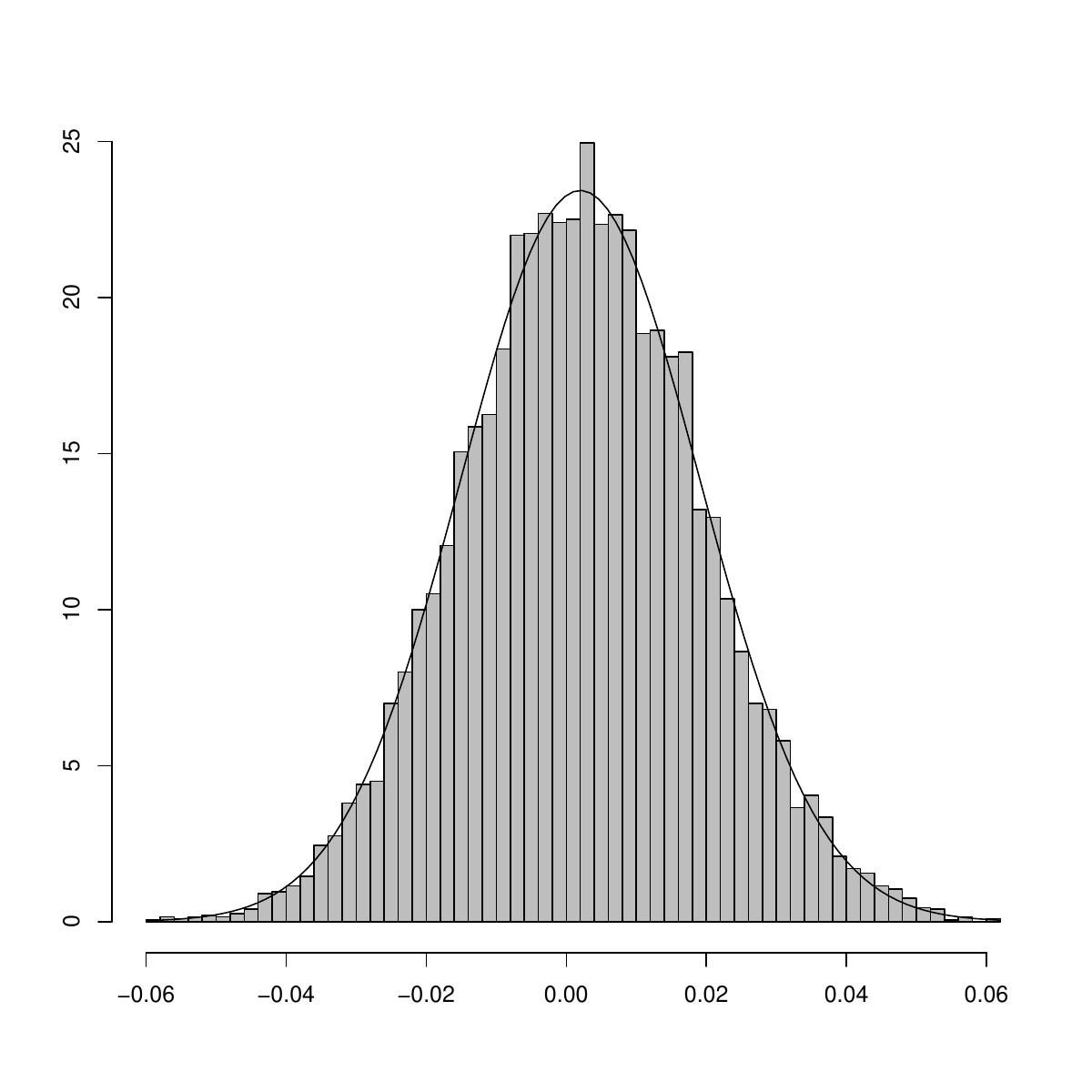}
        \caption{$n = 1000$}
        \label{fig:hist_unif_ind1000}
    \end{subfigure}
    \caption{Histogram of $10000$ simulations of $\nu_{n}^{\text{1-dim}}(Y, X)$ with $X$ and $Y$ independently distributed as \text{Uniform}$[0, 1]$, overlaid with the asymptotic normal density $N(\mu_n, \sigma_n^2)$, where $\mu_n = 2/n$ and $\sigma_n^2 = (\pi^2/3 - 3)/n$.}
    \label{fig:histInd}
\end{figure}
We also examine a setting where $X$ and $Y$ are dependent. To this end, we consider the following simple model: let $X$ and $Z$ be independent random variables, each distributed as \text{Uniform}$[0, 1]$, and define $Y:= XZ$. We have
\begin{equation*}
    \nu(Y,X) = \int_0^1 \frac{1 + 2t \log t - t^2 - \left(1 - t + t \log t\right)^2}{\left(1 - t + t \log t\right)\left(t - t \log t\right)} \cdot (-\log t) \, dt
\end{equation*}
which is approximately equal to $0.3126$. To study the asymptotic behaviour of $\nu_{n}^{\text{1-dim}}(Y, X)$, we perform $10000$ simulations with $n = 1000$. The sample mean of $\nu_{n}^{\text{1-dim}}(Y, X)$ is approximately $0.314$, with a standard deviation of about $0.02$. The resulting histogram, shown in Figure~\ref{fig:histUnifDep1000}, exhibits an excellent fit with a normal distribution having the same mean and standard deviation.
\begin{figure}[ht]
  \centering
    \includegraphics[width=0.5\textwidth]{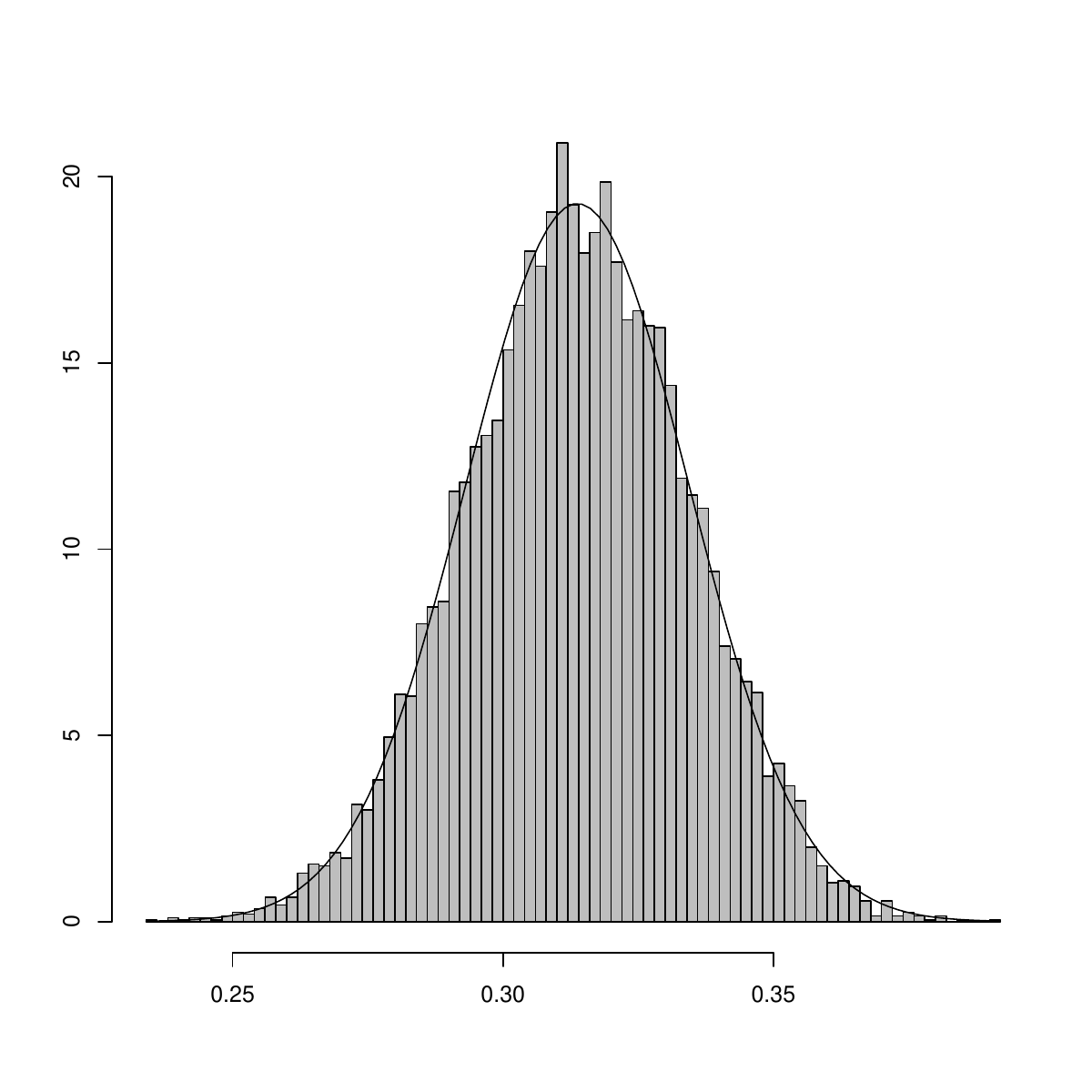}
    \caption{Histogram of $10{,}000$ simulations of $\nu_{n}^{\text{1-dim}}(Y, X)$ under the dependence structure between $X$ and $Y$ described in Example~\ref{exAsymptotic}, overlaid with the normal density curve whose estimated mean and standard deviation are $0.314$ and $0.02$, respectively.}
    \label{fig:histUnifDep1000}
\end{figure}
\end{ex}

\begin{ex}\label{ex:cond_estim}(conditional dependence) Let $X_1$ and $X_2$ be independent $\mathrm{Uniform}[0,1]$ random variables, and define
$Y \coloneqq (X_1 + X_2)\,(\mathrm{mod}\,1).$
The relationship between $Y$ and $(X_1,X_2)$ has the following properties: (i) $Y$ is a function of $(X_1,X_2)$; (ii) unconditionally, $Y$ is independent of $X_2$; (iii) conditional on $X_1$, $Y$ is a function of $X_2$.

Consider the corresponding sample $\{(Y_i, X_{1i}, X_{2i})\}_{i=1}^n$ with $n = 1000$. In approximately $95\%$ of the simulations, 
$\nu_n(Y,(X_1,X_2))$ took values between $0.824$ and $0.891$, 
$\nu_n(Y, X_2 \mid X_1)$ lay between $0.821$ and $0.892$, 
and $\nu_n(Y, X_2)$ ranged from $-0.048$ to $0.046$, consistent with the established properties. 

These results demonstrate that $\nu_n$ effectively captures strong conditional dependence, similar to the statistic $T$ in~\cite{azadkia2021simple}, whereas some alternative measures of conditional dependence—such as conditional distance correlation \cite{Wang2015ConditionalDistanceCorrelation}—fail to quantify the strength of the conditional dependence between $Y$ and $X_2$ given $X_1$.
\end{ex}

\begin{ex}\label{exPowerAnalysis}(power comparison $p=1$) In this example, we assess the power of the independence test based on $\nu_n$ and its one-dimensional variant $\nu_{n}^{\text{1-dim}}$, and compare their performance against several recently proposed, powerful tests. The test statistics included in our comparison are: Maximal information coefficient (MIC)~\cite{reshef2011detecting}, Distance correlation~\cite{szekely2007measuring}, the Hilbert–Schmidt independence criterion (HSIC)~\cite{gretton2005measuring, gretton2008kernel}, the HHG statistic~\cite{heller2013consistent}, Chatterjee's $\xi_n$ xicor correlation coefficient~\cite{chatterjee2021new}, and $T_n$ statistics~\cite{azadkia2021simple}.  
This experiment is conducted in two separate settings: univariate and multivariate.

We consider $(X_1, Y_1), \ldots, (X_n, Y_n)$ an i.i.d. sample drawn from a distribution on $\rr^2$. We adopt the same experimental setup as described in Section~4.3 of~\cite{chatterjee2021new}. Power comparisons were conducted with a sample size of $n = 100$, using $500$ simulations to estimate the power in each scenario. The variable $X$ was generated from the uniform distribution on $[-1, 1]$, the noise parameter $\lambda$ ranged from 0 to 1, and the noise variable $\varepsilon \sim N(0,1)$, which is independent of $X$. The following six alternatives were considered:

\begin{enumerate}
    \item Linear: $Y = 0.5 X + 3 \lambda \varepsilon$, 
    \item Step function: $Y = f(X)+10 \lambda \varepsilon$, where $f$ takes values $-3,2,-4$ and $-3$ in the intervals $[-1,-0.5),[-0.5,0),[0,0.5)$ and $[0.5,1]$,
    \item W-shaped: $Y=\abs{X + 0.5} \bone\{X<0\} + \abs{X-0.5} \bone\{X \geq 0\} + 0.75 \lambda \varepsilon$,
    \item Sinusoid: $Y = \cos 8 \pi X + 3 \lambda \varepsilon$,
    \item Circular: $Y=Z \sqrt{1-X^2} + 0.9 \lambda \varepsilon$, where $Z$ is 1 or -1 with equal probability, independent of $X$,
    \item Heteroskedastic: $Y = 3(\sigma(X)(1-\lambda) + \lambda) \varepsilon$, where $\sigma(X) = 1$ if $\abs{X} \leq$ 0.5 and 0 otherwise.
\end{enumerate}
The \texttt{R} packages \texttt{energy}~\cite{energy}, \texttt{minerva}~\cite{minerva}, \texttt{HHG}~\cite{HHG}, \texttt{dHSIC}~\cite{dHSIC}, \texttt{XICOR}~\cite{chatterjee2020xicor}, and \texttt{FOCI}~\cite{azadkia1foci} were employed to compute the distance correlation, MIC, HHG, HSIC, $\xi_n$ and $T_n$ statistics, respectively. The p-values were calculated using $1000$ independent permutations and the power is estimated at the significance level of $5\%$.

The plots in Figure~\ref{fig:power} illustrate that $\nu_n$ and $\nu_n^{\text{1-dim}}$ are competitive with $\xi_n$ and outperform other tests in scenarios where the underlying dependency has an oscillatory structure, such as the W-shaped and sinusoidal settings. However, their power is relatively lower for smooth alternatives like the linear, circular, and heteroskedastic patterns.

A comparison between $\nu_n^{\text{1-dim}}$ and its counterpart $\xi_n$, as well as between $\nu_n$ and $T_n$, reveals consistently slightly higher power for the former in both pairs. Furthermore, across all alternatives, the simpler one-dimensional statistics, $\nu_n^{\text{1-dim}}$ and $\xi_n$, tend to outperform their more flexible counterparts, $\nu_n$ and $T_n$, respectively. This advantage is likely due to their reduced variance. Specifically, the simpler methods use only the immediate next neighbour when ordering the predictor $X$, whereas the more complex versions can choose freely between preceding and succeeding neighbours. This added flexibility introduces higher variability in the estimation, reducing power.

\begin{figure}[H]
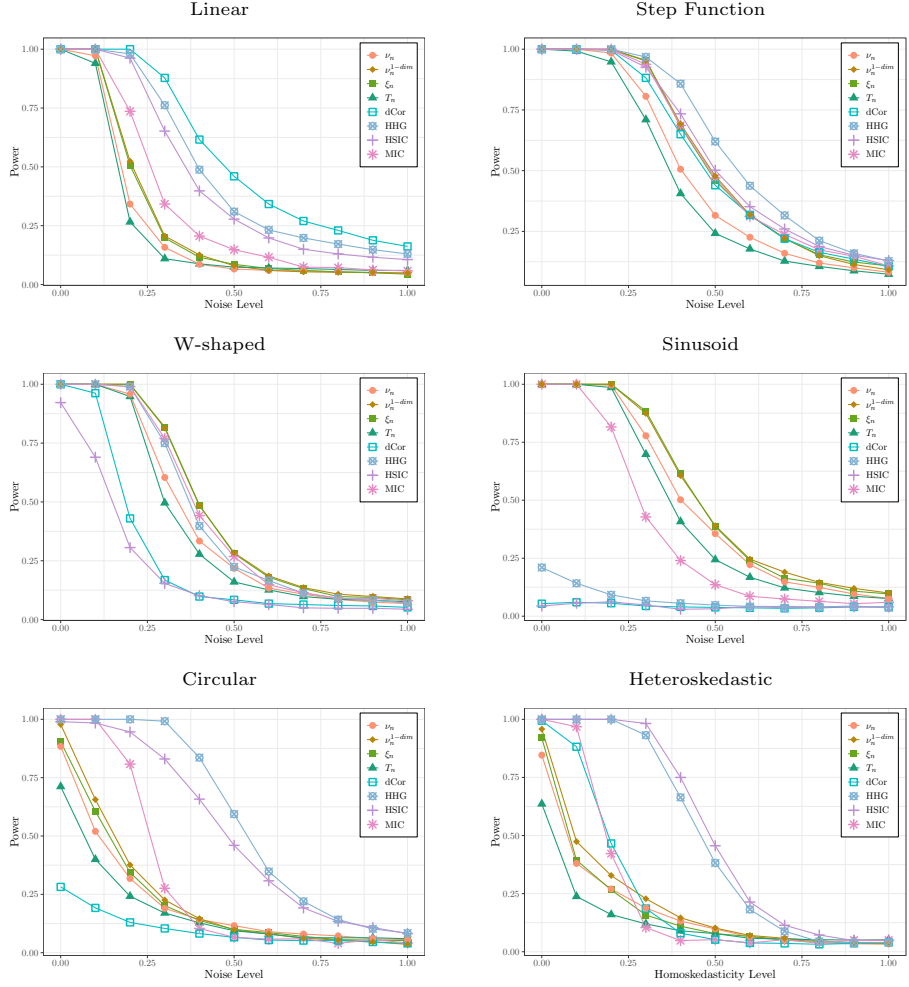

\centering
\begin{tabular}{cc}
\begin{tabular}{c}
\tiny Linear \\[1mm]
\resizebox{0.45\linewidth}{!}{\input{power_vs_lambda_linear}}
\end{tabular}
&
\begin{tabular}{c}
\tiny Step Function \\[1mm]
\resizebox{0.45\linewidth}{!}{\input{power_vs_lambda_step}}
\end{tabular}
\\[6mm]

\begin{tabular}{c}
\tiny W-shaped \\[1mm]
\resizebox{0.45\linewidth}{!}{\input{power_vs_lambda_W}}
\end{tabular}
&
\begin{tabular}{c}
\tiny Sinusoid \\[1mm]
\resizebox{0.45\linewidth}{!}{\input{power_vs_lambda_sin}}
\end{tabular}
\\[6mm]

\begin{tabular}{c}
\tiny Circular \\[1mm]
\resizebox{0.45\linewidth}{!}{\input{power_vs_lambda_circular}}
\end{tabular}
&
\begin{tabular}{c}
\tiny Heteroskedastic \\[1mm]
\resizebox{0.45\linewidth}{!}{\input{power_vs_lambda_heteroskedastic}}
\end{tabular}
\end{tabular}
\caption{Comparison of power of several tests of independence described in Example~\ref{exPowerAnalysis}. The level of the noise or homoskedasticity increases from left to right. In each
case, the sample size is 100, and 500 simulations were used to estimate the power. The p-values were calculated using 1000 independent permutations.}
\label{fig:power}
\end{figure} 
In addition, we consider the following alternatives which highlights some settings that $\nu_n$ and $\nu_n^{\text{1-dim}}$ achieve significantly higher power.
\begin{enumerate}
  \setcounter{enumi}{6}
  \item Heteroskedastic sinusoid: $Y = \cos (20\pi(1 + 10\lambda \varepsilon)X^2)$.
  
  \item Oscillatory in the tails: $Y = \mathbbm{1}\{|X| \le  \lambda \}\, U + \mathbbm{1}\{|X| > \lambda \}\,\cos(10\pi X^2 + U / 10)$, where $U \sim \mathrm{Uniform}[-1, 1]$.
\end{enumerate}
Figure~\ref{fig:power_best_uni} illustrate that in these cases $\nu_n^{\text{1-dim}}$ and $\nu_n$ appear more powerful than other tests, including $\xi_n$. These examples demonstrates that the new coefficient is more effective at detecting sinusoidal relationships and less sensitive to heteroskedasticity compared to $\xi_n$.

\begin{figure}[ht]
\centering
\begin{tabular}{cc}
\begin{tabular}{c}
Heteroskedastic and Sinusoid \\[2mm]
\resizebox{0.45\linewidth}{!}{\input{power_vs_lambda_mix18}}
\end{tabular}
&
\begin{tabular}{c}
Oscillatory in tails  \\[2mm]
\resizebox{0.45\linewidth}{!}{\input{power_vs_lambda_osc_tail67}}
\end{tabular}
\\[6mm]
\end{tabular}
\caption{Comparison of the empirical power of several tests of independence described in Example~\ref{exPowerAnalysis}. The noise level (or degree of homoskedasticity) increases from left to right. The sample size is 
$n=100$, and power is estimated based on 500 Monte Carlo simulations. P-values are computed using 1,000 independent permutations.}
\label{fig:power_best_uni}
\end{figure}
\end{ex}

\begin{ex}\label{ex:multi}(power comparison $p = 3$)
In this experiment, we consider a multivariate predictor $\bbx \in \rr^3$. Specifically,
$(\bbx_1, Y_1), \ldots, (\bbx_n, Y_n)$ are i.i.d.\ samples drawn from a joint
distribution $(\bbx, Y)\in\rr^4$. We adopt the same experimental framework as in Example~\ref{exPowerAnalysis} and conduct power
comparisons with sample size $n = 100$, using 100 simulations to estimate the power in
each scenario. The predictor $\bbx$ is generated from a multivariate normal
distribution $N(\mathbf{0}, \mathbf{I}_3)$. The noise parameter $\lambda$ ranges from
$0$ to $1$, and the noise variable $\varepsilon \sim N(0,1)$ is independent of
$\bbx = (X_1, X_2, X_3)$. We consider the following alternatives:
\begin{enumerate}
    \item Linear:
    $Y = 3 X_1 + 2 X_2 - 3 X_3 + 20 \lambda \varepsilon$.
    \item Non-linear:
    $Y = X_1 X_2 X_3 + X_1 / X_3 + 5 \lambda \varepsilon$.
    \item Oscillatory:
    $Y = \sin(\pi \sqrt{X_1^2 + X_2^2 + X_3^2}) + 2 \lambda \varepsilon$.
    \item XOR:
    $Y = \operatorname{sign}(X_1 X_2 X_3) + 2 \lambda \varepsilon$.
\end{enumerate}
In this multivariate setting, we compare $\nu_n$ with distance correlation, HHG, HSIC, and
$T_n$, since all of these methods extend naturally to multivariate predictors. We use
1000 independent permutations to compute the $p$-values, and estimate power at the
$5\%$ significance level.
\begin{figure}[H]
\centering
\begin{subfigure}[t]{0.48\linewidth}
\centering
Linear\\[2mm]
\resizebox{\linewidth}{!}{\input{power_vs_lambda_multivariate_lin}}
\end{subfigure}
\hfill
\begin{subfigure}[t]{0.48\linewidth}
\centering
Non-linear\\[2mm]
\resizebox{\linewidth}{!}{\input{power_vs_lambda_multivariate_nonlin3}}
\end{subfigure}
\vspace{3mm}
\begin{subfigure}[t]{0.48\linewidth}
\centering
Oscillatory\\[2mm]
\resizebox{\linewidth}{!}{\input{power_vs_lambda_multivariate_osc3}}
\end{subfigure}
\hfill
\begin{subfigure}[t]{0.48\linewidth}
\centering
XOR\\[2mm]
\resizebox{\linewidth}{!}{\input{power_vs_lambda_multivariate_xor}}
\end{subfigure}

\caption{Comparison of the power of several tests of independence in the multivariate setting (selected nonlinear, linear, oscillatory, and XOR alternatives) described in Example~\ref{ex:multi}. Sample size $n = 100$; 500 simulations; $p$-values computed via 1000 permutations.}
\label{fig:power_multivariate_subset}
\end{figure}
Figure~\ref{fig:power_multivariate_subset} shows that the
proposed statistic $\nu_n$ consistently outperforms competing methods in scenarios
involving oscillatory or strongly nonlinear alternatives. Its superior performance is even more pronounced across a broader range of
alternatives in the multivariate setting compared to the univariate case. 
\end{ex}

\begin{ex}\label{exTime}(time complexity)
In this example, we compare the computational complexity of several dependence measures: 
$\xi_n$ from \cite{chatterjee2021new} implemented in the \texttt{R} package \texttt{XICOR}~\cite{chatterjee2020xicor}; 
$T_n$ from \cite{azadkia2021simple} implemented in the \texttt{R} package \texttt{FOCI}; 
the kernel-based measures $\widehat{\rho^2}$ and $\widetilde{\rho^2}$ from \cite{deb2022kernelpartialcorrelation} implemented in the \texttt{R} package \texttt{KPC}~\cite{KPCpackage}; 
and the proposed coefficients $\nu_n$ and $\nu_n^{\text{1-dim}}$. 
As noted in \cite[Table~2]{chatterjee2021new}, $\xi_n$ is hundreds to thousands of times faster than other widely used dependence measures, including MIC~\cite{reshef2011detecting}, distance correlation~\cite{szekely2007measuring}, HSIC~\cite{gretton2005measuring,gretton2008kernel}, and the HHG statistic~\cite{heller2013consistent}. 
Therefore, we restrict our comparison to $\xi_n$ and $T_n$, the proposed coefficients $\nu_n$ and $\nu_n^{\text{1-dim}}$, and the recently developed kernel-based coefficients $\widehat{\rho^2}$ and $\widetilde{\rho^2}$.

We independently sample $X$ and $Y$ from the standard normal distribution and perform $100$ replications. The average computation time in seconds for each method is reported in Table~\ref{table:time_simulations}. The most efficient methods are $\nu_{n}^{\text{1-dim}}$ and $\xi_n$, both exhibiting $O(n \log n)$ computational complexity. The superior runtime of $\nu_n^{\text{1-dim}}$ relative to $\xi_n$ is due to implementation efficiency rather than a difference in asymptotic order. Although $\xi_n$ may appear faster when constant weights are used, as discussed in Section~\ref{sec:comparision_with_xi}, the weight computation for $\nu_{n}^{\text{1-dim}}$ is rank-based and exploits ranks that are already computed and reused as part of the statistic, incurring no additional cost and preserving the $O(n \log n)$ complexity. The statistics $\nu_n$ and $T_n$ also operate in $O(n \log n)$ time. In contrast, the kernel-based measures $\widehat{\rho^2}$ and $\widetilde{\rho^2}$ are substantially more computationally demanding, with a computational complexity of $O(n^2)$.

\begin{table}[ht]
\centering
\small
\begin{tabular}{lcccccc}
\toprule
$n$ & $\nu_n^{\text{1-dim}}$ & $\nu_n$ & $\xi_n$ & $T_n$ & $\widehat{\rho^2}$ & $\widetilde{\rho^2}$ \\
\hline
10    & \textbf{0.00035} & 0.00098 & 0.00092 & 0.01092 & 0.01468 & 0.01036 \\
31    & \textbf{0.00039} & 0.00133 & 0.00059 & 0.00323 & 0.01046 & 0.00982 \\
100   & \textbf{0.00044} & 0.00311 & 0.00069 & 0.00417 & 0.01886 & 0.01558 \\
316   & \textbf{0.00049} & 0.00866 & 0.00079 & 0.00734 & 0.05568 & 0.11863 \\
1000  & \textbf{0.00076} & 0.02761 & 0.00114 & 0.01684 & 0.33250 & 3.03182 \\
3162  & \textbf{0.00176} & 0.11807 & 0.00247 & 0.04560 & 3.11779 & 88.56498 \\
10000 & \textbf{0.00485} & 0.68661 & 0.00731 & 0.14825 & 34.68341 & 2604.97461 \\
\bottomrule
\end{tabular}
\caption{Average runtime (in seconds) of various dependence measures across different sample sizes. The lowest runtime in each row is shown in bold.}
\label{table:time_simulations}
\end{table}
\end{ex}

\begin{ex}\label{exVar}(variable selection with built-in stopping rules) We evaluate the performance of FORD and compare it with FOCI \cite{azadkia2021simple} across a variety of settings. Both FORD and FOCI are model-free, require no tuning parameters, and include built-in stopping rules. In contrast, the high computational complexity of $\widehat{\rho^2}$ and $\widetilde{\rho^2}$ (see Table~\ref{table:time_simulations}) makes KFOCI~\cite{deb2022kernelpartialcorrelation} substantially slower than both FOCI and FORD. Repeated experiments at larger sample sizes ($n = 500$ and $n = 1000$) become prohibitively time-consuming. Moreover, $\widehat{\rho^2}$ and $\widetilde{\rho^2}$—and therefore KFOCI—require hyperparameter tuning, which further increases computational and methodological complexity. For these reasons, we do not report results for KFOCI in this section. In addition, the strong empirical performance of FOCI relative to competing methods such as
LASSO \cite{tibshirani1996regression}, the Dantzig selector \cite{candes2007dantzig},
and SCAD \cite{fan2001variable} has been demonstrated in detail in
\cite[Examples~8.3 and~8.4]{azadkia2021simple} and
\cite[Subsection~6.2.1]{deb2022kernelpartialcorrelation}. Consequently, we do not repeat those comparisons here and focus exclusively on comparing FORD and FOCI in this example.

We consider the following models with sample size $n\in\{100, 500, 1000\}$, covariates $\bbx = (X_1, \ldots, X_p) \sim N(\mathbf{0}, \mathbf{I}_p)$ with $\mathbf{I}_p$ the $p$ by $p$ identity matrix where $p = 1000$, and independent noise variable $\varepsilon$:

\begin{enumerate}
    \item  LM (linear model): $Y = 3X_1 + 2X_2 - X_3 + \varepsilon$, $\varepsilon\sim N(0,1)$

    \item Nonlin1 (nonlinear model):
    $Y = X_1 X_2 + \sin(X_1 X_3)$

    \item Nonlin2 (non-additive noise): 
    $Y = |X_1 + \varepsilon|^{\sin(X_2 - X_3)}$, $\varepsilon \sim \text{Uniform}[0,1]$

    \item Osc1 (oscillatory):
    $Y = \sin(X_1)/\sqrt{|X_1|} + X_2 X_3$

    \item Osc2 (oscillatory with interaction):
    $Y = \sin(X_2)/X_1 + X_1 X_3$
\end{enumerate}
For the implementation, we use the \texttt{R} packages \texttt{FOCI}~\cite{azadkia1foci} and \texttt{FORD}~\cite{azadkiafordpackage}. In all the models considered, the true Markov blanket of $Y$ is $\{X_1, X_2, X_3\}$. Table~\ref{table:simulations} presents the results over $1000$ iterations, summarising the following:
\begin{enumerate}
    \item The proportion of times $\{X_1, X_2, X_3\}$ is exactly recovered,
    \item The proportion of times $\{X_1, X_2, X_3\}$ has been selected, possibly along with additional variables,
    \item The average number of falsely selected variables.
\end{enumerate}

The results in Table~\ref{table:simulations} show that FORD consistently outperforms FOCI across all linear and nonlinear models considered, both in terms of exact recovery and fewer falsely selected variables. 
\begin{table}[ht]
    \centering
    \small
    \renewcommand{\arraystretch}{1.2}
    \begin{tabular}{lccc}
        \toprule
         &  & FORD & FOCI   \\
         Models & n & exact/inclusion/avg.false. & exact/inclusion/avg.false.  \\
        \midrule
        LM      & 100  & \textbf{0.030}/\textbf{0.303}/\textbf{1.609} & 0.003/0.064/2.720 \\
        LM      & 500  & \textbf{0.526}/\textbf{1.000}/\textbf{0.474} & 0.103/0.974/0.932 \\
        LM      & 1000 & \textbf{0.808}/\textbf{1.000}/\textbf{0.192} & 0.253/1.000/0.748 \\
        \midrule
        Nonlin1 & 100  & \textbf{0.015}/\textbf{0.063}/\textbf{3.281} & 0.001/0.015/3.948 \\
        Nonlin1 & 500  & \textbf{0.228}/\textbf{0.479}/\textbf{1.517} & 0.061/0.158/2.445 \\
        Nonlin1 & 1000 & \textbf{0.547}/\textbf{0.824}/\textbf{0.620} & 0.172/0.347/1.751 \\
        \midrule
        Nonlin2 & 100  & 0.000/\textbf{0.002}/\textbf{3.205} & 0.000/0.000/3.988 \\
        Nonlin2 & 500  & \textbf{0.059}/\textbf{0.259}/\textbf{2.091} & 0.004/0.073/2.826 \\
        Nonlin2 & 1000 & \textbf{0.245}/\textbf{0.520}/\textbf{1.388} & 0.042/0.162/2.280 \\
        \midrule
        Osc1  & 100  & \textbf{0.028}/\textbf{0.116}/\textbf{3.071} & 0.001/0.026/3.519 \\
        Osc1  & 500  & \textbf{0.572}/\textbf{0.802}/\textbf{0.602} & 0.243/0.382/1.319 \\
        Osc1  & 1000 & \textbf{0.938}/\textbf{0.992}/\textbf{0.070} & 0.574/0.752/0.569 \\
        \midrule
        Osc2  & 100  & \textbf{0.004}/\textbf{0.026}/\textbf{3.004} & 0.000/0.004/4.046 \\
        Osc2  & 500  & \textbf{0.418}/\textbf{0.661}/\textbf{1.054} & 0.038/0.098/2.754 \\
        Osc2  & 1000 & \textbf{0.809}/\textbf{0.966}/\textbf{0.233} & 0.117/0.229/2.108 \\
        \bottomrule
    \end{tabular}
    \caption{Proportion of times the Markov boundary was exactly recovered, the proportion it was included in the selected set, and the average number of falsely selected variables across $1000$ iterations. For each row, the better-performing method is highlighted in bold. Models described in Example~\ref{exVar}.}
    \label{table:simulations}
\end{table}
\end{ex}

\begin{ex}\label{exVar_screening}(Variable selection with oracle stopping rules) We compare FORD with the Sure Independence Screening (SIS) method \cite{Fan2008SureIndependenceScreening} and its variants, which are designed for ultra–high-dimensional settings. These methods rely on marginal dependence between covariates and the response and serve primarily as a preliminary screening step that yields a reduced variable set to be forwarded to a downstream selection procedure. However, in moderately high-dimensional regimes, SIS-based approaches can be suboptimal: because they depend exclusively on marginal associations, they may fail to recover the correct Markov blanket when signal variables are correlated with other covariates.

In this experiment, we evaluate the performance of FORD and compare it with FOCI~\cite{azadkia2021simple}, as well as two representative SIS methods\footnote{There exists a large class of screening methods based on marginal dependence,
defined using different dependence measures. In principle, one could also construct SIS
procedures based on $\nu_n$, $T_n$, or $\xi_n$. Our choice of methods is guided by the
availability of reliable implementations, as well as computational and memory
considerations. For example, PCSIS, which is based on projection correlation, is substantially more
computationally and memory intensive than the other methods considered here. In our
experiments with $n=1000$ and $p=200$, PCSIS required more than 8 seconds of computation
time and over 1.6~GB of memory, whereas the remaining methods completed in under
2 seconds with significantly lower memory requirements.}:
Distance Correlation Sure Independence Screening (DCSIS) \cite{Li2012DistanceCorrelation}
and Ball Correlation Sure Independence Screening (BCORSIS)
\cite{pan2020ball,Pan2019SureIndependenceScreening}.

We use the available implementations in the \texttt{R} packages \texttt{FORD}~\cite{azadkiafordpackage},
\texttt{FOCI}~\cite{azadkia1foci}, \texttt{MFSIS}~\cite{MFSIS}, and \texttt{Ball}~\cite{Ball}. The most commonly used stopping rule for
SIS methods is an oracle rule that selects the top covariates ranked by marginal
dependence. Therefore, we use the true number of signal variables as stopping rule for all methods.

We consider the following models with sample sizes $n \in \{100, 500, 1000\}$ and
covariates $\bbx = (X_1, \ldots, X_p)$ with $p = 100$ where noise variable $\varepsilon$ is independent of $\bbx$ and $X_i\sim N(0, 1)$.
\begin{enumerate}
    \item LM:
    $Y = 3X_1 + 2X_2 - X_3 + \varepsilon$,
    $\varepsilon \sim N(0,1)$,
    $(X_1, \ldots, X_{100}) \sim N(\mathbf{0}, \mathbf{I}_{100})$.
    \item LM-corr:
    $Y = 3X_1 + 2X_2 - X_3 + \varepsilon$ with $X_1, X_2$ and $X_3$ i.i.d.,
    $\varepsilon \sim N(0,1)$,
    and $\mathrm{corr}(X_m, X_1) = 0.7$ for $m \in \{4, \ldots, 100\}$, where $\mathrm{corr}$ denotes Pearson correlation.
    \item Nonlin2:
    $Y = |X_1 + \varepsilon|^{\sin(X_2 - X_3)}$,
    $\varepsilon \sim \mathrm{Uniform}[0,1]$,
    $(X_1, \ldots, X_{100}) \sim N(\mathbf{0}, \mathbf{I}_{100})$.
    \item Nonlin2-corr:
    $Y = |X_1 + \varepsilon|^{\sin(X_2 - X_3)}$ with $X_1, X_2$ and $X_3$ i.i.d.,
    $\varepsilon \sim \mathrm{Uniform}[0,1]$,
    with $\mathrm{corr}(X_m, X_1) = 0.7$ for $m \in \{4, \ldots, 100\}$, where $\mathrm{corr}$ denotes Pearson correlation.
    \item Osc2:
    $Y = \sin(X_2)/X_1 + X_1 X_3$, with $(X_1, \ldots, X_{100}) \sim N(\mathbf{0}, \mathbf{I}_{100})$
    \item Osc2-corr:
    $Y = \sin(X_2)/X_1 + X_1 X_3$ with $X_1, X_2$ and $X_3$ i.i.d.,
    with $\mathrm{corr}(X_m, X_1) = 0.7$ for $m \in \{4, \ldots, 100\}$, where $\mathrm{corr}$ denotes Pearson correlation.
\end{enumerate}
In all models, the true Markov blanket of $Y$ is $\{X_1, X_2, X_3\}$. Table~\ref{table:corr_simulations}
summarizes results over $1000$ Monte Carlo replications, reporting:
\begin{enumerate}
    \item the proportion of exact recovery of $\{X_1, X_2, X_3\}$,
    \item the average number of truly selected variables,
    \item the average number of falsely selected variables.
\end{enumerate}
Since SIS methods require a pre-specified model size, the total number of selected
variables is fixed at three in these experiments. Consequently, the inclusion and exact
recovery rates coincide, and we report the average numbers of true and false
selections. 

Table~\ref{table:corr_simulations} shows that the presence of collinearity between signal
and noise variables substantially degrades the performance of DCSIS and BCORSIS, which
primarily capture the strongest marginal signal $X_1$ along with correlated variables. The convergence of the average number of truly selected variables to one and
the average number of falsely selected variables to two indicates that SIS methods tend
to select only the strongest signal and its correlated variables. In contrast, FORD and FOCI exploit joint dependence: at each step, they
account for dependence already explained by previously selected variables, enabling them
to identify additional unexplained signals and more accurately recover the true Markov blanket.   
\begin{table}[ht]
    \centering
    \tiny
    \setlength{\tabcolsep}{3pt} 
    \renewcommand{\arraystretch}{1.2}
    \begin{tabular}{lccccc}
        \toprule
         &  & FORD & FOCI & DCSIS & BCORSIS\\
         Models & n & exact/avg.true./avg.false.  & exact/avg.true./avg.false.  &
         exact/avg.true./avg.false.  &
         exact/avg.true./avg.false. \\
        \midrule
        LM      & 100  & \textbf{0.564}/\textbf{2.552}/\textbf{0.406} & 0.233/2.129/0.796 & 0.395/2.391/0.609 & 0.148/2.105/0.895 \\
        LM      & 500  & \textbf{1.000}/\textbf{3.000}/\textbf{0.000} & 0.996/2.996/0.004 & 0.997/2.997/0.003 & 0.878/2.878/0.122 \\
        LM      & 1000 & \textbf{1.000}/\textbf{3.000}/\textbf{0.000} & \textbf{1.000}/\textbf{3.000}/\textbf{0.000} & \textbf{1.000}/\textbf{3.000}/\textbf{0.000} & 0.995/2.995/0.005 \\
        \midrule
        LM-corr & 100  & 0.003/\textbf{1.173}/\textbf{1.725} & \textbf{0.004}/0.617/2.108 & 0.000/1.000/2.000 & 0.000/1.000/2.000 \\
        LM-corr & 500  & \textbf{0.313}/\textbf{2.313}/\textbf{0.659} & 0.152/2.147/0.840 & 0.000/1.000/2.000 & 0.000/1.000/2.000 \\
        LM-corr & 1000 & \textbf{0.708}/\textbf{2.708}/\textbf{0.275} & 0.412/2.412/0.577 & 0.000/1.000/2.000 & 0.000/1.000/2.000 \\
        \midrule
        Nonlin2 & 100  & \textbf{0.033}/0.293/\textbf{1.818} & 0.006/0.146/1.917 & 0.009/0.687/2.313 & 0.011/\textbf{1.089}/1.911 \\
        Nonlin2 & 500  & 0.285/1.236/1.153 & 0.133/0.619/1.535 & \textbf{0.737}/\textbf{2.613}/\textbf{0.387} & 0.213/2.026/0.974 \\
        Nonlin2 & 1000 & 0.545/2.023/0.615 & 0.270/1.143/1.160 & \textbf{0.941}/\textbf{2.884}/\textbf{0.116} & 0.752/2.747/0.253 \\
        \midrule
        Nonlin2-corr  & 100  & \textbf{0.001}/0.434/\textbf{1.940} & 0.000/0.262/2.134 & 0.000/0.879/2.121 & 0.000/0.069/2.931 \\
        Nonlin2-corr & 500  & 0.016/\textbf{1.819}/\textbf{1.094} & 0.012/1.615/1.203 & \textbf{0.021}/1.210/1.790 & 0.000/0.983/2.017 \\
        Nonlin2-corr & 1000 & 0.030/\textbf{2.027}/\textbf{0.971} & 0.024/1.997/0.992 & \textbf{0.130}/1.808/1.192 & 0.000/1.000/2.000 \\
        \midrule
        Osc2  & 100  & 0.101/0.710/1.485 & 0.029/0.268/1.848 & 0.000/0.190/2.810 & \textbf{0.116}/\textbf{1.772}/\textbf{1.228} \\
        Osc2  & 500  & 0.704/2.424/0.363 & 0.203/0.828/1.308 & 0.037/1.314/1.686 & \textbf{0.997}/\textbf{2.997}/\textbf{0.003} \\
        Osc2  & 1000 & 0.923/2.904/0.085 & 0.332/1.346/1.036 & 0.247/1.944/1.056 & \textbf{1.000}/\textbf{3.000}/\textbf{0.000} \\
        \midrule
        Osc2-corr & 100  & 0.013/0.926/\textbf{1.816} & 0.006/0.726/1.914 & \textbf{0.022}/\textbf{1.178}/1.822 & 0.015/1.158/1.842 \\
        Osc2-corr & 500  & \textbf{0.183}/\textbf{2.183}/\textbf{0.812} & 0.036/2.025/0.975 & 0.032/1.237/1.763 & 0.018/1.227/1.773 \\
        Osc2-corr & 1000 & \textbf{0.422}/\textbf{2.422}/\textbf{0.578} & 0.061/2.061/0.939 & 0.006/1.071/1.929 & 0.003/1.100/1.900 \\
        \bottomrule
    \end{tabular}
    \caption{Variable selection performance across $1000$ Monte Carlo replications. For each method, we report (i) the proportion of exact recovery of the true Markov blanket $\{X_1, X_2, X_3\}$, (ii) the average number of truly selected variables, and (iii) the average number of falsely selected variables. The highest exact recovery rate for each model and sample size is highlighted in bold. Because SIS-based procedures require a fixed model size, the total number of selected variables does not vary; consequently, the inclusion rate and exact recovery rate coincide, and we therefore report the average number of truly selected variables rather than the inclusion rate. Models described in Example~\ref{exVar_screening}.}
    \label{table:corr_simulations}
\end{table}
\end{ex}

\subsection{Real Data Examples}
\label{subsec:realdata}
\begin{ex}\label{exUCI}(variable selection) In this example, we evaluate the performance of FORD on three real-world datasets from the UCI Machine Learning Repository, comparing it with existing approaches such as FOCI \cite{azadkia2021simple} and KFOCI~\cite{deb2022kernelpartialcorrelation} using \texttt{R} package \texttt{KPC}~\cite{KPCpackage} (using the default exponential kernel with median bandwidth and 1-nearest neighbour). For each dataset, we describe the train-test split, explain the variables involved, and provide relevant contextual information.

\begin{enumerate}
    \item \textit{Superconductivity:}  The dataset is randomly split into $70\%$ for training and $30\%$ for testing. It comprises $81$ features extracted from $21263$ superconductors, with the \textit{critical temperature} as the target variable (last column). The remaining covariates capture various chemical and thermodynamic properties of the superconductors, provided in both raw and weighted forms. The weighted features include the weighted mean, geometric mean, entropy, range, and standard deviation of the corresponding properties. The primary objective is to predict the critical temperature based on these features. This dataset was introduced and analysed in \cite{RealSP} and is publicly available from the UCI Machine Learning Repository\footnote{\url{https://archive.ics.uci.edu/dataset/464/superconductivty+data}}.

    \item \textit{Wave Energy Converter:}  
    The dataset is randomly split into $70\%$ for training and $30\%$ for testing. It contains the positions and absorbed power outputs of wave energy converters (WECs) operating under real wave conditions off the southern coast of Australia, near Tasmania. The dataset consists of $72000$ samples and includes 32 features representing the positions of the WECs, denoted as $X_1, X_2, \dots, X_{16}$ and $Y_1, Y_2, \dots, Y_{16}$, along with 16 features corresponding to the absorbed power outputs, denoted as $P_1, P_2, \dots, P_{16}$. The target variable, \textit{Powerall}, represents the total power output of the WEC farm. The goal is to predict the total power output based on the individual positions and power outputs of the converters. This dataset and its applications were discussed in \cite{RealWEC} and are publicly available through the UCI Machine Learning Repository\footnote{\url{https://archive.ics.uci.edu/dataset/494/wave+energy+converters}}.

    \item \textit{Lattice Physics:}  
    The dataset consists of a training set with $23999$ observations and a test set with $359$ observations. Each observation corresponds to a distinct fuel enrichment configuration for a NuScale US600 fuel assembly of type C-01 (NFAC-01). The dataset includes 39 features representing U-235 enrichment levels (ranging from 0.7 to 5.0 weight percent) for fuel rods located within a one-eighth symmetric segment of the assembly. The response variable of interest is the infinite multiplication factor (\textit{k-inf}), calculated using the MCNP6 Monte Carlo simulation code. The objective is to predict \textit{k-inf} based on the enrichment levels of the fuel rods. This dataset was generated and described in \cite{RealLP} and is publicly available through the UCI Machine Learning Repository\footnote{\url{https://archive.ics.uci.edu/dataset/1091/lattice-physics+(pwr+fuel+assembly+neutronics+simulation+results}}.
\end{enumerate}

\begin{table}[ht]
\centering
\small
\begin{tabular}{lccccccc}
\toprule
 & \multicolumn{2}{c}{Superconductivity} & \multicolumn{2}{c}{Wave Energy Converter} & \multicolumn{2}{c}{Lattice Physics} \\
\cmidrule(r){2-3} \cmidrule(r){4-5} \cmidrule(r){6-7} 
& Subset size & MSPE & Subset size & MSPE & Subset size & MSPE \\
\midrule
FOCI   & 8  & $106.27$  & 31   & $1.76\times10^9$  & 20  & $1.53\times10^{-4}$  \\
KFOCI  & 11  & $106.53$ & 28  & $5.18\times10^9$  & 6  & $1.53\times10^{-4}$ \\
FORD   & 15  & $\mathbf{97.92}$  & 28  & $\mathbf{1.75\times10^9}$  & 20  & $\mathbf{1.51\times10^{-4}}$ \\
\midrule
Random Forest & - & $92.72$ & - & $2.02\times10^9$ & - & $1.54\times10^{-4}$\\
\bottomrule
\end{tabular}
\caption{Performance comparison of FORD, KFOCI, and FOCI on three datasets, using the MSPE of a random forest fitted with the variables selected by each method. Data described in Example~\ref{exUCI}.}
\label{tab:real1}
\end{table}

For each dataset, we compared the performance of FORD with two competing methods: FOCI and KFOCI (the latter using the default exponential kernel with median bandwidth and 1-nearest neighbour). Following variable selection via each method’s respective stopping rule, the selected subsets were used to train predictive models on the training data using random forests implemented in the \texttt{randomForest} package~\cite{rfpackage} in \texttt{R}. Mean squared prediction errors (MSPEs) were then estimated on the test set. Table~\ref{tab:real1} reports the sizes of the selected subsets along with their corresponding MSPEs. The final row of Table~\ref{tab:real1} shows the performance of a random forest model trained on the full set of variables. In all cases, FORD achieved prediction accuracy comparable to that of FOCI and KFOCI; only in the \textit{Superconductivity} dataset with the full model yield a lower MSPE.
\begin{figure}[ht]
  \centering
  \includegraphics[width=0.7\linewidth]{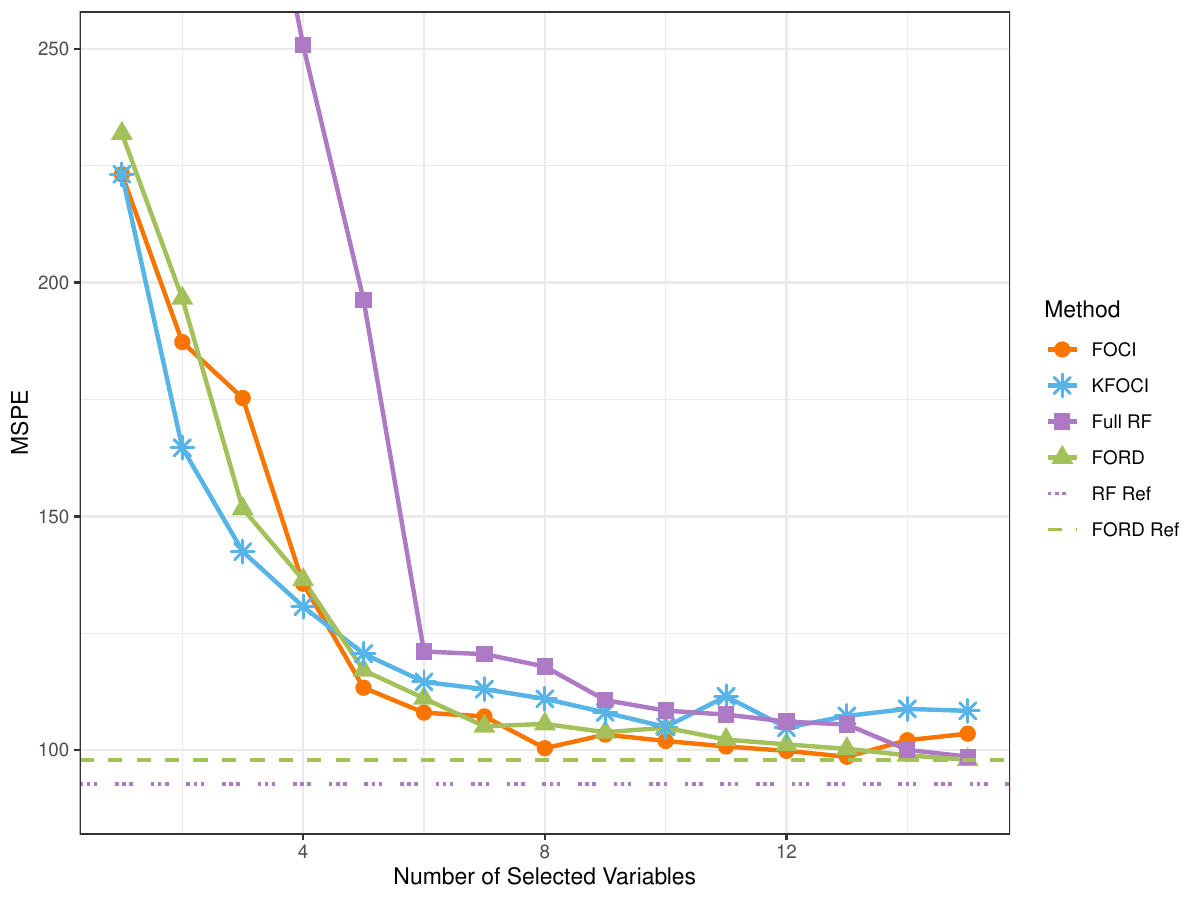}
\caption{Comparison of MSPE as a function of the number of selected variables on the Superconductivity dataset, using variable selection methods FOCI, FORD, and KFOCI, each followed by a random forest trained on the selected variables. The \textit{Full RF} curve represents a random forest model trained on the top-$k$ variables ($k \in \{1, \ldots, 15\}$) ranked by variable importance from a random forest using all features. Dashed and dotted horizontal lines indicate the baseline MSPEs for the initial FORD model and the full random forest model (using all variables), respectively. The results illustrate the advantage of targeted variable selection in reducing model complexity while maintaining or improving predictive performance. Data described in Example~\ref{exUCI}.}
\label{fig:rf_study}
\end{figure}
Since each of these variable selection methods results in a set with possibly different sizes, we compare the performance of the ordered subsets by comparing the MSPE of the fitted random forest on the first $k$ selected variables for $k\in\{1, \ldots, 15\}$. Figure~\ref{fig:rf_study} shows the MSPEs for all these models. 
\end{ex}

\subsubsection{Comparison to Chatterjee's Correlation Coefficient}\label{sec:comparision_with_xi}
To further explain the distinction between the measures $\nu$ and $T$, it is instructive to compare their respective estimators $\nu_n^{\text{1-dim}}$ and $\xi_n$. Suppose there are no ties among the sample observations $X_i$ and $Y_i$. Under this assumption, the estimators can be expressed as
\begin{align*}
   \nu_n^{\text{1-dim}}(Y, X) 
   &= 1 - \sum_{i=1}^{n-1} \sum_{\substack{j \ne i,\, i+1 \\ r_j \ne 1,\, n}} w_{\nu_{n}^{\text{1-dim}}, j} \bone \{r_j \in \mathcal{K}_i\},\\
   \xi_n(Y, X) 
   &= 1 - \sum_{i=1}^{n-1} \sum_{j \ne i} w_{\xi_n} \bone\{r_j \in \mathcal{K}_i\},
\end{align*}
where the weights are given by $w_{\nu_{n}^{\text{1-dim}}, j} = 1/\{2(r_j - 1)(n - r_j)\}$ and $w_{\xi_n} = 3/(n^2 - 1)$. This formulation emphasizes the fundamental distinction in how the two statistics assign weight to rank oscillations.

For $n \geq 5$, the inequality $w_{\xi_n} \geq w_{\nu_{n}^{\text{1-dim}}, j}$ holds precisely when
\begin{align*}
    r_j \in L_n := \left[\frac{n+1 - \sqrt{(n-1)(n-5)/3}}{2}, \frac{n+1 + \sqrt{(n-1)(n-5)/3}}{2}\right],
\end{align*}
and the $w_{\xi_n} < w_{\nu_{n}^{\text{1-dim}}, j}$ otherwise. Thus, for any rank oscillation interval $\mathcal{K}_i$ containing $r_j \in L_n$, the statistic $\xi_n$ imposes a greater penalty—interpreted in terms of deviation from independence-than does $\nu_n^{\text{1-dim}}$. In general, the weight ratio satisfies
\begin{align*}
    \frac{w_{\nu_{n}^{\text{1-dim}}, j}}{w_{\xi_n}} \geq \frac{2}{3}.
\end{align*}
However, this ratio does not admit a uniform upper bound; instead, its maximal value grows asymptotically as $n/6$. Consequently, when $r_j \notin L_n$, the estimator $\nu_n^{\text{1-dim}}$ penalises the corresponding rank oscillation more heavily than $\xi_n$, with the disparity increasing with the sample size $n$.

In the following example, we consider the Yeast gene expression data analyzed in~\cite{chatterjee2021new} and examine how this difference manifests in the identification of genes with oscillating transcript levels over time.
\begin{ex}\label{exYeast}(yeast gene expression data)
We follow the Yeast gene expression example in~\cite{chatterjee2021new} and investigate the effectiveness of $\nu_{n}^{\text{1-dim}}(Y, X)$ in identifying genes with oscillating transcript levels over time. Specifically, we apply it to the curated \texttt{Spellman} dataset available in the \texttt{R} package \texttt{minerva}, which contains gene expression data for $4381$ transcripts measured at $23$ time points. In this context, $Y$ denotes the transcript level of a gene, while $X$ represents the time of recording. 
\begin{figure}[ht]
  \centering
    \begin{subfigure}[b]{0.49\textwidth}
    \includegraphics[width=\linewidth]{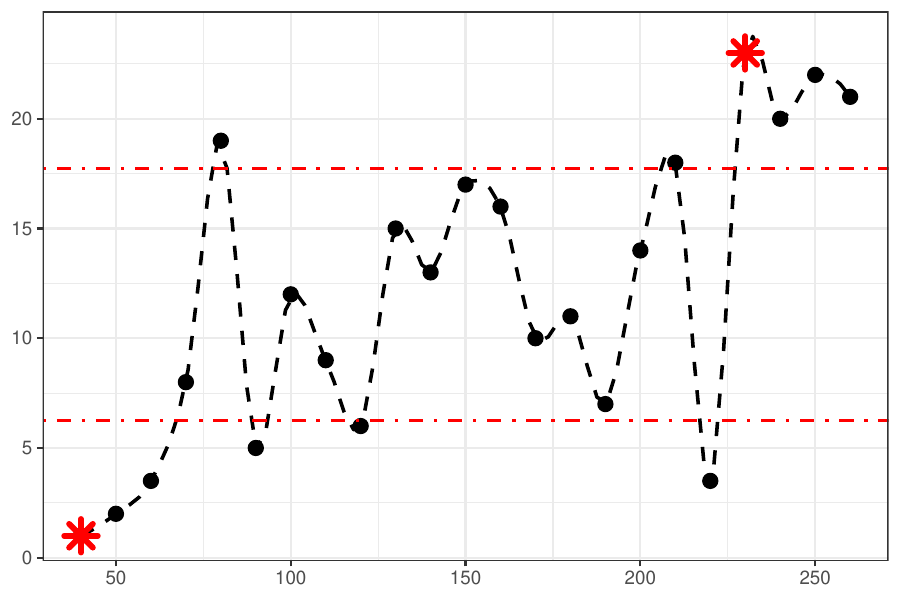}
    \caption*{\scalebox{0.99}{\tiny YDR148C: $q_{\xi_n} = 0.0554,\ q_{\nu_n^{\text{1-dim}}} = 0.0140$}}
  \end{subfigure}
  \hfill
  \begin{subfigure}[b]{0.49\textwidth}
    \includegraphics[width=\linewidth]{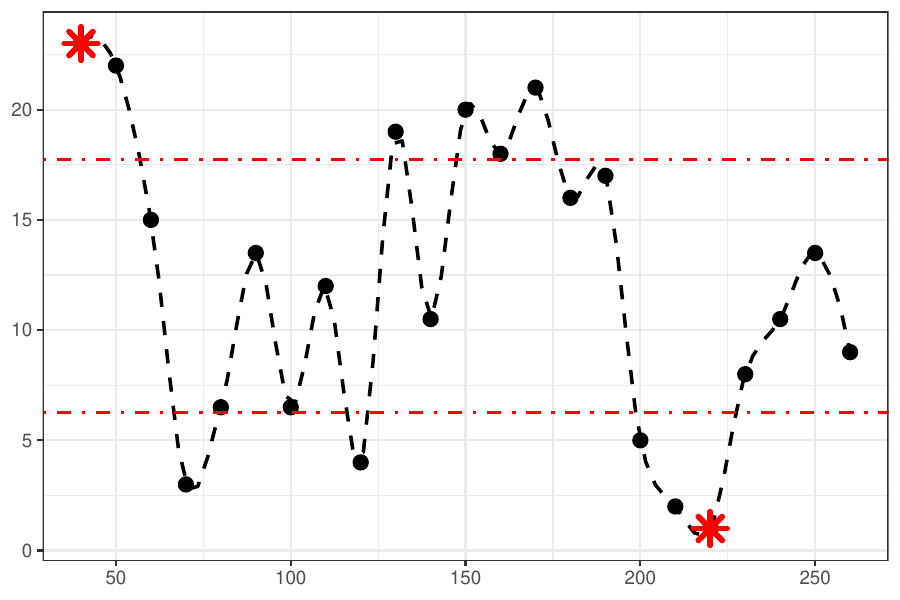}
    \caption*{\scalebox{0.99}{\tiny YOR038C: $q_{\xi_n} = 0.0810,\ q_{\nu_n^{\text{1-dim}}} = 0.0178$}}
  \end{subfigure}
  
   \vspace{1cm}
   
    \begin{subfigure}[b]{0.49\textwidth}
    \includegraphics[width=\linewidth]{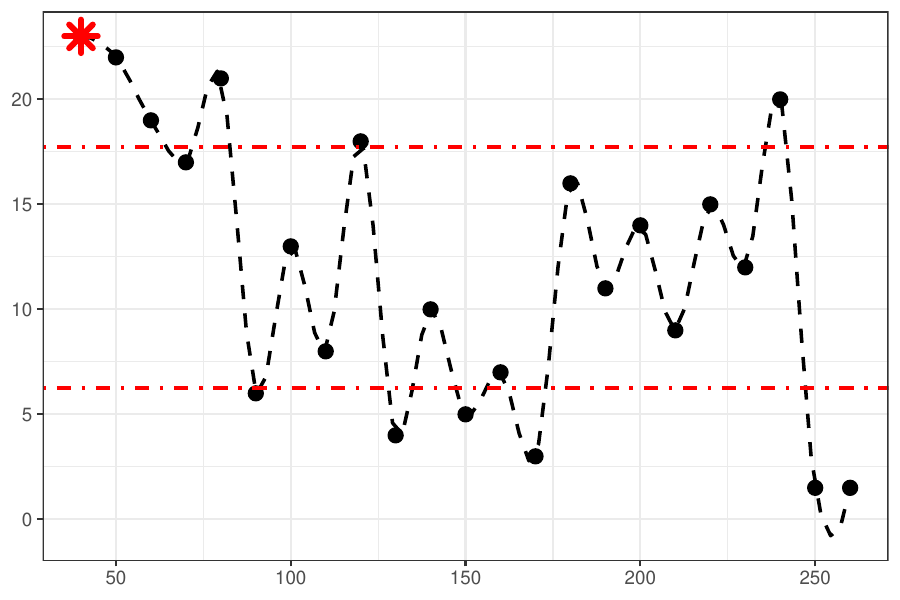}
    \caption*{\scalebox{0.99}{\tiny YHR216W: $q_{\xi_n} = 0.2106,\ q_{\nu_n^{\text{1-dim}}} = 0.0312$}}
  \end{subfigure}
  \hfill
    \begin{subfigure}[b]{0.49\textwidth}
    \includegraphics[width=\linewidth]{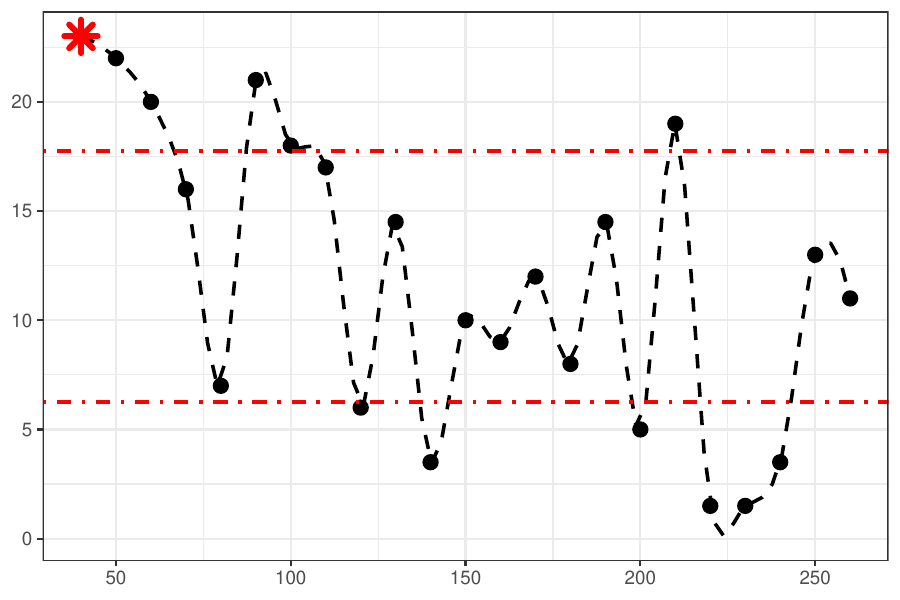}
    \caption*{\scalebox{0.99}{\tiny YKL144C: $q_{\xi_n} = 0.2251,\ q_{\nu_n^{\text{1-dim}}} = 0.0465$}}
  \end{subfigure}
  \vspace{0.5cm}
\caption{
Plots of four genes detected by $\nu_{n}^{\text{1-dim}}$ but not by $\xi_n$, the first row figures are selected based on the smallest q-values under $\nu_{n}^{\text{1-dim}}$ and the second row figures are selected based on the largest q-values under $\xi_n$. The vertical axis shows the gene expression ranks, and the horizontal axis represents time. Ranks 1 and 23 are marked with red stars. The region between the two horizontal red dot-dashed lines indicates where $w_{\xi_n}$ exceeds $w_{\nu_{n}^{\text{1-dim}}, j}$. A LOESS regression curve (black dashed line) is overlaid using a smoothing parameter of 0.2.
}
\label{fig:nu-only-genes}
\end{figure}
\begin{figure}[ht]
  \centering
  
  \begin{subfigure}[b]{0.49\textwidth}
    \includegraphics[width=\linewidth]{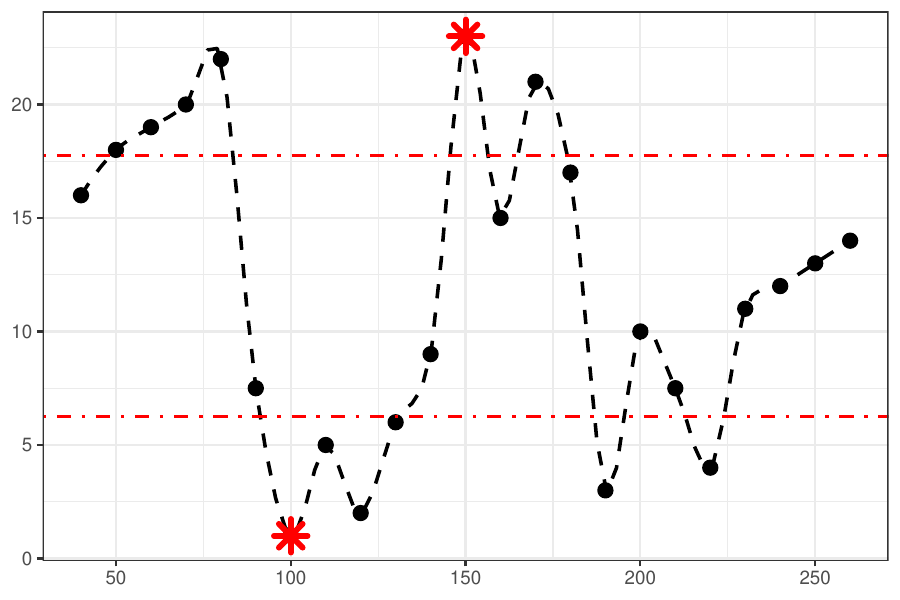}
    \caption*{\scalebox{0.99}{\tiny YKL001C: $q_{\xi_n} = 0.0138,\ q_{\nu_n^{\text{1-dim}}} = 0.0598$}}
  \end{subfigure}
  \hfill
  \begin{subfigure}[b]{0.49\textwidth}
    \includegraphics[width=\linewidth]{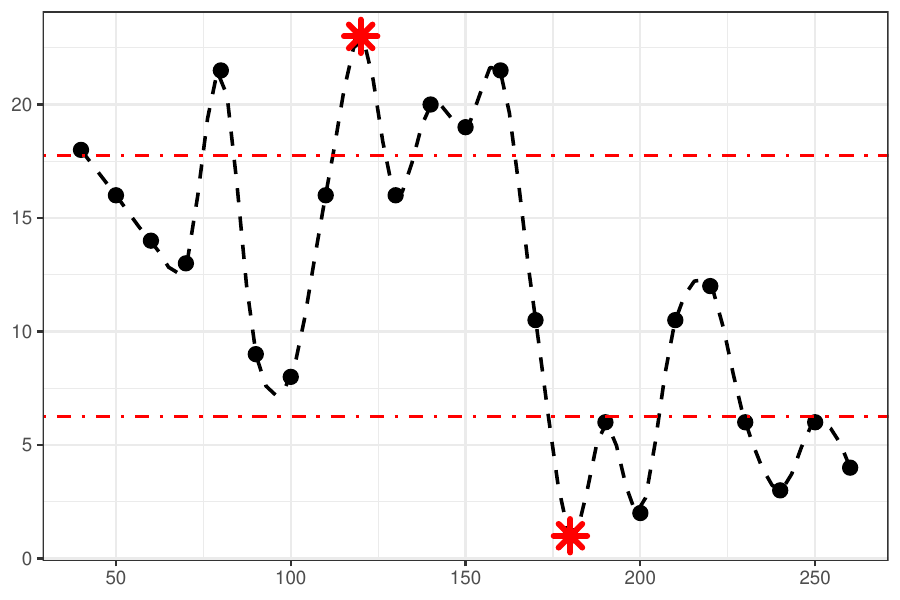}
    \caption*{\scalebox{0.99}{\tiny YGL063W: $q_{\xi_n} = 0.0152,\ q_{\nu_n^{\text{1-dim}}} = 0.0718$}}
  \end{subfigure}
  
  \vspace{1cm}
  
  \begin{subfigure}[b]{0.49\textwidth}
    \includegraphics[width=\linewidth]{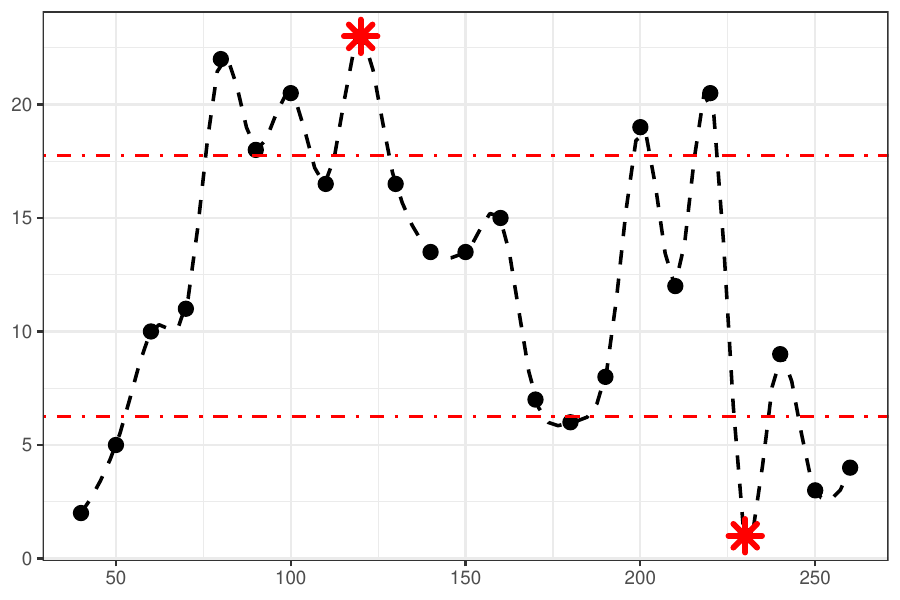}
    \caption*{\scalebox{0.99}{\tiny YKL056C: $q_{\xi_n} = 0.0423,\ q_{\nu_n^{\text{1-dim}}} = 0.1105$}}
  \end{subfigure}
  \hfill
   \begin{subfigure}[b]{0.49\textwidth}
    \includegraphics[width=\linewidth]{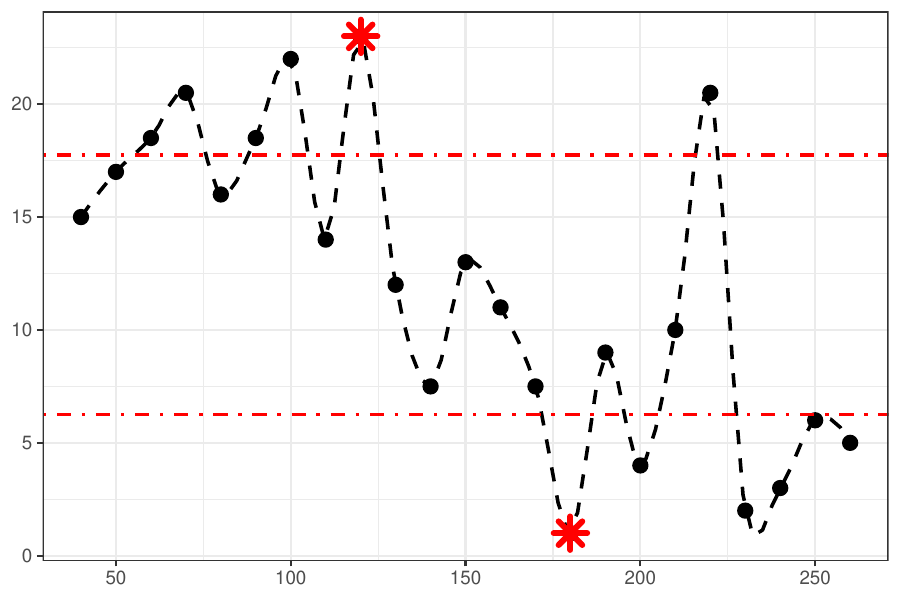}
    \caption*{\scalebox{0.99}{\tiny YDR483W: $q_{\xi_n} = 0.0408,\ q_{\nu_n^{\text{1-dim}}} = 0.1240$}}
  \end{subfigure}
  \vspace{0.5cm}
  \caption{
Plots of four genes detected by $\xi_n$ but not by $\nu_{n}^{\text{1-dim}}$, the first row figures are selected based on the smallest q-values under $\xi_n$ and the second row figures are selected based on the largest q-values under $\nu_{n}^{\text{1-dim}}$. The vertical axis represents gene expression ranks, and the horizontal axis represents time. Ranks 1 and 23 are marked with red stars. The region between the two horizontal red dot-dashed lines indicates where $w_{\xi_n}$ exceeds $w_{\nu_{n}^{\text{1-dim}}, j}$. A LOESS regression curve (black dashed line) is fitted using a smoothing parameter of 0.2.
}
\label{fig:xi-only-genes}
\end{figure}

To identify the genes whose transcript levels exhibit oscillatory patterns, we conduct a permutation test on the dependence measures $\nu_{n}^{\text{1-dim}}$ and $\xi_n$ using $10000$ replications. Genes with significantly large values of these dependence measures are identified as having time-dependent expression patterns, as determined by an independence-based permutation test. For both statistics, p-values are computed and the Benjamini–Hochberg procedure~\cite{benjamini1995controlling} is applied to control the false discovery rate (FDR) at the $0.05$ level. We refer to the adjusted p-values using Benjamini–Hochberg procedure as q-values.

As a result, out of $4381$ genes, $685$ are found to be significant using $\nu_{n}^{\text{1-dim}}$. Among these, $78$ genes are uniquely detected by $\nu_{n}^{\text{1-dim}}$ and not by $\xi_n$. Conversely, $\xi_n$ detects $679$ significant genes, of which $72$ are not detected by $\nu_{n}^{\text{1-dim}}$. This slight discrepancy suggests that $\nu_{n}^{\text{1-dim}}$ may have an edge in identifying certain types of dependence patterns.

Figure~\ref{fig:nu-only-genes} illustrates four gene expression patterns exclusively detected by $\nu_{n}^{\text{1-dim}}$. Specifically, the first row of Figure~\ref{fig:nu-only-genes} presents the two genes with the smallest q-values under $\nu_{n}^{\text{1-dim}}$ among those not identified by $\xi_n$, highlighting cases where $\nu_{n}^{\text{1-dim}}$ shows strong confidence in detection. The second row of Figure~\ref{fig:nu-only-genes} displays two genes selected by $\nu_{n}^{\text{1-dim}}$ but not by $\xi_n$, which exhibit the largest q-values under $\xi_n$. Both figures support the observation that when oscillations occur around mid-range rank values—where $w_{\xi_n} \geq w_{\nu_{n}^{\text{1-dim}}, j}$—$\nu_{n}^{\text{1-dim}}$ is more effective at capturing dependencies than $\xi_n$.

On the other hand, Figure~\ref{fig:xi-only-genes} displays gene expression patterns detected by $\xi_n$ but not by $\nu_{n}^{\text{1-dim}}$. The first row of Figure~\ref{fig:xi-only-genes} presents the two genes with the smallest q-values under $\xi_n$ among those not identified by $\nu_{n}^{\text{1-dim}}$, highlighting cases where $\xi_n$ showed strong confidence in selection. The second row of Figure~\ref{fig:xi-only-genes} shows two genes selected by $\xi_n$ and not by $\nu_{n}^{\text{1-dim}}$ that have the largest q-values under $\nu_{n}^{\text{1-dim}}$. 

In conclusion, it seems $\nu_{n}^{\text{1-dim}}$ excels at detecting smooth, mid-rank oscillatory patterns, whereas $\xi_n$ is more sensitive to sharp transitions at the extremes. Independence testing using the respective asymptotic distributions—established for $\xi_n$ and conjectured for $\nu_{n}^{\text{1-dim}}$—further supports the advantage of $\nu_{n}^{\text{1-dim}}$, which identified 677 genes compared to 586 by $\xi_n$. Among these 586 genes, only 39 were not detected by $\nu_{n}^{\text{1-dim}}$.
\end{ex}

\clearpage
\section*{Acknowledgement}
We are grateful to Sourav Chatterjee and Rina Foygel Barber for helpful comments. Part of this work was conducted during M.A.'s visit to the Institute for Mathematical and Statistical Innovation (IMSI), which is supported by the National Science Foundation under Grant No. DMS-1929348.

\section{Proofs}\label{sec:proof}
\subsection{Proof of Theorem~\ref{thm:NuProperties}}

\begin{proof}
    Remember that $S$ is the support of $\mu$ and we define $\tilde{\mu}$, the modified version of $\mu$, in the following way: If $S$ attains a maximum $s_{\max}$, let $\tilde{S} = S\setminus \{s_{\max}\}$ otherwise let $\tilde{S} = S$, and for any measurable set $A\subseteq S$, let $\tilde{\mu}(A) = \mu(A\cap\tilde{S}\})/\mu(\tilde{S})$. In addition, for simplicity in notation, since $\var(\ee[\bone\{Y > t\}\mid\bbx]) = 0$ whenever $\var(\bone\{Y > t\}) = 0$ we define $\var(\ee[\bone\{Y > t\}\mid\bbx])/\var(\bone\{Y > t\})$ to be equal to 1.
    
    Assuming that $Y$ is not almost surely a constant guarantee that for almost all values of $t$ with respect to $\tilde{\mu}$, $\var(\bone\{Y > t\})$ is non-zero and hence $\nu(Y, \bbx)$ is well-defined. Note that by the law of total variance and non-negativity of variance, we have
    \[
    0\leq \var(\ee[\bone\{Y > t\}\mid \bbx]) \leq \var(\bone\{Y > t\}),
    \]
    which gives $\nu(Y, \bbx) \in [0, 1]$. 

    When $Y$ is independent of $\bbx$ for all $t\in\rr$ we have 
    \[
    \ee[\bone\{Y > t\}\mid \bbx] = \ee[\bone\{Y > t\}],
    \]
    therefore $\var(\ee[\bone\{Y > t\}\mid \bbx]) = 0$ which gives $\nu(Y, \bbx) = 0$. 
    
    For each $t$ let $G(t) := \pp(Y > t)$, and $G_\bbx(t) := \pp(Y > t\mid \bbx)$. Note that $\nu(Y, \bbx) = 0$ implies that there exists a Borel set $A\subseteq \rr$ such that $\tilde{\mu}(A) = 1$ and for any $t\in A$, $\var(G_\bbx(t)) = 0$. This implies that for $t\in A$, $G_\bbx(t) = G(t)$ almost surely with respect to $\tilde{\mu}$. We claim that $A = \rr$. 

    Take any $t\in\rr$. If $\tilde{\mu}(\{t\}) > 0$, then $t\in A$. So w.l.o.g assume that $\tilde{\mu}(\{t\}) = 0$. Note that this also implies $\mu(t) = 0$, unless $t = s_{\max}$. We also have $\var(G(s_{\max})) = \var(G_\bbx(s_{\max})) = 0$ which implies $s_{\max}\in A$. Therefore, for any other such $t$, $\mu(t) = 0$. This implies that $G$ is right-continuous at $t$.

    Suppose for all $s > t$ we have $G(s) < G(t)$. Then for each $s > t$, $\mu([t, s)) > 0$ and hence $A\cap[t, s)\neq \emptyset$. Therefore, there exists a sequence $r_n \in A$ such that $r_n \downarrow t$. Since $r_n\in A$, we have $G_\bbx(r_n) = G(r_n)$ almost surely for all $n$. Therefore with probability $1$ we have 
    \begin{eqnarray*}
        G_\bbx(t) \geq \lim_{n\rightarrow\infty} G_\bbx(r_n) = \lim_{n\rightarrow\infty} G(r_n) = G(t)
    \end{eqnarray*}
    because of the right-continuity of $G$. Note that $\ee[G_\bbx(t)] = G(t)$, hence this implies $G_\bbx(t) = G(t)$ almost surely and therefore $t\in A$. 

    Suppose there exist $s > t$ such that $G(s) = G(t)$. Take the largest such $s$, which exists because $G$ is left-continuous. If $s = \infty$, then $G(t) = G(s) = 0$. Since $\ee[G_\bbx(t)] = G(t) = 0$ this implies $G_\bbx(t) = G(t) = 0$ almost surely which implies $t\in A$. So assume $s < \infty$. Either $\mu(\{s\}) > 0$, which  implies $G_\bbx(s) = G(s)$ almost surely, or $\mu(\{s\}) = 0$ and $G(r) < G(s)$ for all $r > s$, which again implies $G_\bbx(s) = G(s)$ almost surely as in the previous paragraph. Therefore, in either case, with probability $1$, we have 
    \begin{eqnarray*}
        G_\bbx(t) \geq G_\bbx(s) = G(s) = G(t).
    \end{eqnarray*}
    Since $\ee[G_\bbx(t)] = G(t)$, this implies $G_\bbx(t) = G(t)$ almost surely. Therefore $t\in A$. This shows we can take $A$ as big as $\rr$. 

    Now, for an arbitrary Borel set $B\subseteq \rr$,
    \begin{eqnarray*}
        \pp(\{Y > t\}\cap\{\bbx\in B\}) &=& \ee[\ee[\bone\{Y > t\}\mid \bbx]\bone\{X\in B\}] \\
        &=& \ee[G_\bbx(t)\bone\{\bbx\in B\}] \\
        &=& \ee[G(t)\bone\{\bbx\in B\}] \\
        &=& G(t)\pp(\bbx\in B) \\
        &=& \pp(Y > t)\pp(\bbx\in B).
    \end{eqnarray*}
    This proves that $Y$ and $\bbx$ are independent.

    Assume there exists a measurable function $f:\rr^p\rightarrow\rr$ such that $Y = f(\bbx)$. This implies that for all $t\in\rr$, $\ee[\bone\{Y > t\}\mid \bbx] = \bone\{Y > t\}$ and therefore 
    \[
    \var(\ee[\bone\{Y > t\}\mid \bbx]) = \var(\bone\{Y > t\}).
    \]
    This gives $\nu(Y, \bbx) = 1$. On the other hand, assume $\nu(Y, \bbx) = 1$. This implies for almost all $t\in\rr$ w.r.t $\tilde{\mu}$ we have 
    \[
    \var(G_\bbx(t)) = \var(\bone\{Y > t\}).
    \]
    If $S$, the support of $\mu$ attains the minimum $s_{\max}$, then note that we also have
    \[
    \var(G_\bbx(s_{\max})) = \var(\bone\{Y > s_{\max}\}).
    \]
    This implies $\ee[\var(\bone\{Y > t\}\mid \bbx)] = \ee[G_\bbx(t)(1 - G_\bbx(t))] = 0$ for almost all $t$ with respect to $\mu$. Therefore, $G_\bbx(t)$ almost surely takes only the values of $0$ and $1$ with respect to $\mu$. Let $E$ ($\bbx$-measurable) the event that $G_\bbx(t) \in\{0, 1\}$ for almost all values of $t$ and note that $\pp(E) = 1$. Let $a_\bbx$ be the largest value such that $G_\bbx(a_\bbx) = 1$ and $b_\bbx$ be the smallest value such that $G_\bbx(b_\bbx) = 0$. Note that $a_\bbx\leq b_\bbx$. Suppose $\{a_\bbx < b_\bbx\}\cap E$ happens. This means that for all $t\in (a_\bbx, b_\bbx)$ we have $G_\bbx(t)\in (0, 1)$ therefore $\mu((a_\bbx, b_\bbx)) = 0$.
    Then we have $\pp(Y\in (a_\bbx, b_\bbx)\mid \bbx) = 0$ which implies event $\{a_\bbx < b_\bbx\}\cap E$ is of measure $0$ and hence $a_\bbx = b_\bbx$ almost surely. Then this gives us $Y = a_\bbx$ almost surely, which completes the proof. 
\end{proof}

\subsection{Proof of Theorem~\ref{thm:conditional}}
\begin{proof}
    If $Y$ is not almost surely equal to a measurable function of $\bbz$, Theorem~\ref{thm:NuProperties} gives us $\nu(Y, \bbz) < 1$, using this and the fact that by Theorem~\ref{thm:NuProperties} $\nu(Y, (\bbx, \bbz))$ and $\nu(Y, \bbz)$ are well-defined, $\nu(Y, \bbx\mid \bbz)$ is well-defined. Additionally 
    \[
    \nu(Y, (\bbx, \bbz)) - \nu(Y, \bbz) \leq 1 - \nu(Y, \bbz),
    \]
    and hence $\nu(Y, (\bbx, \bbz))\in[0, 1]$.

    Note that $\nu(Y, \bbx\mid \bbz) = 1$ if and only if $\nu(Y, (\bbx, \bbz)) = 1$ which happens if and only if $Y$ is a measurable function of $(\bbx, \bbz)$ which is equivalent to $Y$ being a measurable function of $\bbx$ given $\bbz$.

    Finally $\nu(Y, \bbx\mid \bbz) = 0$ if and only if $\nu(Y, (\bbx, \bbz)) = \nu(Y, \bbz)$. Note that 
    \begin{align*}
        \var(\ee[\bone\{Y > t\}\mid \bbx, \bbz]) = \var(\ee[\bone\{Y > t\}\mid \bbz]) + \ee[\var(\ee[\bone\{Y > t\}\mid\bbx, \bbz]\mid \bbz)]
    \end{align*}
    Since $\nu(Y, (\bbx, \bbz)) \geq \nu(Y, \bbz)$, equality happens if and only if for $\tilde{\mu}$ almost every $t$ we have 
    \begin{align*}
        \var(\ee[\bone\{Y > t\}\mid \bbx, \bbz]) = \var(\ee[\bone\{Y > t\}\mid \bbz]).
    \end{align*}
    Putting these together means $\ee[\var(\ee[\bone\{Y > t\}\mid\bbx, \bbz]\mid \bbz)] = 0$ for $\tilde{\mu}$ almost every $t$ which means $\var(\ee[\bone\{Y > t\}\mid\bbx, \bbz]\mid \bbz) = 0$ almost surely thus
    \[
        \ee[\bone\{Y > t\}\mid\bbx, \bbz] = \ee[\bone\{Y > t\}\mid \bbz],
    \]
    and hence $Y$ is independent of $\bbx$ given $\bbz$.
\end{proof}

\subsection{Proof of Theorem~\ref{thm:consistency}}
For more clarity in the notation of our proof, we rewrite the estimator $\nu_n$ in terms of the empirical cumulative function. Let 
\begin{eqnarray*}
    \mathcal{I}_i^j := [\min\{Y_i, Y_{N^{-j}(i)}\}, \max\{Y_i, Y_{N^{-j}(i)}\}].
\end{eqnarray*}
For each $j\in[n]$ and $t\in\rr$ let
\begin{eqnarray*}
    F_{n,j}(t) := (n - 1)^{-1}\sum_{k\neq j}\bone\{Y_k \leq t\}, \qquad F_{n}(t) := n^{-1}\sum_{k=1}^n\bone\{Y_k \leq t\}.
\end{eqnarray*}
Note that 
\begin{eqnarray*}
    R_j = n F_n(Y_j), \qquad F_{n,j}(Y_j) = \big(\frac{n}{n - 1}\big)F_n(Y_j) - \frac{1}{n-1} = \frac{R_j - 1}{n - 1}.
\end{eqnarray*}
Using these, we can rewrite $\nu_n(Y, \bbx)$ as
\begin{eqnarray*}
    \nu_n(Y, \bbx) = 1 - \frac{1}{2(n-1)(n - n_0)}\sum_{j=1}^n\sum_{i\neq j}\frac{\bone\{Y_j\in\mathcal{I}_i^j\}\bone\{F_n(Y_j)\neq 1, 1/n\}}{F_{n,j}(Y_j)(1 - F_{n,j}(Y_j))},
\end{eqnarray*}
where $n_0 = n_{\max} + c_{\min}$, with $n_{\max}$ and $c_{\min}$ defined as before: $n_{\max}$ number of $Y_j$'s that are equal to the maximum of $Y_i$'s and $c_{\min} = 1$ if $Y_j$'s minimum is unique and zero otherwise.  

\begin{proof}
Let 
\begin{align}
    Q_n &:= \frac{1}{2(n - 1)(n - n_0)}\sum_{j = 1}^n\sum_{i\neq j}\frac{\bone\{Y_j \in \mathcal{I}_i^j\}\bone\{F_{n}(Y_j)\neq 1, 1/n\}}{F_{n,j}(
    Y_j)(1 - F_{n,j}(Y_j))}, \label{eq:Qn}\\
    Q_n^\prime &:= \frac{1}{2(n - 1)(n - n_0)}\sum_{j = 1}^n\sum_{i\neq j}\frac{\bone\{Y_j \in \mathcal{I}_i^j\}\bone\{F_{n}(Y_j)\neq 1, 1/n\}}{F(
    Y_j)(1 - F(Y_j))}, \label{eq:Qpn}\\
    Q &:= \int\frac{\ee[F_\bbx(t)(1 - F_\bbx(t))]}{F(t)(1 - F(t))}d\tilde{\mu}(t). \label{eq:Q}
\end{align}

\begin{lmm}\label{lmm1} With $Q_n$ and $Q$ defined in \eqref{eq:Qn} and \eqref{eq:Q}
\begin{eqnarray*}
    \lim_{n\rightarrow\infty}\ee[Q_n] = Q.
\end{eqnarray*}
\end{lmm}
\begin{proof}
To prove the convergence of $\ee[Q_n]$ to $Q$, we divide the argument into two steps: first, we show that $\ee[\abs{Q_n - Q_n'}]$ converges to zero; second, we show that $\ee[Q_n']$ converges to $Q$.

\noindent\textbf{Step I.} In this step we show that $\ee[\abs{Q_n - Q_n^\prime}]$ converges to zero.
    \begin{eqnarray*}
        \lefteqn{\ee[\big(\frac{n - n_0}{n}\big)\abs{Q_n - Q_n^\prime}] }\\
        &\leq& \frac{1}{2}\ee[\frac{\ee[\bone\{Y_j\in\mathcal{I}_i^j\}\mid F_n(Y_j), F(Y_j)]\abs{F_{n,j}(Y_j) - F(Y_j)}\bone\{F_n(Y_j)\neq 1, n^{-1}\}}{\max\{F_{n,j}(Y_j)(1 - F_{n,j}(Y_j)), \frac{n-1}{n^2}\}F(Y_j)(1 - F(Y_j))}] \\
        &\leq & \ee[\frac{\abs{F_{n, j}(Y_j) - F(Y_j)}}{F(Y_j)(1 - F(Y_j))}].
    \end{eqnarray*}
    Note that
    \begin{eqnarray*}
        \ee[\frac{\abs{F_{n, j}(Y_j) - F(Y_j)}}{F(Y_j)(1 - F(Y_j))}] = \int_{t\in\rr}\frac{\ee[\abs{F_{n - 1}(t) - F(t)}]}{F(t)(1 - F(t))}d\mu(t).
    \end{eqnarray*}
    Using Theorem 1.2 of \cite{bartl2023variance}, there exists absolute constants $c_0$ and $c_1$ such that for every $\Delta \geq c_0\log\log m/m$ with probability at least $1 - \exp(-c_1\Delta m)$, for every $t$ such that $\Delta \leq F(t)(1 - F(t))$ we have 
    \begin{eqnarray*}
        \abs{F_m(t) - F(t)} \leq \sqrt{F(t)(1 - F(t))\Delta}.
    \end{eqnarray*}
Let $\delta = (1-\sqrt{1-4\Delta})/2$. Using this and symmetry, we have 
\begin{align}
    &\int_{t\in\rr}\frac{\ee[\abs{F_{n - 1}(t) - F(t)}]}{F(t)(1 - F(t))}d\mu(t) \nonumber\\
    &\leq 2 \int_{\delta \leq F(t)\leq 0.5}\frac{\ee[\abs{F_{n - 1}(t) - F(t)}]}{F(t)(1 - F(t))}d\mu(t) + 2\int_{F(t) < \delta}\frac{\ee[\abs{F_{n - 1}(t) - F(t)}]}{F(t)(1 - F(t))}d\mu(t) \nonumber\\
    &\leq 2\int_{\delta \leq F(t)\leq 0.5}\frac{\sqrt{\Delta F(t)(1 - F(t))}(1 - 2\exp(-c_1(n-1)\Delta)) }{F(t)(1 - F(t))} d\mu(t) + \nonumber\\
    &2\int_{\delta \leq F(t)\leq 0.5}\frac{2\exp(-c_1(n-1)\Delta)}{F(t)(1 - F(t))}d\mu(t) + 2\int_{F(t) < \delta}\frac{\ee[F_{n-1}(t)] + F(t)}{F(t)(1 - F(t))} d\mu(t) \nonumber\\
    &\leq \pi\sqrt{\Delta} + 4\exp(-c_1(n-1)\Delta)\log\big(\frac{1 + \sqrt{1- 4\Delta}}{1 - \sqrt{1- 4\Delta}}\big). \label{eqn:Delta}
\end{align}
Now let $\Delta = c_1^{-1}\log(n)/(n - 1)$. Then, as $n$ goes to infinity, \eqref{eqn:Delta} goes to zero. Hence $\ee[\big(\frac{n - n_0}{n}\big)\abs{Q_n - Q_n^\prime}]$ converges to zero. Since $Y$ is not almost surely a constant, as $n$ grows to $\infty$, $(n - n_0)/n$ converges to constant $\mu(\tilde{S}) > 0$. For large enough $n$ we have $(n - n_0)/n > \mu(\tilde{S})/2$. Therefore, for large enough $n$ we have
\begin{eqnarray*}
    \ee[\abs{Q_n - Q_n^\prime}] \leq \frac{2}{\mu(\tilde{S})}\ee[\big(\frac{n - n_0}{n}\big)\abs{Q_n - Q_n^\prime}].
\end{eqnarray*}
Since the right-hand side of the above inequality converges to zero, we conclude that $\lim_{n\rightarrow\infty}\ee[\abs{Q_n - Q_n^\prime}] = 0$.

\noindent\textbf{Step II.} In this step we show that $\ee[Q_n^\prime]$ converges to $Q$.
\begin{eqnarray*}
    \ee[\big(\frac{n - n_0}{n}\big)Q_n^\prime] &=& \frac{1}{2}\ee[\frac{\bone\{Y_j\in\mathcal{I}_i^j\}\bone\{F_n(Y_j)\neq 1, 1/n\}}{F(Y_j)(1 - F(Y_j))}].
\end{eqnarray*}
First, let's study the case when $\mu$ is continuous. In this case, by conditioning on the value of $F_n(Y_j)$, we have 
\begin{eqnarray*}
    \lefteqn{\ee[\frac{\bone\{Y_j\in\mathcal{I}_i^j\}\bone\{F_n(Y_j)\neq 1, 1/n\}}{F(Y_j)(1 - F(Y_j))}]}\\
    &=& \frac{1}{n}\sum_{r = 1}^n \ee[\frac{\bone\{Y_j\in\mathcal{I}_i^j\}\bone\{F_n(Y_j)\neq 1, 1/n\}}{F(Y_j)(1 - F(Y_j))}\mid F_n(Y_j) = r/n] \\
    &=& \frac{1}{n}\sum_{r = 2}^{n-1} \ee[\frac{\ee[\bone\{Y_j\in\mathcal{I}_i^j\}\mid Y_j]}{F(Y_j)(1 - F(Y_j))}\mid F_n(Y_j) = r/n] \\
    &\leq& \frac{1}{n}\sum_{r = 2}^{n-1} \big(\frac{(r - 1)(n - r)}{(n - 1)(n - 2)}\big)\ee[\frac{1}{F(Y_j)(1 - F(Y_j))}\mid F_n(Y_j) = r/n]
\end{eqnarray*}
Given $F_n(Y_j) = r/n$, $F(Y_j)\sim\text{Beta}(r, n - r + 1)$, therefore this gives us
\begin{eqnarray*}
    \ee[\frac{\bone\{Y_j\in\mathcal{I}_i^j\}\bone\{F_n(Y_j)\neq 1, 1/n\}}{F(Y_j)(1 - F(Y_j))}] \leq 2,
\end{eqnarray*}
which means $\big(\frac{n - n_0}{n}\big)Q_n^\prime$ is uniformly integrable. If $\mu$ does not have a continuous density, showing uniform integrability of $\big(\frac{n - n_0}{n}\big)Q_n^\prime$ requires extra work. We divide the argument into the following four cases: (i) Support $\mu$ attains a minimum $s_{\min}$ and a maximum $s_{\max}$ which $\mu$ has point masses on; (ii) Support $\mu$ attains a maximum $s_{\max}$ which $\mu$ has a mass point on but support $\mu$ either does not attain a minimum or it does not have a mass point on its minimum; (iii) Support $\mu$ attains a minimum $s_{\min}$ which $\mu$ has a mass point on but support $\mu$ either does not attain a maximum or it does not have a mass point on its maximum; (iv) Support $\mu$ attains a minimum or maximum or does not have point masses on them. 

\noindent\textbf{Case (i).} There exists $\delta > 0$ such that $\mu(s_{\max}), \mu(s_{\min})\geq \delta$.
\begin{eqnarray*}
     \lefteqn{\ee[\frac{\bone\{Y_j\in\mathcal{I}_i^j\}\bone\{F_n(Y_j)\neq 1, 1/n\}}{F(Y_j)(1 - F(Y_j))}]}\\
     &=& \int_{S\setminus\{ s_{\max}\}}\frac{\ee[\bone\{Y_j\in\mathcal{I}_i^j\}\bone\{F_n(Y_j)\neq 1, 1/n\}\mid Y_j = t]}{F(t)(1 - F(t))}d\mu(t) \\
     &\leq & \frac{1 + \delta}{\delta(1 - \delta)}.
\end{eqnarray*}
\noindent\textbf{Case (ii).} There exists $\delta > 0$ such that $\mu(s_{\max})\geq \delta$ and $\mu(s_{\min}) = 0$ or $S$ does not have a minimum. 
\begin{eqnarray*}
     \lefteqn{\ee[\frac{\bone\{Y_j\in\mathcal{I}_i^j\}\bone\{F_n(Y_j)\neq 1, 1/n\}}{F(Y_j)(1 - F(Y_j))}]}\\
     &=& \int_{S\setminus\{s_{\max}\}}\frac{\ee[\bone\{Y_j\in\mathcal{I}_i^j\}\bone\{F_n(Y_j)\neq 1, 1/n\}\mid Y_j = t]}{F(t)(1 - F(t))}d\mu(t) \\
     &\leq & \int_{F(t) < (n-1)^{-1}}\frac{\ee[\bone\{F_n(Y_j)\neq 1, 1/n\}\mid Y_j = t]}{F(t)(1 - F(t))}d\mu(t) + \\
     && \int_{(n-1)^{-1} \leq F(t) < 1 - \delta }\frac{\ee[\bone\{Y_j\in\mathcal{I}_i^j\}\mid Y_j = t]}{F(t)(1 - F(t))}d\mu(t) \\
     &\leq & \int_0^{(n-1)^{-1}}\frac{1 - (1 - x)^{n - 1}}{x(1 - x)}dx + \int_{(n-1)^{-1}}^{1 - \delta}\frac{2x(1 - x)}{x(1 - x)}dx.
\end{eqnarray*}
For large $n$ we have
\begin{eqnarray*}
    \int_0^{(n - 1)^{-1}}\frac{1 - (1 - x)^{n - 1}}{x(1 - x)}dx & \lesssim& \int_0^{(n - 1)^{-1}}\frac{(n - 1)x}{x(1 - x)}dx \leq \frac{n-2}{n-1} \leq 2.
\end{eqnarray*}
Therefore 
\begin{eqnarray*}
    \ee[\frac{\bone\{Y_j\in\mathcal{I}_i^j\}\bone\{F_n(Y_j)\neq 1, 1/n\}}{F(Y_j)(1 - F(Y_j))}] \leq 4.
\end{eqnarray*}
\noindent\textbf{Case (iii).} There exists $\delta > 0$ such that $\mu(s_{\min})\geq \delta$ and $\mu(s_{\max}) = 0$ or $S$ does not have a maximum. Note that by symmetry, this is equivalent to the previous case.

\noindent\textbf{Case (iv).} $\mu$ is not continuous but does not have point masses at minimum or maximum. Note that this is similar to case (ii). 
\begin{eqnarray*}
    \lefteqn{\ee[\frac{\bone\{Y_j\in\mathcal{I}_i^j\}\bone\{F_n(Y_j)\neq 1, 1/n\}}{F(Y_j)(1 - F(Y_j))}]}\\
    &\leq & \int_{\min\{F(t), 1-F(t)\}\leq (n-1)^{-1}}\frac{\ee[\bone\{F_n(Y_j)\neq 1, 1/n\}\mid Y_j = t]}{F(t)(1 - F(t))}d\mu(t) + \\
    && \int_{(n-1)^{-1} < F(t) < 1-(n-1)^{-1}}\frac{\ee[\bone\{Y_j\in\mathcal{I}_i^j\}\mid Y_j = t]}{F(t)(1 - F(t))}d\mu(t) \\
    &\leq & 6.
\end{eqnarray*}
Therefore $\big(\frac{n - n_0}{n}\big)Q_n^\prime$ is uniformly integrable.

Note that by Lemma 11.3. in \cite{azadkia2021fast} $\bbx_{N^{-j}(i)}\rightarrow \bbx_i$ with probability one. Then, using Lemma 11.7. in \cite{azadkia2021fast} with probability one we have
\begin{align*}
    \ee[\bone\{Y_j\in\mathcal{I}_i^j\}\mid Y_j, \bbx_i, \bbx_{N^{-j}(i)}] - \ee[\bone\{Y_j\in\mathcal{I}_i^\prime\}\mid Y_j, \bbx_i] \rightarrow 0,
\end{align*}
where $\mathcal{I}_i^\prime = [\min\{Y_i, Y_i^\prime\}, \max\{Y_i, Y_i^\prime\}]$ in which $Y_i$ and $Y_i^\prime$ are i.i.d. given $\bbx_i$. Also 
\begin{align*}
    \ee[\bone\{Y_j\in\mathcal{I}_i^\prime\}\mid Y_j] &= \ee[\ee[\bone\{Y_j\in\mathcal{I}_i^\prime\}\mid Y_j, \bbx_i]\mid Y_j] \\
    & \rightarrow 2\ee[F_{\bbx_i}(Y_j)(1 - F_{\bbx_i}(Y_j))\mid Y_j]. 
\end{align*}
Since $\bone\{F_n(Y_j)\neq 1, 1/n\}$ converges almost surely to $\bone\{Y_j\in\tilde{S}\}$, by the dominated convergence theorem, we have
\begin{align*}
\ee[\big(\frac{n - n_0}{n}\big)Q_n^\prime]\rightarrow\int_{\tilde{S}}\frac{\ee[F_\bbx(t)(1 -  F_\bbx(t))]}{F(t)(1 - F(t))}d\mu(t).
\end{align*}
Considering that $1-n_0/n$ converges almost surely to $\mu(\tilde{S})$ which is bounded away from zero, $(1 - \frac{n_0}{n})^{-1} - \mu(\tilde{S})^{-1}$ converges almost surely to zero. Finally the uniformly integrability of $\big(\frac{n - n_0}{n}\big)Q_n^\prime$ gives us
\begin{align*}
   \ee[Q_n^\prime] = \ee\bigg[\big(\frac{n}{n - n_0} - \frac{1}{\mu(\tilde{S})}\big)\big(\frac{n - n_0}{n}\big) Q_n^\prime\bigg] + \frac{1}{\mu(\tilde{S})} \ee\bigg[\big(\frac{n - n_0}{n}\big) Q_n^\prime\bigg].
\end{align*}
The first term on the right-hand side of the above equality converges to zero by the Vitali convergence theorem. Therefore 
\begin{eqnarray*}
    \lim_{n\rightarrow\infty}\ee[Q_n^\prime] = \frac{1}{\mu(\tilde{S})}\int_{\tilde{S}}\frac{\ee[F_\bbx(t)(1 -  F_\bbx(t))]}{F(t)(1 - F(t))}d\mu(t) = Q.
\end{eqnarray*}
Putting steps I and II together gives us $\lim_{n\rightarrow\infty}\ee[Q_n] = Q$.
\end{proof}

\begin{lmm}\label{lmm2}
    For $Q_n$ defined in ~\eqref{eq:Qn}, there are constants $C_1$ and $C_2$ such that 
    \begin{eqnarray*}
        \pp(\abs{Q_n - \ee[Q_n]}\geq t) \leq C_1e^{-C_2nt^2/\log^2 n}.
    \end{eqnarray*}
\end{lmm}
\begin{proof}
We apply the bounded difference inequality~\cite{mcdiarmid1989method} to establish concentration. To do so, we first derive an upper bound on the maximum change in $Q_n$ resulting from replacing a single observation $(\bbx_k, Y_k)$ with an alternative value $(\bbx_k^\prime, Y_k^\prime)$ for any $k \in [n]$. We decompose this change into two steps: first, replacing $(Y_k, \bbx_k)$ with $(Y_k^\prime, \bbx_k)$, and second, replacing $(Y_k^\prime, \bbx_k)$ with $(Y_k^\prime, \bbx_k^\prime)$.
    
Take an arbitrary $k\in[n]$. Let $Q_n^{k_Y}$ be defined similar to $Q_n$ but using sample $\{(Y_i, \bbx_i)\}_{i \neq k}\cup\{(Y_k^\prime, \bbx_k)\}$. We show that $\abs{Q_n - Q_n^{k_Y}} \leq C\log n/n$ for some constant $C$ that only depends on the dimension of $\bbx$.

First, observe that since $\bbx_k$ remains unchanged, the nearest neighbour indices are unaffected. We analyse the effect of modifying $Y_k$ under two distinct scenarios: (i) neither $Y_k$ nor $Y_k^\prime$ is the minimum or maximum among $\{Y_i\}_{i \ne k}$; (ii) at least one of $Y_k$ or $Y_k^\prime$ is the minimum or maximum relative to $\{Y_i\}_{i \ne k}$.

\noindent\textbf{Case (i). Neither $Y_k$ nor $Y_k^\prime$ attains the minimum or maximum value.} Note that in this case, for all indices $j \in [n]$, the indicator $\bone\{n^{-1} < F_n(Y_j) < 1\}$ remains unchanged, as replacing $Y_k$ with $Y_k^\prime$ does not alter the minimum or maximum of the $\{Y_i\}$. Consequently, $n_0$ also remains unchanged. Without loss of generality, we assume $Y_k < Y_k^\prime$. Then we have
\begin{align*}
    2(n-1)(n - n_0)Q_n &= \sum_{\substack{j: Y_j < Y_k \text{ or } Y_j > Y_k' \\ n^{-1} < F_n(j) < 1}} \sum_{\substack{i \neq j \\ i \neq k,\, N^{-j}(i) \neq k}} \frac{\bone\{Y_j \in \mathcal{I}_i^j\}}{F_{n,j}(Y_j)(1 - F_{n,j}(Y_j))} + \\
    &\quad \sum_{\substack{j: Y_j < Y_k \text{ or } Y_j > Y_k' \\ n^{-1} < F_n(j) < 1}} \sum_{\substack{i \neq j \\ i = k \text{ or } N^{-j}(i) = k}} \frac{\bone\{Y_j \in \mathcal{I}_i^j\}}{F_{n,j}(Y_j)(1 - F_{n,j}(Y_j))} + \\
    &\quad \sum_{\substack{j: Y_k \leq Y_j \leq Y_k' \\ n^{-1} < F_n(j) < 1}} \sum_{\substack{i \neq j \\ i \neq k,\, N^{-j}(i) \neq k}} \frac{\bone\{Y_j \in \mathcal{I}_i^j\}}{F_{n,j}(Y_j)(1 - F_{n,j}(Y_j))} + \\
    &\quad \sum_{\substack{j: Y_k \leq Y_j \leq Y_k' \\ n^{-1} < F_n(j) < 1}} \sum_{\substack{i \neq j \\ i = k \text{ or } N^{-j}(i) = k}} \frac{\bone\{Y_j \in \mathcal{I}_i^j\}}{F_{n,j}(Y_j)(1 - F_{n,j}(Y_j))} + \\
    &\quad \sum_{i \neq k} \frac{\bone\{Y_k \in \mathcal{I}_i^k\}}{F_{n,k}(Y_k)(1 - F_{n,k}(Y_k))} \\
    &= A_1 + A_2 + A_3 + A_4 + A_5.
\end{align*}
We denote the corresponding terms involving $Y_k^\prime$ by $A_i^{k_Y}$ for $i = 1, \ldots, 5$. Observe that for all $j$ such that $Y_j < Y_k$ or $Y_j > Y_k^\prime$, the empirical distribution values remain unchanged, i.e., $F_{n,j}(Y_j) = F_{n,j}^k(Y_j)$, where $F_{n,j}^k(Y_j)$ denotes the empirical distribution after replacing $Y_k$ with $Y_k^\prime$. Consequently, in the terms $A_1$ and $A_2$, all denominators remain unchanged after the modification. In contrast, for indices $j$ such that $Y_k \leq Y_j \leq Y_k^\prime$, the value of $F_{n,j}(Y_j)$ changes by exactly $(n - 1)^{-1}$.

We first focus on $A_1$ and $A_2$. Since changing $Y_k$ to $Y_k^\prime$ does not affect the denominators, it suffices to analyse the numerator term $\bone\{Y_j \in \mathcal{I}_i^j\}$. In the case of $A_1$, the intervals $\mathcal{I}_i^j$ remain unchanged under the replacement of $Y_k$ with $Y_k^\prime$, so $A_1$ is unaffected, i.e., $A_1 = A_1^{k_Y}$.  

For $A_2$, consider first the case where $Y_j < Y_k < Y_k^\prime$. For any $i$ such that $N^{-j}(i) = k$, the indicator $\bone\{Y_j \in \mathcal{I}_i^j\}$ remains unchanged when $Y_k$ is replaced by $Y_k^\prime$. A similar argument holds when $Y_k < Y_k^\prime < Y_j$.  

Finally, consider the case where $i = k$. Even in this situation, the indicator $\bone\{Y_j \in \mathcal{I}_k^j\}$ remains unchanged under the modification of $Y_k$, and thus $A_2 = A_2^{k_Y}$.

Now consider $A_3$. Note that all indicator terms $\bone\{Y_j \in \mathcal{I}_i^j\}$ remain unchanged when $Y_k$ is replaced by $Y_k^\prime$. Therefore, it suffices to bound the difference
\[
\left| \frac{1}{F_{n,j}(Y_j)(1 - F_{n,j}(Y_j))} - \frac{1}{F_{n,j}^k(Y_j)(1 - F_{n,j}^k(Y_j))} \right|
\]
for those indices $i$ such that $\bone\{Y_j \in \mathcal{I}_i^j\} = 1$. We first consider the case where there are no ties among the $Y_i$'s. In this setting, for each $j$, Lemma 11.4. in \cite{azadkia2021simple} implies that there are at most $nC(p)\min\{F_n(Y_j) - n^{-1}, 1 - F_n(Y_j)\}$ such indices $i$ for which $Y_j \in \mathcal{I}_i^j$. This gives us 
\allowdisplaybreaks
\begin{align*}
    \lefteqn{\abs{A_3 - A_3^{k_Y}} \leq \sum_{\substack{j: Y_k \leq Y_j \leq Y_k^\prime \\ n^{-1} < F_n(Y_j) < 1}} 
    n C(p)\min\{F_n(Y_j) - \frac{1}{n}, 1 - F_n(Y_j)\} \times} \\
    &\quad \abs{\frac{1}{F_{n,j}(Y_j)(1 - F_{n,j}(Y_j))} - \frac{1}{F_{n,j}^k(Y_j)(1 - F_{n,j}^k(Y_j))}} \\
    &= nC(p)\sum_{j = 3}^{n - 1}\min\left\{\frac{j-1}{n}, 1- \frac{j}{n}\right\}\left|\frac{1}{\left(\frac{j - 1}{n - 1}\right)\left(1 - \frac{j - 1}{n - 1}\right)} - \frac{1}{\left(\frac{j - 2}{n - 1}\right)\left(1 - \frac{j - 2}{n - 1}\right)}\right| \\
    &\leq 2nC(p)\Bigg(\sum_{j = 1}^{n/2 - 1} \left( \frac{1}{j} + \frac{1}{n - j} \right) 
    - n\sum_{i = n/2}^{n - 2}\frac{1}{i(i + 1)} \Bigg) \\
    &= O(n\log n).
\end{align*}
The case where ties exist among the $Y_i$'s is similar but requires additional care. Let $r_1 < \cdots < r_m$ denote the ordered sequence of distinct values taken by the empirical ranks of $Y_j$ for $j \in [n]$. Define $\ell_*$ as the smallest index $i \in [m]$ such that for every $j$ satisfying $Y_k \leq Y_j \leq Y_k^\prime$, we have $F_n(Y_j) \leq r_i$. Similarly, define $\ell^*$ as the largest index $i \in [m]$ such that for every such $j$, $F_n(Y_j) \geq r_i$.
Then
\begin{align*}
    \lefteqn{\abs{A_3 - A_3^{k_Y}} \leq C(p)(n-1)^2\times}\\
    & \sum_{i=\ell_*}^{\ell^*} (r_i - r_{i-1})\min\{(r_i - 1), (n - r_i)\}\left|\frac{1}{(r_i - 1)(n - r_i)} - \frac{1}{(r_i - 2)(n - r_i + 1)}\right|.
\end{align*}

For all indices $i$ such that $r_i \leq n/2$, replacing the corresponding $Y_j$ values with distinct (tie-free) values can only increase the difference $\abs{A_3 - A_3^{k_Y}}$. Therefore, it suffices to bound this difference in the case where $r_{\ell_*} \geq n/2$, since for all $r_i\leq n/2$ we can use the bound on this difference when there are no ties. In this case, we have
\begin{align*}
    \abs{A_3 - A_3^{k_Y}} \leq C(p)n^2\sum_{i={\ell_*}}^{\ell^*} (r_i - r_{i-1})\frac{\abs{2r_i - n - 2}}{(r_i - 1)(r_i - 2)(n - r_i + 1)}.
\end{align*}
Define
\[
g(r)=\frac{\abs{2r - n - 2}}{(r - 1)(r - 2)(n - r + 1)},
\qquad r \in \{1,\dots,n-1\}.
\]

For $r \ge n/2$, the function $g$ is U‑shaped and attains its minimum at
$\lceil n/2 + 1\rceil$.
Hence, for every index $i$ with $r_i \ge \lceil n/2 + 1\rceil$ we bound $g(r_i)$ above by
$g(r_{\ell^*})$.
Because $r_{\ell^*} < n$, we have
$n - r_{\ell^*} + 1 = O(n)$, which implies
\[
\abs{A_3 - A_3^{k_Y}} = O(n).
\]

Combining this bound with those obtained in the remaining cases yields the overall estimate
\[
\abs{A_3 - A_3^{k_Y}} = O(n\log n).
\]

For $A_4$, observe that for any fixed $j$ there are at most $C(p)$ indices $i$ such that either $i = k$ or $N^{-j}(i) = k$. Consequently,
\allowdisplaybreaks
\begin{align*}
    \abs{A_4 - A_4^{k_Y}} 
    &\leq \sum_{\substack{j: Y_k \leq Y_j \leq Y_k^\prime \\ n^{-1} < F_n(Y_j) < 1}} 
    C(p)\abs{\frac{1}{F_{n,j}(Y_j)(1 - F_{n,j}(Y_j))} 
    - \frac{1}{F_{n,j}^k(Y_j)(1 - F_{n,j}^k(Y_j))}}
\end{align*}
We first examine the case in which the $Y_j$ are all distinct. Then
\begin{align*}
    \abs{A_4 - A_4^{k_Y}} &\leq 2C(p)n^2\sum_{j = 2}^{n/2}
    \abs{\frac{1}{(j - 1)(n - j)} 
    - \frac{1}{(j - 2)(n - j + 1)}}  = O(n).
\end{align*}
When ties are present among the $Y_j$ values, we have
\begin{align*}
    \abs{A_4 - A_4^{k_Y}} &\leq C(p)n^2\sum_{i = \ell_*}^{\ell^*}
    (r_i - r_{i-1})\abs{\frac{1}{(r_i - 1)(n - r_i)} 
    - \frac{1}{(r_i - 2)(n - r_i + 1)}}  \\
    &= O(n).
\end{align*}
Finally, observe that 
\[
\abs{A_5 - A_5^{k_Y}} \le A_5 + A_5^{k_Y}.
\]
We therefore bound $A_5$ only, as the same argument applies verbatim to $A_5^{k_Y}$.  
Since there are at most 
$n C(p)\min\{F_n(Y_k),\,1 - F_n(Y_k)\}$ 
indices \(i\) for which \(\bone\{Y_k \in \mathcal{I}_i^k\}=1\), we have
\begin{eqnarray*}
    A_5 \leq \frac{n C(p)\min\{F_n(Y_k), 1 - F_n(Y_k)\}}{F_{n,k}(Y_k)(1 - F_{n,k}(Y_k))} = O(n).
\end{eqnarray*}
Consequently, provided that replacing $Y_k$ with $Y_k^\prime$ leaves the sample minimum and maximum unchanged, we obtain
\begin{eqnarray*}
    \abs{Q_n - Q_n^{k_Y}} = O(\frac{\log n}{n}).
\end{eqnarray*}

\noindent\textbf{Case (ii). $Y_k$ or $Y_k^\prime$ is minimum or maximum.} Without loss of generality, assume $Y_k < Y_k^\prime$. Then one of the following scenarios arises:

\noindent\textbf{(a)} \textit{Replacing $Y_k$ with $Y_k^\prime$ leaves both the sample minima and maxima unchanged:} $Y_k < Y_k^\prime \le Y_j$ for all $j \ne k$. Consequently, $Q_n = Q_n^{k_Y}$.

\noindent\textbf{(b)} \textit{Replacing $Y_k$ with $Y_k^\prime$ alters the set of minima but leaves the set of maxima unchanged:} 
we have $Y_k \le Y_j$ for every $j \ne k$, and there exists at least one index $j$ with $Y_j \ge Y_k^\prime$.  
If, for some such $j$, we have $n^{-1} < F_n(Y_j) < 1$ before the change and $F_n^{k}(Y_j)=n^{-1}$ afterwards, then the contribution of that $j$ to $\lvert Q_n - Q_n^{k_Y}\rvert$ is bounded by  
\[
\frac{C(p)}{2\,(n-2)\,(n - n_0 - 1)} = O(n^{-1}).
\]
For every other index $j$, the argument from case (i) applies.

\noindent\textbf{(c)} \textit{Replacing $Y_k$ with $Y_k^\prime$ changes both the sample minima and maxima:} 
indeed, $Y_k \le Y_j \le Y_k^\prime$ for every $j \ne k$.  
Assume there exist indices $j_1$ and $j_2$ such that 
\[
n^{-1} < F_n(Y_{j_1}) < 1,\quad F_n^{k}(Y_{j_1}) = n^{-1},\quad
F_n(Y_{j_2}) = 1,\quad n^{-1} < F_n^{k}(Y_{j_2}) < 1 .
\]
The combined contribution of these two indices to $\abs{Q_n - Q_n^{k_Y}}$ is bounded by
\[
\frac{C(p)}{(n-2)\,(n - n_0 - 1)} = O(n^{-1}).
\]
For all remaining indices $j$, the reasoning from case (i) applies unchanged.

\noindent\textbf{(d)} \textit{Replacing $Y_k$ with $Y_k^\prime$ leaves the set of sample minima unchanged but alters the set of sample maxima:}  
we have $Y_j \leq Y_k^\prime$ for every $j \ne k$, and there exists at least one index $j$ with $Y_j \leq Y_k$.  
If there is an index $j$ for which $F_n(Y_j)=1$ and $n^{-1} < F_n^{k}(Y_j) < 1$, the contribution of that $j$ to $\abs{Q_n - Q_n^{k_Y}}$ is bounded by
\[
\frac{C(p)}{2\,(n-2)\,(n - n_0 - 1)} = O(n^{-1}).
\]
For every other index $j$, the argument from case (i) applies.

\medskip
\noindent
Combining Cases (i)–(iv), we obtain
\[
\abs{Q_n - Q_n^{k_Y}} \leq \frac{C(p)\,\log n}{n},
\]
whenever $(Y_k,\bbx_k)$ is replaced by $(Y_k^\prime,\bbx_k)$.

We now analyse the change induced when replacing $(Y_k^\prime,\bbx_k)$ with $(Y_k^\prime,\bbx_k^\prime)$.  
Because the $Y_i$ values remain unchanged, both the denominators
$F_{n,j}(Y_j)\bigl(1 - F_{n,j}(Y_j)\bigr)$ and the index set
$\{j : n^{-1} < F_n(Y_j) < 1\}$ are unaffected.  
For notational convenience, therefore, we study the effect of changing
$(Y_k,\bbx_k)$ to $(Y_k,\bbx_k^\prime)$.

Let $Q_n^{k_\bx}$ denote the analogue of $Q_n$ computed from the sample in
which $\bbx_k$ is replaced by $\bbx_k^\prime$.
For each fixed $j$, modifying $\bbx_k$ can alter at most $C(p)$ of the
intervals $\mathcal{I}_i^j$.
Among those indices $i$ whose intervals change, only those for which
$\bone\{Y_j \in \mathcal{I}_i^j\}$ flip value matters—namely, the indices
where $Y_j \in \mathcal{I}_i^j$ under $\bbx_k$ but $Y_j \notin \mathcal{I}_i^j$ under $\bbx_k^\prime$, or vice versa.

Finally, if $Y_j$ has rank $r_i$, then at most $\min\{r_i - 1, n - r_i\}$ of the indicators $\bone\{Y_j \in \mathcal{I}_i^j\}$ equal $1$ under either
$\bbx_k$ or $\bbx_k^\prime$. Therefore
\begin{align*}
    \abs{Q_n - Q_n^{k_\bx}} &\leq \left(\frac{n - 1}{n - n_0}\right)\sum_{i = 1}^m (r_i - r_{i-1})\frac{\min\{C(p), r_i - 1, n - r_i\}}{(r_i - 1)(n - r_i)}\\
    &\leq C(p)\left(\frac{n - 1}{n - n_0}\right)\sum_{i = 1}^\ell \frac{(r_i - r_{i-1})}{(r_i - 1)(n - r_i)} \\
    & \leq \frac{C(p)\log n}{n}.
\end{align*}
Combining the bounds for $\abs{Q_n - Q_n^{k_Y}}$ and $\abs{Q_n - Q_n^{k_\bx}}$, we obtain that replacing $(Y_k,\bbx_k)$ with $(Y_k^\prime,\bbx_k^\prime)$ yields
\begin{eqnarray*}
    \abs{Q_n - Q_n^{k}} \leq \frac{C(p)\log n}{n}. 
\end{eqnarray*}
Applying McDiarmid's bounded‑difference inequality~\cite{mcdiarmid1989method} gives
\begin{eqnarray*}
    \pp(\abs{Q_n - \ee[Q_n]}\geq t) \leq 2\exp(-Cnt^2/\log^2 n)
\end{eqnarray*}
\end{proof}
Using Lemma~\ref{lmm2}, set $t_n = \sqrt{2}(\log n)^{3/2}/\sqrt{Cn}$. Then note that 
\begin{eqnarray*}
    \sum_{n = 1}^\infty\pp(\abs{Q_n - \ee[Q_n]}\geq t_n) \leq 2\sum_{i = 1}^n\frac{1}{n^2} < \infty.
\end{eqnarray*}
By the Borel–Cantelli lemma, it follows that $\abs{Q_n - \ee[Q_n]}$ converges to zero almost surely. This, combined with Lemma~\ref{lmm1}, establishes the almost sure convergence of $Q_n$ to $Q$.
\end{proof}

\subsection{Proof of Corollary~\ref{cor:conditional}}
\begin{proof}
    Theorem~\ref{thm:consistency} guarantees the convergence of $\nu_n(Y, (\bbx, \bbz))$ and $\nu_n(Y, \bbz)$ to their population counterparts. Additionally since $Y$ is not almost surely a function of $\bbz$, we have $1 - \nu(Y, \bbz) \neq 0$. Applying continuous mapping theorem gives the desired result. 
\end{proof}

\subsection{Proof of Theorem~\ref{thm:simpleConsistency}} 
\begin{proof}
    Using Lemma 9.3. in \cite{chatterjee2021new}, the proof closely mirrors that of Theorem~\ref{thm:consistency}, hence we omit it here. The only difference is that the constant  $C(p)$ can be bounded above by 3 throughout the argument.
\end{proof}

\subsection{Proof of Proposition~\ref{prop:momentsNun}}
\begin{proof}
    For $Y$ with continuous distribution we have $n_0 = 2$, and therefore 
    \[
    \nu_n(Y, \bbx) = 1 - \frac{1}{2}\left(\frac{n-1}{n-2}\right) S_n ,
    \]
    where $S_n := \sum_{j = 1}^n U_j$ for
    \begin{align*}
        U_j := \frac{1}{(R_j - 1)(n - R_j)}\sum_{i\neq j}\bone\{R_j\in\mathcal{R}_i^j\} \bone\{R_j \neq 1, n\}.
    \end{align*}
    For $\bbx$ and $Y$ independent we have
    \begin{align*}
        \ee[U_j] &= \frac{1}{n}\sum_{r = 2}^{n - 1}\frac{1}{(r - 1)(n - r)}\ee\left[\sum_{i\neq j}\bone\{r\in\mathcal{R}_i^j\}\mid R_j = r\right] \\
        &= \frac{1}{n}\sum_{r = 2}^{n - 1}\frac{(n - 1)}{(r - 1)(n - r)}\frac{2(r - 1)(n - r)}{(n - 1)(n - 2)} \\
        &= \frac{2}{n},
    \end{align*}
    therefore 
    \[
    \ee[\nu_n(Y, \bbx)] = \frac{-1}{n-2}.
    \]

    Note that $\var(\nu_n(Y, \bbx)) = \frac{1}{4}\left(\frac{n-1}{n-2}\right)^2\var(S_n)$, and
    \begin{align*}
        \var(S_n) = \sum_{j = 1}^n\var(U_j) + \sum_{i \neq j}\Cov(U_i, U_j). 
    \end{align*}
    Hence we need to find $\var(U_j)$ and $\Cov(U_i, U_j)$. Note that 
    \begin{align*}
        \ee[U_j^2] &= \ee\left[\frac{\bone\{R_j\neq 1, n\}}{(n - R_j)^2(R_j - 1)^2}\left(\sum_{i\neq j}\bone\{R_j\in\mathcal{R}_i^j\}\right)^2\right] \\
        &= \frac{1}{n}\sum_{r = 2}^{n - 1}\frac{1}{(r - 1)^2(n - r)^2}\ee\left[\left(\sum_{i \neq j}\bone\{r\in\mathcal{R}_i^j\}\right)^2\mid R_j = r\right] \\
        &= \frac{1}{n}\sum_{r = 2}^{n - 1}\frac{1}{(r - 1)^2(n - r)^2}\ee\left[\sum_{i \neq j}\bone\{r\in\mathcal{R}_i^j\}\mid R_j = r\right] + \\
        & \frac{1}{n}\sum_{r = 2}^{n - 1}\frac{1}{(r - 1)^2(n - r)^2}\ee\left[\sum_{i, k \neq j, i\neq k}\bone\{r\in\mathcal{R}_i^j\}\bone\{r\in\mathcal{R}_k^j\}\mid R_j = r\right] \\
        &= \frac{1}{n}\sum_{r = 2}^{n - 1}\frac{1}{(r - 1)^2(n - r)^2}\frac{2(n - 1)(r - 1)(n - r)}{(n - 1)(n - 2)} + \\
        & \frac{1}{n}\sum_{r = 2}^{n - 1}\frac{(n - 1)(n - 2)}{(r - 1)^2(n - r)^2}\ee\left[\bone\{r\in\mathcal{R}_i^j\}\bone\{r\in\mathcal{R}_k^j\}\mid R_j = r\right]
    \end{align*}
    for $i \neq k$ we have two scenarios, either $\abs{\{R_i, R_k, R_{N^{-j}(i)}, R_{N^{-j}(k)}\}} = 4$ or it is smaller. In the second case, either we have $R_{N^{-j}(i)} = R_{N^{-j}(k)}$ or $R_{N^{-j}(i)} = R_k$ (or $R_{N^{-j}(k)} = R_i$). We let $p_n$ be the probability of $\abs{\{R_i, R_k, R_{N^{-j}(i)}, R_{N^{-j}(k)}\}} < 4$. Note that $p_n = O(1/n)$.
    \begin{align*}
        \ee\left[\bone\{r\in\mathcal{R}_i^j\}\bone\{r\in\mathcal{R}_k^j\}\mid R_j = r\right] &= \frac{4(r - 1)(r - 2)(n - r)(n - r - 1)}{(n - 1)(n - 2)(n - 3)(n - 4)}(1 - p_n) + \\
        & cp_n\frac{(r - 1)(r - 2)(n - r) + (r - 1)(n - r)(n - r - 1)}{(n - 1)(n - 2)(n - 3)}.
    \end{align*}
    Putting these together gives us
    \begin{align}\label{eq:varOrder}
        \var(U_i) = O\left(\frac{\log n}{n^3}\right).
    \end{align}

    \begin{align*}
        \ee[U_a U_b] &= \ee\left[\frac{\bone\{R_a\neq 1, n\}\bone\{R_b\neq 1, n\}}{(n - R_a)(R_a - 1)(n - R_b)(R_b - 1)}\left(\sum_{i\neq a}\bone\{R_a\in\mathcal{R}_i^a\}\right)\left(\sum_{j\neq b}\bone\{R_b\in\mathcal{R}_j^b\}\right)\right] \\
        &= \frac{2}{n(n - 1)}\sum_{2\leq r < s \leq n-1}\frac{1}{(r - 1)(n - r)(s - 1)(n - s)}\ee\left[\sum_{i\neq a}\sum_{j\neq b}\bone\{r\in\mathcal{R}_i^a\}\bone\{s\in\mathcal{R}_j^b\}\mid R_a = r, R_b = s\right] \\
        &= \frac{2}{n(n - 1)}\sum_{2\leq r < s \leq n-1}\frac{E_1 + E_2}{(r - 1)(n - r)(s - 1)(n - s)},
    \end{align*}
    where
    \begin{align*}
        E_1 &:= \sum_{i, j\neq a, b, i\neq j}\ee\left[\bone\{r\in\mathcal{R}_i^a\}\bone\{s\in\mathcal{R}_j^b\}\mid R_a = r, R_b = s\right], \\
        E_2 &:= \sum_{i\neq a, b}\ee\left[\bone\{r\in\mathcal{R}_i^a\}\bone\{s\in\mathcal{R}_i^b\}\mid R_a = r, R_b = s\right] + \\
        &\quad \sum_{i\neq a}\ee\left[\bone\{r\in\mathcal{R}_i^a\}\bone\{s\in\mathcal{R}_a^b\}\mid R_a = r, R_b = s\right] + \\
        &\quad \sum_{j\neq b}\ee\left[\bone\{r\in\mathcal{R}_b^a\}\bone\{s\in\mathcal{R}_j^b\}\mid R_a = r, R_b = s\right].
    \end{align*}
    Note that 
    \begin{align*}
        E_1 = \frac{4}{n^2} + O(\frac{1}{n^3}),\qquad E_2 = O(\frac{1}{n^3}).
    \end{align*}
    Therefore we have
    \begin{align}\label{eq:covOrder}
        \Cov(U_a, U_b) = O\left(\frac{1}{n^3}\right).
    \end{align}
    Putting~\ref{eq:varOrder} and~\ref{eq:covOrder} together gives us $\var(\nu_n(Y, \bbx)) = O(1/n)$. 
\end{proof}

\subsection{Proof of Proposition~\ref{thm:nullSimpleAsymp}}
\begin{proof}
Lemma~\ref{lmm:NullSimpleMoments} gives us $\ee[\nu_{n}^{\text{1-dim}}(Y, X)] = 2/n$. For the variance, let 
\begin{align*}
    A_n:=\sum_{\ell=2}^{n-2} \sum_{k=\ell+1}^{n-1} \frac{1}{(k-1)(n-\ell)}, \quad B_n:=\sum_{\ell=2}^{n-2} \sum_{k=\ell+1}^{n-1} \frac{1}{k-1}.
\end{align*}
Note that by Lemma~\ref{lmm:NullSimpleMoments} we have
\begin{align}
    n\var(\nu_{n}^{\text{1-dim}}(Y, X)) &= n\left(\frac{2}{n} A_n-\frac{1}{n}-\frac{2}{n(n-1)} B_n+o\left(\frac{1}{n}\right)\right) \nonumber\\
    &= 2 A_n-1-\frac{2}{n-1} B_n+n \cdot o\left(\frac{1}{n}\right). \label{eq:AnBn}
\end{align}
We have
\begin{align*}
    A_n=\sum_{2 \leq \ell<k \leq n-1} \frac{1}{(k-1)(n-\ell)}=H_{n-2}^{(2)}-\frac{2}{n-1} H_{n-2},
\end{align*}
and
\begin{align*}
    B_n = n-2-H_{n-2} .
\end{align*}
where $H_m=\sum_{j=1}^m 1/j$ and $H_m^{(2)}=\sum_{j=1}^m 1/j^2$. Plugging this into \eqref{eq:AnBn} we have
\begin{align*}
    n\var(\nu_{n}^{\text{1-dim}}(Y, X)) = 2 H_{n-2}^{(2)}-1-\frac{2 H_{n-2}+2 n-4}{n-1},
\end{align*}
and since $H_{n-2}^{(2)}\rightarrow\pi^2/6$ and $H_{n-2}\sim\log(n)$ we get 
\[
\lim_{n\rightarrow\infty}n\var(\nu_{n}^{\text{1-dim}}(Y, X)) = \frac{\pi^2}{3} - 3,
\]
which finishes the proof.
\end{proof}

\begin{lmm}\label{lmm:NullSimpleMoments}
    Suppose that $X$ and $Y$ are independent and $Y$ is continuous. Then 
    \begin{align*}
        \ee[\nu_{n}^{\text{1-dim}}(Y, X)] = \frac{2}{n},
    \end{align*}
    and 
    \begin{align*}
        \lefteqn{\var(\nu_{n}^{\text{1-dim}}(Y, X)) = }\\
        & \frac{2}{n}\sum_{\ell = 2}^{n - 2}\sum_{k = \ell + 1}^{n - 1}\frac{1}{(k - 1)(n - \ell)} - \frac{1}{n} - \frac{2}{n(n - 1)}\sum_{\ell = 2}^{n - 2}\sum_{k = \ell + 1}^{n - 1}\frac{1}{k - 1} + o(\frac{1}{n}).
    \end{align*}
\end{lmm}
\begin{proof}
    When $Y\perp X$ then $(r_1, \ldots, r_n)$ is random uniform permutation of $1, \ldots, n$. In this case $\nu_{n}^{\text{1-dim}}(Y, X)$ can be written as 
\begin{align*}
    \nu_{n}^{\text{1-dim}}(Y, X) = 1 - \frac{1}{2}\sum_{i = 1}^{n - 1}\sum_{j = 2}^{n-1}\frac{\bone\{j\in \mathcal{K}_i\}}{(j - 1)(n - j)}
\end{align*}
Let's focus on 
\begin{align*}
    A := \sum_{\ell = 2}^{n-1}\sum_{i = 1}^{n - 1}\frac{\bone\{\ell\in\mathcal{K}_i\}}{(\ell - 1)(n - \ell)}.
\end{align*}
We first work out the mean and variance of $A$.
\begin{align*}
    \ee[A] =& (n - 1)\sum_{\ell = 2}^{n-1}\frac{2(\ell - 1)(n - \ell)}{n(n-1)(\ell - 1)(n - \ell)} = 2 - \frac{4}{n}.
\end{align*}
For variance, we first look at the second moment of $A$
\begin{align*}
    A^2 =& \sum_{\ell = 2}^{n - 1}\sum_{k = 2}^{n - 1}\sum_{i = 1}^{n - 1}\sum_{j = 1}^{n - 1}\frac{\bone\{\ell\in\mathcal{K}_i\}\bone\{k\in\mathcal{K}_j\}}{(\ell - 1)(n - \ell)(k - 1)(n - k)} \\
    =& \sum_{\ell = 2}^{n - 1}\sum_{i = 1}^{n - 1}\frac{\bone\{\ell\in\mathcal{K}_i\}}{(\ell - 1)^2(n - \ell)^2} + \\
    & 2\sum_{\ell = 2}^{n - 1}\sum_{i = 1}^{n - 2}\frac{\bone\{\ell\in\mathcal{K}_i\}\bone\{\ell\in\mathcal{K}_{i+1}\}}{(\ell - 1)^2(n - \ell)^2} +\\
    & 2\sum_{\ell = 2}^{n - 1}\sum_{i = 1}^{n - 3}\sum_{j = i+2}^{n - 1}\frac{\bone\{\ell\in\mathcal{K}_i\}\bone\{\ell\in\mathcal{K}_j\}}{(\ell - 1)^2(n - \ell)^2} + \\
    & 2\sum_{\ell = 2}^{n - 2}\sum_{k = \ell + 1}^{n - 1}\sum_{i = 1}^{n - 1}\frac{\bone\{\ell\in\mathcal{K}_i\}\bone\{k\in\mathcal{K}_i\}}{(\ell - 1)(n - \ell)(k - 1)(n - k)} + \\
    & 4\sum_{\ell = 2}^{n - 2}\sum_{k = \ell + 1}^{n - 1}\sum_{i = 1}^{n - 2}\frac{\bone\{\ell\in\mathcal{K}_i\}\bone\{k\in\mathcal{K}_{i+1}\}}{(\ell - 1)(n - \ell)(k - 1)(n - k)} + \\
    & 4\sum_{\ell = 2}^{n - 2}\sum_{k = \ell + 1}^{n - 1}\sum_{i = 1}^{n - 3}\sum_{j = i + 2}^{n - 1}\frac{\bone\{\ell\in\mathcal{K}_i\}\bone\{k\in\mathcal{K}_j\}}{(\ell - 1)(n - \ell)(k - 1)(n - k)} \\
    =& A_1 + A_2 + A_3 + A_4 + A_5 + A_6.
\end{align*}

Let $H_m = \sum_{j = 1}^m 1/j$. Then
\begin{align*}
    \ee[A_1] =& \frac{2}{n}\sum_{\ell = 2}^{n - 1}\frac{1}{(\ell - 1)(n - \ell)} = \frac{4H_{n - 2}}{n(n - 1)}.
\end{align*}

\begin{align*}
    \ee[A_2] =& \frac{2}{n(n - 1)}\sum_{\ell = 2}^{n - 1}\frac{(n - \ell - 1) + (n - \ell)(\ell - 1)(\ell - 2)}{(\ell - 1)(n - \ell)} = \frac{4(n - 3)}{n(n - 1)^2}H_{n - 2}.
\end{align*}

\begin{align*}
    \ee[A_3] =& \frac{4}{n(n - 1)}\sum_{\ell = 2}^{n - 1}\frac{(\ell - 2)(n - \ell - 1)}{(\ell - 1)(n - \ell)} \\
    =& \frac{4(n - 2)}{n(n - 1)}(1 - \frac{2H_{n - 2}}{n - 2} + \frac{2H_{n - 2}}{(n-1)(n - 2)}).
\end{align*}

\begin{align*}
    \ee[A_4] =& \frac{4}{n}\sum_{\ell = 2}^{n - 2}\sum_{k = \ell + 1}^{n - 1}\frac{1}{(n - \ell)(k - 1)}.
\end{align*}

\begin{align*}
    \ee[A_5] =& \frac{4}{n(n-1)}\sum_{\ell = 2}^{n - 2}\sum_{k = \ell + 1}^{n - 1}\frac{(n - \ell) + (k - \ell - 3) + (k - 1)}{(n - \ell)(k - 1)} \\
    =& \frac{4}{n(n-1)}\sum_{\ell = 2}^{n - 2}\sum_{k = \ell + 1}^{n - 1}\frac{1}{k - 1} + \frac{2}{n - \ell} - \frac{\ell + 2}{(n - \ell)(k - 1)} \\
    =& \frac{4}{n(n-1)}\sum_{\ell = 2}^{n - 2}\sum_{k = \ell + 1}^{n - 1}\frac{2}{k - 1} + \frac{2}{n - \ell} - \frac{n + 2}{(n - \ell)(k - 1)} \\
    =& \frac{8}{n} + \frac{8}{n(n - 1)}\sum_{\ell = 2}^{n - 2}\sum_{k = \ell + 1}^{n - 1}\frac{1}{k - 1} - \frac{16}{n(n - 1)} - \frac{8H_{n - 2}}{n(n - 1)^2} \\
    & - \frac{4}{n}\sum_{\ell = 2}^{n - 2}\sum_{k = \ell + 1}^{n - 1}\frac{1}{(k - 1)(n - \ell)} - \frac{16}{n(n - 1)}\sum_{\ell = 2}^{n - 2}\sum_{k = \ell + 1}^{n - 1}\frac{1}{(k - 1)(n - \ell)}.
\end{align*}

\begin{align*}
    \ee[A_6] =& \frac{8}{n(n - 1)}\sum_{\ell = 2}^{n - 2}\sum_{k = \ell + 1}^{n - 1}\frac{(n - \ell)(k - 1) - 3(k - 1) - (n - \ell) + \ell + 2 + (k-2)}{(n - \ell)(k - 1)} \\
    =& \frac{8}{n(n - 1)}\sum_{\ell = 2}^{n - 2}\sum_{k = \ell + 1}^{n - 1} 1 - \frac{2}{n - \ell} - \frac{2}{k - 1} + \frac{n+1}{(n - \ell)(k - 1)}\\
    =& 4 - \frac{32}{n} + \frac{16H_{n-2}}{n(n - 1)} + \frac{24}{n(n - 1)} - \frac{16}{n(n - 1)}\sum_{\ell = 2}^{n - 2}\sum_{k = \ell + 1}^{n - 1}\frac{1}{k - 1} +\\
    &\qquad\frac{8(n + 1)}{n(n - 1)}\sum_{\ell = 2}^{n - 2}\sum_{k = \ell + 1}^{n - 1}\frac{1}{(k - 1)(n - \ell)}.
\end{align*}
Putting these together, we have 
\begin{align*}
    \var(A) = \frac{8}{n}\sum_{\ell = 2}^{n - 2}\sum_{k = \ell + 1}^{n - 1}\frac{1}{(k - 1)(n - \ell)} - \frac{8}{n(n - 1)}\sum_{\ell = 2}^{n - 2}\sum_{k = \ell + 1}^{n - 1}\frac{1}{k - 1} - \frac{4}{n} + o(\frac{1}{n^2}).
\end{align*}
Then note that $\ee[\nu_{n}^{\text{1-dim}}(Y, X)] = 1 - \ee[A]/2 = 2/n$, and $\var(\nu_{n}^{\text{1-dim}}(Y, X)) = \var(A)/4$. This finishes the proof.  
\end{proof}

\subsection{Proof of Theorem~\ref{thm:concentration}}
\begin{proof}
    This results immediately from Lemma~\ref{lmm2}.
\end{proof}

\subsection{Proof of Theorem~\ref{thm:rate}}
Throughout this section, we will assume that the assumptions (A1) and (A2) from Sub Section \ref{subsec:rate} hold. In the following, we restate Lemma 14.1~\cite{azadkia2021simple} and its proof for convenience. Let $\bbx_{n, 1}$ be the nearest neighbour of $\bbx_1$ among $\bbx_2, \ldots, \bbx_n$ (with ties broken at random). 

\begin{lmm}\label{lmm:nn}
    Under assumption \textnormal{(A2)}, there is some $C$ depending only on $K$ and $p$ such that

\begin{eqnarray*}
    \ee\left(\left\|\bbx_1-\bbx_{n, 1}\right\|\right) \leq \begin{cases}C n^{-1}(\log n)^2 & \text { if } p=1 \\ C n^{-1 / p}(\log n) & \text { if } p \geq 2\end{cases}
\end{eqnarray*}
\end{lmm}

\begin{proof}
    Throughout this proof, $C$ will denote any constant that depends only on $K$ and $p$. Take $\varepsilon \in(n^{-1 / p}, 1)$. Let $B$ be the ball of radius $K$ in $\rr^p$ centred at the origin. Partition $B$ into at most $C K^p \varepsilon^{-p}$ small sets of diameter $\leq \varepsilon$. Let $E$ be the small set containing $\bbx_1$. Then
\begin{align*}
    \pp(\left\|\bbx_1-\bbx_{n, 1}\right\| \geq \varepsilon) = \pp(\bbx_2 \notin E, \ldots, \bbx_n \notin E).
\end{align*}
Now note that
\begin{align*}
    \pp(\bbx_2 \notin E, \ldots, \bbx_n \notin E \mid \bbx_1)=(1-\pp(\bbx_2 \in E \mid \bbx_1))^{n-1}=(1-\lambda(E))^{n-1},
\end{align*}
where $\lambda$ is the law of $\bbx$. Let $A$ be the collection of all small sets with $\lambda$-mass less than $\delta$. Since there are at most $C K^p \varepsilon^{-p}$ small sets, we get
\begin{align*}
\ee\left[(1-\lambda(E))^{n-1}\right] \leq(1-\delta)^{n-1}+\pp\left(\bbx_1 \in A\right) \leq(1-\delta)^{n-1}+C K^p \varepsilon^{-p} \delta.
\end{align*}
This gives
\begin{align*}
\pp(\left\|\bbx_1-\bbx_{n, 1}\right\| \geq \varepsilon) \leq (1-\delta)^{n-1} + C K^p \varepsilon^{-p} \delta.
\end{align*}

Now choosing $\delta = n^{-1} \log n$, we get
\begin{align*}
\pp(\left\|\bbx_1-\bbx_{n, 1}\right\| \geq \varepsilon) \leq \frac{1}{n} + \frac{C K^p\log n}{n \varepsilon^p}.
\end{align*}
Thus,
\begin{align*}
\ee(\left\|\bbx_1-\bbx_{n, 1}\right\|) \leq & n^{-1 / p}+\int_{n^{-1 / p}}^{2K}\pp(\left\|\bbx_1 - \bbx_{n, 1}\right\| \geq \varepsilon) d \varepsilon \\
\leq & n^{-1 / p}+\frac{C K^p\log n}{n} \int_{n^{-1 / p}}^{2K} \varepsilon^{-p} d \varepsilon.
\end{align*}

Finally, the last term is bounded by $C K n^{-1}(\log n)^2$ when $p = 1$, and by $C K^p n^{-1/p} \log n$ when $p \ge 2$.

\end{proof}

\begin{lmm}\label{lmm:rate}
    Let $C$ and $\beta$ be as in assumption \textnormal{(A1)} and $K$ be as in assumption \textnormal{(A2)}. Then there are $K_1, K_2$ and $K_3$ depending only on $C$, $\beta, K$ and $p$ such that for any $t \geq 0$,

\begin{eqnarray*}
    \pp\left(\abs{\nu_n - \nu} \geq K_1 n^{-{1/p\vee 2}}(\log n)^{\bone\{p=1\}+1} + t\right) \leq K_2 e^{-K_3 n t^2/\log n}
\end{eqnarray*}
\end{lmm}
\begin{proof}
    Recall $Q_n^\prime$ defined in \eqref{eq:Qpn}. Let $\mathcal{F}_\bbx$ be the $\sigma$-algebra generate by $\bbx_1, \ldots, \bbx_n$. Since $F_n(Y_j) = 1/n$ implies $\bone\{Y_j\in\mathcal{I}_i^j\} = 0$, we have
    \begin{align*}
        \ee[\big(\frac{n - n_0}{n}\big)Q_n^\prime] &= \frac{1}{2}\ee[\frac{\bone\{Y_j\in\mathcal{I}_i^j\}\bone\{n^{-1} < F_n(Y_j) < 1\}}{F(Y_j)(1 - F(Y_j))}] \\
        &= \frac{1}{2}\ee\left[\ee[\frac{\bone\{Y_j\in\mathcal{I}_i^j\}\bone\{n^{-1} < F_n(Y_j) < 1\}}{F(Y_j)(1 - F(Y_j))}\mid Y_j, \mathcal{F}_\bbx]\right] \\
        &= \frac{1}{2}\ee\left[\frac{\ee[\bone\{Y_j\in\mathcal{I}_i^j\}\bone\{F_n(Y_j) < 1\}\mid Y_j, \mathcal{F}_\bbx]}{F(Y_j)(1 - F(Y_j))}\right].
    \end{align*}
    In addition, note that 
    \begin{align*}
        Q = \frac{1}{2\mu(\tilde{S})}\ee\left[\frac{\ee[\bone\{Y_j\in\mathcal{I}_i^\prime\}\bone\{F_n(Y_j) < 1\}\mid Y_j, \mathcal{F}_\bbx]}{F(Y_j)(1 - F(Y_j))}\right],
    \end{align*}
    where $\mathcal{I}_i^\prime = [\min\{Y_i, Y_i^\prime\}, \max\{Y_i, Y_i^\prime\}]$ such that $Y_i$ and $Y_i^\prime$ are i.i.d. given $\bbx_i$. Note that 
    \begin{align*}
        \ee[\bone\{Y_j\in\mathcal{I}_i^\prime\}\mid Y_j, \mathcal{F}_\bbx] =& 1 - F_{\bbx_i}^2(Y_j) - (1 - F_{\bbx_i}(Y_j))^2, \\
        \ee[\bone\{Y_j\in\mathcal{I}_i^j\}\mid Y_j, \mathcal{F}_\bbx] =& 1 - F_{\bbx_i}(Y_j)F_{\bbx_{N^{-j}(i)}}(Y_j) - (1 - F_{\bbx_i}(Y_j))(1 - F_{\bbx_{N^{-j}(i)}}(Y_j)).
    \end{align*}
    Assumption (A1) yields
    \begin{align*}
        \lefteqn{\abs{F_{\bbx_i}(Y_j) - F_{\bbx_{N^{-j}(i)}}(Y_j)} \leq }\\
        &\qquad\qquad C(1 + \|\bbx_{N^{-j}(i)}\|^\beta + \|\bbx_i\|^\beta)\|\bbx_{N^{-j}(i)} - \bbx_i\|\min\{F(Y_j), 1 - F(Y_j)\}.
    \end{align*}
    By assumption (A2) there exists $K$ such that $\|\bbx_i\|, \|\bbx_{N^{-j}(i)}\|\leq K$. This gives us 
    \begin{align*}
        \abs{\ee[\big(\frac{n - n_0}{n}\big)Q_n^\prime] - \mu(\tilde{S})Q} &= \left|\ee\left[\frac{\ee[(2F_{\bbx_i}(Y_j) - 1)(F_{\bbx_{N^{-j}(i)}} - F_{\bbx_i}(Y_j))\mid \mathcal{F}_\bbx, Y_j]}{F(Y_j)(1 - F(Y_j))}\right]\right| \\
        & \leq  C K^\beta \ee[\|\bbx_i - \bbx_{N^{-j}(i)}\|]. 
    \end{align*}
    Therefore by Lemma~\ref{lmm:nn} 
    \begin{eqnarray*}
    \abs{\ee[\frac{(1 - n_0/n)}{\mu(\tilde{S})}Q_n^\prime] - Q} \leq \begin{cases}C n^{-1}(\log n)^2 & \text { if } p=1 \\ C n^{-1 / p}(\log n) & \text { if } p \geq 2\end{cases}.
\end{eqnarray*}

\begin{align*}
   \abs{\ee[Q_n^\prime] - Q} &\leq \ee\bigg[\abs{1 - \frac{1 - n_0/n}{\mu(\tilde{S})}}\big(\frac{n - n_0}{n}\big) Q_n^\prime\bigg] +  \left|\ee\bigg[\frac{(1 - n_0/n)}{\mu(\tilde{S})} Q_n^\prime\bigg] - Q\right|.
\end{align*}
Note that the first term on the right-hand side is $O(n^{-1/2})$ since $\big(\frac{n - n_0}{n}\big) Q_n^\prime$ is uniformly integrable and $n_0/n$ converges at the rate of $1/\sqrt{n}$ to $\mu(\tilde{S})$. Following the proof of Theorem~\ref{thm:consistency}, with the choice of $\Delta = c_1^{-1}\log(n)/(n-1)$ in \eqref{eqn:Delta} we have
\begin{eqnarray*}
    \ee[\abs{Q_n - Q_n^\prime}] \leq C\sqrt{\frac{\log n}{n}}.
\end{eqnarray*}
Finally, using Lemma~\ref{lmm2} and noting that $\nu_n = 1 - Q_n$ and $\nu = 1 - Q$ finishes the proof.
\end{proof}
Lemma~\ref{lmm:rate} implies
\begin{eqnarray*}
    \abs{\nu_n - \nu} = \frac{(\log n)^{1 + \bone\{p = 1\}}}{n^{1/(p\vee 2)}},
\end{eqnarray*}
which gives the proof of Theorem~\ref{thm:rate}.

\subsection{Proof of Proposition~\ref{prop:reg}}
\begin{proof}
We have the conditional CDF of $Y$ as
\[F_{Y \mid \bbx = \bx}(t) = \int_{-\infty}^t f_{Y \mid \bbx = \bx}(u) d u.
\]
For fixed $t$ and $\bx, \bx^{\prime}$, integrate along the line segment $\bx_\theta:=\bx + \theta(\bx^{\prime}-\bx), \theta \in[0,1]$ :
\begin{align*}
F_{Y \mid \bbx = \bx}(t)-F_{Y \mid \bbx = \bx}(t)=\int_0^1 \frac{d}{d \theta} F_{Y \mid \bbx = \bx_\theta}(t) d \theta =\int_0^1 \nabla_\bx F_{Y \mid \bbx = \bx_\theta}(t) \cdot(\bx^{\prime}-\bx) d \theta,
\end{align*}
so
\begin{align*}
|F_{Y \mid \bbx = \bx^{\prime}}(t)-F_{Y \mid \bbx = \bx}(t)|\leq\|\bx^{\prime}-\bx\|\sup_{\theta \in[0,1]}\|\nabla_\bx F_{Y \mid \bbx = \bx_\theta}(t)\|.
\end{align*}
Then using~\eqref{eq:cond2} we have
\begin{align*}
\nabla_\bx F_{Y \mid \bbx = \bx}(t)=\int_{-\infty}^t \nabla_\bx f_{Y \mid \bbx = \bx}(u) d u.
\end{align*}
Then
\begin{align*}
\|\nabla_\bx F_{Y \mid \bbx = \bx}(t \mid \bx)\| \leq K_1(1+\|\bx\|^\beta) \int_{-\infty}^t f(u) d u=K_1(1+\|\bx\|^\beta) F(t)
\end{align*}
Similarly, integrating from $t$ to $\infty$ gives the same bound with $1-F(t)$. This gives us 
\[
\|\nabla_\bx F_{Y \mid \bbx = \bx}(t)\| \leq K_1(1+\|\bx\|^\beta) \min \{F(t), 1-F(t)\} .
\]
Along the line segment between $\bx$ and $\bx^{\prime},\|\bx_\theta\|$ is bounded by $\|\bx\|+\|\bx^{\prime}\|$, so
\begin{align*}
\sup_\theta(1+\|\bx_\theta\|^\beta) \leq c_\beta(1+\|\bx\|^\beta+\|\bx^{\prime}\|^\beta),
\end{align*}
for some constant $c_\beta$. Putting this together,
\begin{align*}
|F_{Y \mid \bbx = \bx^{\prime}}(t)-F_{Y \mid \bbx = \bx}(t)| \leq C(1+\|\bx\|^\beta+\|\bx^{\prime}\|^\beta)\|\bx-\bx^{\prime}\| \min \{F(t), 1-F(t)\},
\end{align*}
with $C = c_\beta K_1$. 
\end{proof}

\subsection{Proof of Theorem~\ref{thm:FOCI}}
Let $j_1, j_2, \ldots, j_p$ be the complete ordering of all variables by FORD. Let $V_0 = \emptyset$, and for each $1\leq k \leq p$, let $V_k := \{j_1, \ldots, j_k\}$. For $k > p$, let $V_k := V_p$. Note that for each $k$, $j_k$ is the index $j\not\in V_{k - 1}$ that maximizes $\nu_n(Y, \bbx_{V_{k-1}}\cup\{j\})$. Let $K = \lfloor 4/\delta + 2\rfloor$. Let $E^\prime$ be the event that $\abs{\nu_n(Y, \bbx_{V_k}) - \nu(Y, \bbx_{V_k})} \leq \delta/8$ for all $1 \leq k \leq K$, and let $E$ be the event that $V_K$ is sufficient.

\begin{lmm}\label{lmm:Eprime}
    Suppose that $E^\prime$ has happened, and for some $1\leq k\leq K$
    \begin{eqnarray}\label{eq:Eprime}
        \nu_n(Y, \bbx_{V_k}) - \nu_n(Y, \bbx_{V_{k - 1}}) \leq \delta/2.
    \end{eqnarray}
    Then $V_{k - 1}$ is sufficient.
\end{lmm}
\begin{proof}
    Take any $k \leq K$ such that~\ref{eq:Eprime} holds. If $k > p$ there is nothing to prove. So let us assume that $k\leq p$. Since $E^\prime$ has happened, this implies that for any $j\not\in V_{k - 1}$,
    \begin{align*}
        \nu(Y, \bbx_{V_{k - 1}\cup\{j\}}) - \nu(Y, \bbx_{V_{k-1}}) \leq \nu_n(Y, \bbx_{V_k}) - \nu_n(Y, \bbx_{V_{k-1}}) + \frac{\delta}{4} \leq \frac{3\delta}{4}.
    \end{align*}
    Then note that by definition of $\delta$, $V_{k - 1}$ must be a sufficient set.
\end{proof}
\begin{lmm}\label{lmm:EprimeGivesE}
    The event $E^\prime$ implies $E$.
\end{lmm}
\begin{proof}
    Suppose $E^\prime$ has happened but there is no $k$ such that~\ref{eq:Eprime} is valid. Therefore for all $1\leq k\leq K$ we have 
    \begin{align*}
        \nu_n(Y, \bbx_{V_k}) - \nu_n(Y, \bbx_{V_{k - 1}}) > \delta/2.
    \end{align*}
    This implies that 
    \begin{align*}
        \nu(Y,\bbx_{V_k}) - \nu(Y,\bbx_{V_{k - 1}}) \geq \nu_n(Y, \bbx_{V_k}) - \nu_n(Y, \bbx_{V_{k - 1}}) - \frac{\delta}{4} \geq \frac{\delta}{4}.
    \end{align*}
    This gives
    \begin{align*}
        \nu(Y, \bbx_{V_K}) =& \sum_{k = 1}^K \nu(Y, \bbx_{V_k}) - \nu(Y, \bbx_{V_{k - 1}}) \\
        \geq & \frac{K\delta}{4} \geq \left(\frac{4}{\delta} + 2\right)\frac{\delta}{4} > 1.
    \end{align*}
    Note that this contradicts the fact that $\nu(Y, \bbx_{V_k})\in[0, 1]$. Therefore, this shows that~\ref{eq:Eprime} must hold for some $k \leq K$. Therefore, Lemma~\ref{lmm:Eprime} implies that $V_K$ is sufficient. 
\end{proof}

\begin{lmm}\label{lmm:rateEprime}
There are positive constants $L_1, L_2$ and $L_3$ depending only on $C, \beta, K$ and $\delta$ such that 
    \begin{align*}
        \pp(E^\prime) \geq 1 - L_1p^{L_2}\exp(-L_3n/\log n).
    \end{align*}
\end{lmm}
\begin{proof}
    By assumptions (A1$^\prime$) and (A2$^\prime$), and Lemma~\ref{lmm:rate}, there exists $L_1, L_2$ and $L_3$ such that for any $V$ of size at most $K$ and any $t\geq 0$,
    \begin{eqnarray*}
        \pp(\abs{\nu_n(Y, \bbx_V) - \nu(Y, \bbx_V)}\geq L_1n^{-1/K\vee 2}(\log n)^2 + t) \leq L_2\exp(-L_3 n t^2/\log n).
    \end{eqnarray*}
    Let the event on the left be $A_{V,t}$ and $A_t := \bigcup_{\abs{V}\leq K}A_{V,t}$. By union bound we have $\pp(A_t)\leq L_2 p^K\exp(-L_3nt^2/\log n)$. Choose $t = \delta/16$. If $n$ is large enough so that 
    \begin{align}\label{eq:largen}
        L_1n^{-1/K\vee 2}(\log n)^2 \leq \frac{\delta}{16},
    \end{align}
    then the above bound implies that 
    \begin{align}\label{eq:largen2}
        \pp(E^\prime) \geq 1 - L_2p^K\exp(-L_4 n/\log n).
    \end{align}
    Equivalently, one can write ~\ref{eq:largen} as $n \geq L_5$ for some large $L_5$. Then we choose $L_6\geq L_2$ such that for any $n < L_5$,
    \begin{eqnarray*}
        L_6p^K\exp(-L_3 n/\log n) \geq 1. 
    \end{eqnarray*}
    Therefore for $n < L_5$, we have the trivial bound $\pp(E^\prime)\geq 1 - L_6p^K\exp(-L_3 n)$. Combining this with~\ref{eq:largen2} finishes the proof.
\end{proof}
\begin{lmm}\label{lmm:EprimeSufficient}
Event $E^\prime$ implies that $\hat{V}$ is sufficient.
\end{lmm}
\begin{proof}
    Suppose that $E^\prime$ has happened. First, suppose that FORD has stopped at step $K$ or later. Then $V_K\subseteq \hat{V}$ and, therefore, Lemma~\ref{lmm:EprimeGivesE} implies that $E$ has also happened, and therefore $\hat{V}$ is sufficient. Next, suppose that FORD has stopped at step $k - 1 < K$. Then, by definition of the stopping rule, we have 
    \begin{align*}
        \nu_n(Y, \bbx_{V_k}) \leq \nu_n(Y, \bbx_{V_{k - 1}}),
    \end{align*}
    which implies~\ref{eq:Eprime}. Since $E^\prime$ has happened, Lemma~\ref{lmm:Eprime} implies that $\hat{V} = V_{k - 1}$ is sufficient.
\end{proof}
Theorem~\ref{thm:FOCI} is an immediate result of Lemma~\ref{lmm:EprimeSufficient} and \ref{lmm:rateEprime}.

\subsection{Proof of Theorem~\ref{thm:deltaAndPrime}}
\begin{proof}
    Note that in our Gaussian linear model we have
    \begin{align}
        Y=\beta\bbx + \varepsilon,\qquad \varepsilon\perp\bbx, \qquad \varepsilon\sim N(0, \sigma^2).
    \end{align}
    Therefore we have 
    \[
    Y\mid\bbx_S \sim N(Z_S, \sigma_S^2), \qquad Z_S = \ee[Y\mid\bbx_S], \qquad \sigma_S^2 = \var(Y\mid\bbx_S),
    \]
    where $Z_S$ is linear in $\bbx_S$ and $\sigma_S^2$ is a constant that does not depend on $\bbx_S$. Note that 
    \begin{align*}
        \rho(\emptyset, S) = R^2_S = R^2(Y;\bbx_S) = \frac{\var(Z_S)}{\var(Y)}\in[0, 1),
    \end{align*}
    and $(Y, Z_S)$ are jointly Gaussian with mean zero and
    \[
    \var(Y) = \tau^2, \qquad \var(Z_S) = \tau^2R^2_S,\qquad \cov(Y, Z_S) = \var(Z_S),
    \]
    and
    \begin{align*}
        \sigma_S^2 = \var(Y\mid\bbx_S) = \var(Y) - \var(Z_S) = \tau^2(1 - R^2_S).
    \end{align*}
    Let $R_*^2 = R^2(Y; \bbx)\in (0, 1)$. Also 
    \begin{align}\label{eq:smoothCond}
        \ee[\bone\{Y > t\}\mid\bbx_S] = \ee[\bone\{Y > t\}\mid Z_S] = \Phi(\frac{Z_S - t}{\sigma_S}) = \Phi\left(\frac{Z_S - t}{\tau(1  -R^2_S)^{1/2}}\right).
    \end{align}
    Thus the variance term in the numerator of the integrand in $\nu$ depends on $\bbx_S$ only via the one-dimensional Gaussian random variable $Z_S$, therefore $\nu(Y, \bbx_S) = \nu(Y, Z_S)$.

    Note that by~\eqref{eq:smoothCond}, $\nu(Y, \bbx_S)$ is a smooth function of $\tau^2$ and $R_S^2$, i.e. there exists $\psi:[0, R_*^2]\rightarrow[0, 1]$ such that $\nu(Y, \bbx_S) = \psi(R_S^2)$. In other words, in this Gaussian linear setting, $\nu$ is just a scalar function of the usual $R^2$ such that (i) $\psi(0) = 0$, (ii) $\psi$ is strictly increasing and smooth on $[0, R^2_*]$. Thus $\psi^\prime$ is continuous and strictly positive on $[0, R^2_*]$. Define
    \begin{align}\label{eq:unifBound}
        m := \min_{r\in[0, R^2_*]} \psi^\prime(r) > 0, \qquad M := \max_{r\in[0, R^2_*]} \psi^\prime(r) < \infty.
    \end{align}
    For any insufficient $S$ and $j\notin S$, note that
    \begin{align*}
        \rho(S, j) = \frac{R^2_{S\cup\{j\}} - R^2_S}{1 - R^2_S},
    \end{align*}
    and therefore $R^2_{S\cup\{j\}} =  R^2_S + \rho(S, j)(1 - R^2_S)$. We have 
    \begin{align*}
        \Delta\nu_{S,j} := \nu(Y, \bbx_{S\cup\{j\}}) - \nu(Y, \bbx_S) = \psi(R^2_{S\cup\{j\}}) - \psi(R^2_{S}).
    \end{align*}
    By the mean value theorem, there exists $\xi_{S, j}\in[R^2_S, R^2_{S\cup\{j\}}]$ such that 
    \begin{align*}
        \Delta\nu_{S,j} = \psi^\prime(\xi_{S, j})(R^2_{S\cup\{j\}} - R^2_S).
    \end{align*}
    Using the uniform bounds~\eqref{eq:unifBound} on $\psi^\prime$ and the fact that $R^2_S \leq R^2_*$ for all $S$ we have 
    \begin{align*}
         m\rho(S, j)(1 - R^2_*) \leq \Delta\nu_{S,j} \leq M\rho(S, j)(1 - R^2_*).
    \end{align*}
    Let $c:= m(1 - R^2_*)$ and $C:= M(1 - R^2_*)$. Recall from definition of $\delta$ and $\delta^\prime$ that
    \begin{align*}
        \delta^\prime = \inf_{S\text{ is insufficient}}\max_{j\notin S} \rho(S, j), \qquad \delta = \inf_{S\text{ is insufficient}}\max_{j\notin S} \Delta\nu_{S,j}.
    \end{align*}
    Then for each insufficient $S$,
    \begin{align*}
         c\max_{j\notin S}\rho(S, j) \leq \max_{j\notin S}\Delta\nu_{S,j} \leq C\max_{j\notin S}\rho(S, j).
    \end{align*}
    Then taking infimum over all insufficient $S$ we have
    \begin{align*}
         c\delta^\prime \leq \delta \leq C\delta^\prime.
    \end{align*}
    Note that $m$ and $M$ only depend on $\psi$ which depends on $\tau$. Additionally $R^2_*$ depends only on $\sigma$ and $\tau$. Therefore $c$ and $C$ are constants that only depend on $\tau$ and $\sigma$.
\end{proof}

\bibliographystyle{plain}
\bibliography{ref}
\end{document}